\documentclass{article}
\usepackage [a4paper, total={210mm,397mm}, left=10mm, right=5mm, top=5mm, bottom=20mm]{geometry}
\usepackage[utf8]{inputenc}
\usepackage[russian]{babel}
\usepackage{tikz}
\usepackage{amssymb,amsfonts,amsmath,amsthm, amsbsy}
\usepackage{hyperref}
\hypersetup{
    colorlinks=true,
    linkcolor=blue,
    filecolor=magenta,      
    urlcolor=cyan,
    pdftitle={Overleaf Example},
    pdfpagemode=FullScreen,
}
\usepackage{ulem}
\usepackage{tabularx}

\newtheorem {note}	{Замечание}
\newtheorem {stat}	{Утверждение}

\newtheorem {defn}	{Определение}
\newtheorem {sgn}	{Обозначение}

\newtheorem {lem}	{Лемма}

\newtheorem {prop}	{Свойство}
\newtheorem {cor}	{Следствие}
\newtheorem {quest}	{Вопрос}
\newtheorem {con}	{Гипотеза}

\newcommand{\mi}{\raisebox{0.5pt}{\scalebox{1.5}[0.75]{{\pmb-}}}}
\newcommand{\pl}{\raisebox{0.5pt}{\scalebox{0.75}[0.75]{{\pmb+}}}}
\newcommand{\ze}{\raisebox{0.5pt}{\scalebox{1.0}[0.75]{{\pmb0}}}}

\newcommand{\lcev}[1]{\reflectbox{\ensuremath{\vv{\reflectbox{\ensuremath{#1}}}}}}

\newcommand*{\vv}[1]{\vec{\mkern0mu#1}}

\newcommand{\rc}{\vec{c}}
\newcommand{\lc}{\lcev{c}}
\newcommand{\nc}{\overline{c}}
\newcommand{\rb}{\vec{b}}		
\newcommand{\lb}{\lcev{b}}		
\newcommand{\nb}{\overline{b}}					
\newcommand{\rbc}{\vec{bc}}	
\newcommand{\lbc}{\lcev{bc}}	
\newcommand{\nbc}{\overline{bc}}				

\def\a     {{\bf a}}
\def\b     {{\bf b}}
\def\c     {{\bf c}}
\def\v     {{\bf v}}
\def\x     {{\bf x}}
\def\y     {{\bf y}}
\def\z     {{\bf z}}
\def\F     {{\bf F}}

\def\A     {{\bf A}}
\def\C     {{\bf C}}
\def\D     {{\bf D}}
\def\E     {{\bf E}}
\def\I     {{\bf I}}
\def\P     {{\bf P}}

\def\S     {{\bf S}}
\def\L     {{\bf L}}
\def\N     {{\bf N}}
\def\R     {{\bf R}}
\def\T     {{\bf T}}
\def\U     {{\bf U}}
\def\V     {{\bf V}}
\def\W     {{\bf W}}
\def\p     {{\bf p}}
\def\l     {{\bf l}}
\def\r     {{\bf r}}
\def\u     {{\bf u}}
\def\v     {{\bf v}}
\def\w     {{\bf w}}
\def\pref  {{\bf pref}}
\def\suff  {{\bf suff}}
\def\PREF  {{\bf PREF}}
\def\SUFF  {{\bf SUFF}}
\def\subw  {{\bf subw}}
\def\per   {{\sf per}}
\def\ex    {{\sf ex}}
\def\exp   {{\sf exp}}
\def\lexp  {{\sf lexp}}
\def\RT    {{\sf RT}}
\def\eps   {{\varepsilon}}

\def\mN     {{\mathbb N}}

\def\mR     {{\mathbb R}}

\def\mV     {{\mathbb V}}
\def\mZ     {{\mathbb Z}}

\def\contr  {{$\triangleright\triangleleft$}}
\def\ss     {{\subset}}
\def\sse    {{\subseteq}}

\def\letterText{\textbf}

\def\tx{\tilde{\letterText{x}}}
\def\ty{\tilde{\letterText{y}}}
\def\tz{\tilde{\letterText{z}}}
\def\sm{\setminus}
\def\dt    {{\delta}}

\def\cRe    {\color{red}}
\def\cGr    {\color{gray}}

\def\Anz	{$A^0_n$}
\def\Anm	{$A^-_n$}

\makeatletter
\let\orgdescriptionlabel\descriptionlabel
\renewcommand*{\descriptionlabel}[1]{%
	\let\orglabel\label
	\let\label\@gobble
	\phantomsection
	\edef\@currentlabel{#1}%
	\let\label\orglabel
	\orgdescriptionlabel{#1}%
}
\makeatother

\title{
	Кольцевые Граничные Слова и Другие\\Усиления Граничной Теоремы
}
\author{Тунёв Игорь Николаевич
	\footnote{Уральский Федеральный Университет}
}
\date{itnvi@mail.ru}

\begin{document}
	\maketitle
	
	\begin{abstract}
		Данный манускрипт основан на студенческих работах автора
		\footnote{Дипломы бакалавра 2011 и магистра 2013 годов}.
		В работе 2011 года автором предлагается новый метод доказательства частных случаев
		теоремы Дежан (окончательно доказанной в 2009) с использованием компьютера за полиномиальное время (от размера алфавита).
		Там же предлагается общая конструкция для первых нечётных случаев (число букв алфавита) начиная с пяти букв,
		позволяющая доказать теорему Дежан для этих случаев (с использованием компьютера).
		Предложенные граничные слова (далее ГС) так же являются циркулярными/кольцевыми (т.е. любой циклический сдвиг является ГС).
		
		В работе 2013 года улучшается метод доказательства (предложенный автором в 2011) с сокращением необходимых условий.
		Предлагается общая конструкция для первых чётных случаев (с чилом букв не менее 6).
		А так же предлагается метод построения ГС на экспоненциально растущем дереве ГС.
		Точнее, существуют ГС у которых все достаточно длинные факторы имеют экспоненту сколь угодно близкую к 1.
		
		Здесь мы представим отредактированную версию этих методов с некоторыми улучшениями.
	\end{abstract}
	
	Ключевые слова: гипотеза Дежан, граничная теорема, циркулярные граничные слова, алгоритм, перестановки

\setcounter{section}{-1}

\section{Введение}

\subsection{История}
{
Букву в позиции $i$ слова $w$ будем обозначать $w(i)$.
Длину слова $w$ обозначим $|w|$.
Периодом слова $w$ называется наименьшее число $p$,
для которого $w(i) = w(p+i)$, для всех $i=1,...,|w|-p$, и будем обозначать $per(w)$.

Граничное слово (ГС) над алфавитом из $n\ge5$ букв --- слово,
у которого для любого его фактора (подслова) $v$ выполняется неравенство $\frac{|v|}{\per(v)}\le\frac{n}{n-1}$.
Множество ГС над $n$ буквами обозначается $\T_n$.

Гипотеза о бесконечности множества ГС над любым алфавитом с $n\ge5$ буквами была предложена в 1972 году Дежан~\cite{Dej}.
Доказательство последних частных случаев было получено в 2009 г,
через 37 лет после формулировки этой гипотезы (теперь теоремы).
Доказательство потребовало много времени машинных вычислений для многих частных случаев различными методами.
Эти результаты для $k\ge5$ опубликованы в работах математиков:
Мулен-Оланье~\cite{M}, Мохаммада-Нури и Карри~\cite{MC}, Карри и Рамперсада~\cite{CR1}, \cite{CR2}, Рао~\cite{Rao}, Карпи~\cite{C}.

Но появляются новые вопросы, например,
о числе ГС для многих алфавитов.
Или о существовании циркулярных /кольцевых/циклических ГС (ЦГС) --- циклический сдвиг которых так же является ГС.
В данной работе предлагается авторская методика построения ГС со свойством цикличности т.е. ЦГС,
даже с усиленными ограничениями на экспоненту.
А так же построение экспоненциального множества бесконечных ГС (БГС).
Так же, предлагается доказательство существования ГС, у которых все длинные факторы имеют <<почти>> единичную экспоненту,
только на основе факта существования равномерно экспоненциально растущего дерева ГС.

}

\newpage

\hypertarget{contents}
\tableofcontents

\newpage

\subsection{Подробнее о работе}


В данной работе описываются основные методы решения нескольких вопросов о свойствах граничных слов.
Методы были разработаны в дипломных работа автора в 2011 и 2013 года.
В 2011 году автором были разработаны новые методы построения множества граничных слов c экспоненциальным ростом 
для частных случаев (с несколько небрежными формулировками и черновыми доказательствами ключевых утверждений).
Используя эти методы, построена предположительно общая конструкция для алфавитов из нечётного числа букв не меньше 5,
и доказательство её корректности сведена к компьютерной проверке за полином.
Эти методы были обсуждены на семинаре и описаны (в сыром виде)
в нашей первой дипломной работе (ДР1) в июне 2011 года.
\footnote{ДР1, текст лекций для семинара и дальнейшие обобщения метода проверки
	выложены в 1-й версии (т.е. в этой) данной серии версий.}

В первой части данной работы мы разберём основные <<результаты>> нашей 1-й дипломной работы (ДР1).
\footnote{Хотя, в строгом смысле результатами они, возможно, не являются, но будем уловно называть их так.}
Точнее, разберём 3 основные леммы.
Формализуем доказательство экспоненциальности в более строгом смысле
на основе схемы, предложенной в ДР1,
что позволит полностью доказать главный результат.
Для доказательства главного результата (более строгой --- теоремы  о равномерно растущем дерве ГС (РРДГС)),
формально потребуется экспоненциально растущее дерево с линейным показателем в экспоненте (равномерный эксп.рост бесконечного слова),
что будет доказанно в 1-й части.

Во 2-й части опишем основные результаты из ДР2.
Точнее, улучшение схемы построения РРДГС с сокращением достаточного числа слов (с $3n+1$ как в ДР1 до $2n$).
А так же приведём доказательство из ДР2 существования ГС (для частных случаев),
у которых <<все достаточно длинные факторы имеют экспоненту сколь угодно близкую к 1>>
\footnote{На сколько помнит автор, такую формулировку предложил наш НР прямо во время защиты автором ДР2},
только на основе факта существования экспоненциального растущего множества ГС с линейным показателем.
Точнее, экспоненциальный рост должен быть представим как дерево,
с удваивающимся (или более) числом веток не более чем через некоторое константное число следующих букв
(т.е. длина веток в дереве без <<бифукаций>> (или более) ограничена константой сверху, где длина ветки это число букв).

В 3-й части опишем конструкцию для нечётных из ДР1 и конструкцию для чётных из 2-й дипломной работы (ДР2).
Опишем общий подход сведения проверки частных случаев этих конструкций к полиномиальной с доказательством корректности.

В 4-й части разберём эффективный алгоритм проверки слова $w$ на граничность, работающий за $O(|w|\log{|\w|})$


В общем, ДР1 можно считать черновым вариантом основных результатов,
где всё сведено к несложным утверждениям, но, возможно, рутинным.
Оригиналы студенческих работ автора, разбираемые здесь, выложены (если удалось) в исходниках данной работы.

%

\hyperlink{contents}{$\upuparrows$}

\paragraph{Неформальное дополнение.}

Чтобы не быть голословным о датах разбираемых здесь работ (вспомогательный текст для 2-х лекций автора на семинаре, ДР1 и ДР2),
автор выложил в общий доступ к письмам из переписок (с нашим нуачным руководителем (НР) Шуром А.М. на то время) с этими работами \ref{links}.
В исходниках данной работы выложены оригинальные файлы, скачанные из этих писем (кроме папки с файлами *.docx):

{\bf(1)} Dejean.docx --- вспомогательный текст для 2-х лекций семинара.

{\bf(2)} Папка с ДР1 (скачано из переписки 24.06.11 с др. человеком, файлы в формате *.docx для удобства чтения).

{\bf(3)} 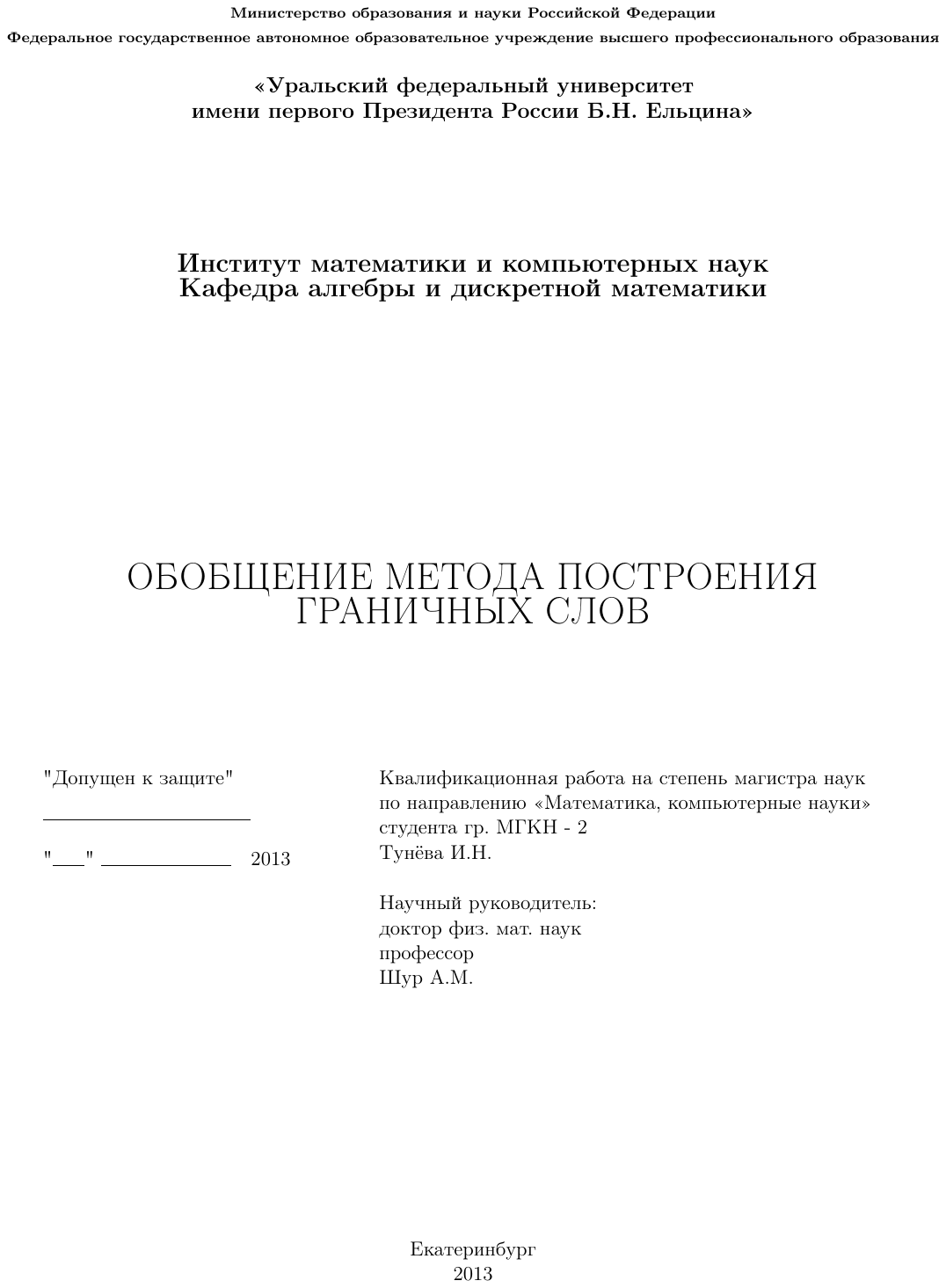 --- ДР2 (скачано из представленной переписки и переименовано).

Файлы *.docx можно смотреть через сервис в mail.ru:
отправляем файл себе на почту в mail.ru и открываем его там же.

В 1-й части работы мы подробно разберём главные <<результаты>> из ДР1.
Мы более полно формализуем основные конструкции и идеи с доказательством их корректности,
предложенных в ДР1.
Предложим некоторые вариации их усиления и обобщения.
Хоть в авторстве статьи и обратный порядок фамилий, но это не объясняет вклад каждого автора.
Поэтому, здесь несколько прояснится ответ и на этот вопрос.

В исходниках выложены файлы (если удалось) с оригинальным текстом с леммами,
которые были рассмотрены на семинаре нашего научного руководителя (Шура А.М.. Далее НР,
хоть формально он уже не научный руководитель автора, условно будем называть его так) в конце марта 2011г.
\footnote{Файл взят из переписки с НР, отправленная ему для ознакомления 07.03.2011 перед 1-й лекцией.
}
На первой лекции автором была разобрана лемма о 3-хзначнй подстановке
\footnote{На лекциях и в ДР1 автор некорректно называл эту подстановку морфизмом}.
На 2-й лекции была разобрана лемма про циклические граничные слова (ЦГС).
Про ЦГС автор достаточно полно объяснил (основные идеи и) ключевые утверждения в доказательстве леммы, 
кроме, возможно, некоторого утверждения в конце текста доказательства о ЦГС (из лекций).

Доказательства основных лемм довольно небрежное
(с незначительным, по мнению автора, но легко поправимыми ошибками и недочётами),
но в них можно вычленить ключевые утверждения и условия,
на основе которых не сложно доработать доказательство до презентабельного вида.

В формальной части ДР1 доказательства основных 3-х лемм опускаются достаточно простые случаи
(часто невозможные, которые тогда автор считал их не сложными, но рутинными).
\footnote{Они будут описываться в сносках внизу.}
В доказательстве были проверены все части, требующие аккуратности
(когда все неравенства сходятся точно на границе допустимых значений). 
Но есть и пробелы,
когда не описаны даже некоторые обозначения о которых можно понять только по контексту.
И даже отсутствует 1 пояснение о свойстве (хоть и несложно доказываемое), неявно используемом в дальнейшем.
Поэтому, можно рассматривать как черновое доказательство с разбором всех нетривиальных случаев
\footnote{Под <<тривиальным>> автор подразумевает техническое/рутинное или идейно не сложное.
Под силу специалисту с достаточным опытом.}
с небрежной формализацией ключевых этапов доказательства (сами этапы доказываются почти очевидно,
но автор попытался формализовать строго и наспех,
поэтому много недочётов и, даже, арифметических ошибок, но легко поправимых).
В любом случае, (в одиночку) разбор предложенного метода займёт не более 37 лет у любого терпеливого новичка в комбинаторике слов.


Автор считает, что людям интеллектуальной профессии достаточно объяснить
решение более общими приёмами без подробностей когда б$\acute{\text{о}}$льшая детализация уже тривиальна
(даже если в обяснении есть ошибки, но легко поправимые).
Поэтому, по некоторым критериям можно считать, что решение уже было представлено в ДР1.
Учитывая, что данная задача --- не тривиальный результат, а задача с прохождением <<зоны смерти>>,
то доработку, темболее, можно считать тривиальной (т.е. решаема опытным специалистом).
И, в общем, по мнению автора, теорию лучше воспринимать как творческий
путь развития решения с различными его (решения) вариациями.
Как минимум этот этап, когда доработки/поправки являются уже тривиальными, должен быть выделен в решениях.
А вот промышленное применение теории уже должно быть отлажено точно во избежание проблем на практике.

При этом, формат доказательства почти на чистом формальном языке (+симпл русский) ---
так автор, самоуверенно, считал такой язык наиболее подходящим для научных результатов
(т.к. формальный язык почти однозначно интерпретируется).
При этом, такой формат ещё и легче переводить на другие языки.

Не каждый разберётся в предложеном автором доказательстве основных лемм.
Но достаточно терпеливый специалист способен разобраться с доказательством,
а достаточно опытный и доказать корректность методов своим способом.
Научные работы, всётаки, пишутся для думающих людей.
В крайнем случае можно было попросить переоформить доказательства ключевых лемм
(с достаточным временем вне учебного времени --- работа достаточно большая и нетривиальная).

В доказательстве, как оказалось, так же присутствуют смысловые ошибки (опечатки) и избыточные утверждения.
\footnote{Но все они поправимы. О них так же будет в сносках.}
А по прошествии более 10 лет почти все детатали доказательства выветрились из памяти автора.

Будучи студентом, слабым в вопросах организации, была стойкая уверенность,
что университетские научные работы уровня диплома и выше,
являются более фундаментальными для определения авторства работы (первоисточником для науки),
чем газетно-журнальные издательства, которые могут и закрытся
(Ведь не редакции называют <<храмом науки/знаний>>).
И, вообще, трудно было даже сломать догму о том,
что университетские научные работы (дипломы, диссертации) это пыль для науки
по сравнению со статьями (НАУЧНЫМИ).
И уж, тем более, дико было предположить, что универские работы уровня дипломов стирают из истории.




\hyperlink{contents}{$\upuparrows$}

\subsection{Общие определения и обозначения по умолчанию}

Во всей работе по умолчанию $n,k\in\mN, \eps\in\mR$, если не оговорено другое.
Обозначение $\mN_i$ означать множество всех целых чисел не менее $i\in\mZ$.
$\A_n=\{a_1,...,a_n\}$ --- алфавит (кортеж), над которым изучаются слова в данной работе.
$\A_n^k$ --- все слова над $\A_n$ длины $k$.
По умолчанию $n\ge5$ и используется для обозначения числа букв в алфавите.

Термины, введёные здесь, являются стандартными, но мы определим их, возможно с некоторыми изменениями.

\begin{defn}
	Длиной слова $v$ называется количество букв в нём и обозначается как $|v|$.
\end{defn}

\begin{defn}
	Степень $k$ слова $v$ называется $k$ последовательных копий $v$
	%
	и обозначается как $v^k=vv...v, |v^k|=k\cdot|v|$.
\end{defn}

\begin{defn}
	Фактором $u$ слова $v$ является подпоследовательность подряд идущих букв в слове $v$,
	сохраняющая порядок букв и обозначается как $u\subseteq v$.
	Множество всех факторов слова $w$ обозначим $\F(w)$.
\end{defn}

\begin{defn}
	Префиксом слова $v$ является фактор слова $v$, содержащий первую букву слова $v$
	и обозначается $\pref(v)$.
\end{defn}

\begin{defn}
	{Корнем слова $v$ назовём кратчайший его фактор $x$ (для определённости --- префикс) такой,
	что $v$ является фактором слова $x^{|v|}$.}
\end{defn}

\begin{defn}
	Периодом слова назовём длину его корня и обозначим $|v|-\per(v)$.
\end{defn}

\begin{defn}
	Повтором слова $v$ называется его префикс длины $|v|-\per(v)$
	(он же и суффикс слова $v$).
\end{defn}

\begin{defn}
	Слово $v$ называется простым, если $|v|=\per(v)$, т.е. слово с пустым повтором.
\end{defn}

\begin{defn}
	Экспонентой $\exp(v)$ слова $v$ называется отношение его длины к длине корня
	т.е. $\exp(v)=\frac{|v|}{\per(v)}$.
\end{defn}

\begin{defn}
	Локальной экспонентой $\lexp(v)$ слова $v$ называется верхняя граница экспонент всех его непустых факторов
	т.е. $\lexp(v)=\sup\{\exp(u):u\subseteq v\}$.
\end{defn}

\begin{defn}
	Слово $v$ избегает экспоненту $\beta$ ($\beta^+$), если $\lexp(v)<\beta$ ($\lexp(v)\le\beta$).
\end{defn}

\begin{defn}
	Экспонента $\beta$ ($\beta^+$) избегаема в алфавите с $k$ буквами,
	если существует бесконечно много слов над этим алфавитом, избегающих $\beta$ ($\beta^+$).
\end{defn}

\begin{defn}
	Слова избегающие экспоненту $\beta$ ($\beta^+$) обзначаются $\beta${--}free ($\beta^+${--}free).
\end{defn}

\begin{defn}
	Границей повторяемости называется функция $\RT(n)$ от $n\in\mN$
	$$
	\RT(n) = \inf\{\beta > 1: \text{ существует бесконечно много } \beta^+\text{{--}free слов над } \A_n\}
	\bigg(=\frac{n}{n-1}\text{ при }n\ge5\bigg)
	$$
\end{defn}

\begin{defn}
	Слово $w$ над $\A_n$ называется граничным, если $\lexp(w)\le\RT(n)$.
	Множество всех таких граничных слов обозначим $\T_n$.
\end{defn}

\begin{defn}
	$|u|_a$ --- количество букв $a$ в слове $u$.
\end{defn}

\begin{sgn}
	Символом \contr\ будем обозначать противоречие.
	Т.е. его можно заменить фразой <<что противоречит>> или <<что противоречит тому, что>>.
\end{sgn}

\hyperlink{contents}{$\upuparrows$}

\subsection{Критерий граничной теоремы (ГТ) о бесконечном граничном слове (БГС)}

Здесь мы докажем, что существует БГС,
из того, что существует бесконечное число конечных граничных слов (ГС).
Даже не только для ГС, но и для всех факториальных языков.

Граничная теорема утверждает, что $|\T_n|=\infty$

\begin{defn}
	{\bf Бесконечное ГС} (БГС) вправо (по умолчанию все БГС в работе считаем БГС вправо) --- бесконечное вправо слово,
	все факторы которого являются ГС.
	Критерий БГС вправо --- любой префикс являеется ГС.
\end{defn}

\begin{stat}
	$|\T_n|=\infty$ если и только, если существует БГС. 
\end{stat}
\begin{proof}
	Очевидно, что если $\omega\in\T_n$, то любой её префикс из $\T_n$ (в силу факториальности $\T_n$).
	
	Тогда, если $|\omega|=\infty$, то множество всех перфиксов слова $\omega$ бесконечно(счётно),
	и является подмножеством в $\T_n$.
	Значит из $|\omega|=\infty$ следует что $|\T_n|=\infty$.
	
	Теперь докажем в другую сторону.
	Достаточно доказать,
	что существует (бесконечная) последовательность конечных ГС $u_1,u_2,...,u_i,...$, 
	где $|u_i|<u_{i+1}$ и $u_i=\pref_{|u_i|}(u_{i+1})$.
	Тогда пределом и будет бесконечное (вправо) слово.

	Разобьём $\T_n$ на подмножества по длинам слов.
	Обозначим как
	$$\U_{n, m} = \big\{u\in\T_n: |u|=m\big\}$$
	Тогда для любых $m\in\mN$ группа $\U_{n, m}$ конечна т.к. ограничена числом $n^m$.
	Тогда, т.к. $\T_n$ бесконечно, то таких непустых групп бесконечно много
	т.е. для любых $m\in\mN$ группа $\U_{n, m}$ не пуста.
	
	Заметим, что любое слово $u_2\in\U_{n, m_2}$ является удлинением некоторго слова из $u_1\U_{n, m_1}$
	(т.е. $u_1$ --- префикс слова $u_2$) при любых $m_1>m_2\in\mN$
	(в силу факториальности $\T_n$).
	Обозначим это отношение как $\U_{n, m_2}\Subset\U_{n, m_1}$.
	Если слово $u\in\U_{m_1}$ не удлинняется ни до какого слова из $\U_{m_2}$
	Тогда докажем, что существует последовательность $u_1,u_2,...,u_m,...$ такая,
	что $u_i=\pref_{|u_i|}(u_j)$ при любых $i<j\in\mN$ и $u_i\in\U_{n, i}$ при любых $i\in\mN$.
	
	Докажем ОП.
	Предположим, что такой (бесконечной) последовательности удлинняющихся префиксов не существует.
	Тогда, никакое $u\in\U_{n,1}$ не удлинняется неограниченно.
	Но т.к. $\U_{n,1}$ конечно, то существует такое $m\in\mN$,
	что никакое $u\in\U_{n,1}$ не удлинняется до $\U_m$.
	Тогда, $\U_m\not\Subset\U_1$, что противоречит факториалььности $\T_n$.
	
	Т.е. существует неограниченная последовательность удлиняющихся слов из $\T_n$,
	где каждое следующее слово содержит все предыдущие (в виде префиксов).
	Пределом этой последовательности и будет БГС.
%
\end{proof}


Если считать, что $\T_n$ содержит пределы любой сходящейся последовательности слов из $\T_n$,
то $\T_n$ содержит БГС.

\hyperlink{contents}{$\upuparrows$}

\newpage

\section{Разбор ДР1 (2011) --- построение частных случаев равномерно растущего дерева БГС (РРДГС)}

Равномерно растущим деревом граничных слов будем называть дерево слов,
где каждая (нисходящая) ветка является ГС.
При этом, расстояние между соседними вершинами дерева ограничено константами сверху и снизу.

Здесь мы предложим метод построения такого дерева для частных случаев $n\ge5$.

\subsection{Лемма 1. Метод спуска}

\subsubsection{Дополнительные определения}

Здесь мы введём дополнительные вспомогательные объекты для 1-й Леммы \ref{l1}.
Новые определения будут представлены в терминах множеств как в оригинальной работе.
Через $\mN$ обозначается множество всех натурльных чисел т.е. $\mN=\{1, 2, ...\}$.

Определим слова с более строгими свойствами чем граничные.

\begin{defn}
	Пусть $p, n\in\mN$.
	Обозначим за $\D_{p, n}$ множество слов
	таких, что любой их фактор с корнем $p$ имеет длину не больше $p\cdot n/(n-1)$ т.е.
	$$\D_{p, n} = \Big\{w\in\A_n^*: \forall v\subset w,
	\text{ выполняется импликация } per(v)=p\to \exp(v)\le\RT(n)\Big\}$$
\end{defn}

\begin{defn}
	Пусть $n\in\mN$, $P\subset\mN$.
	Обозначим за $\D_{P, n}$ множество слов таких,
	что любой их фактор с корнем $p$ из $P$ имеет длину не больше $p\cdot n/(n-1)$ т.е.
	$$\D_{P, n} = \bigcap_{p\in P}\D_{p, n}$$
\end{defn}

\begin{note}
	Заметим, что $\T_n=\D_{\mN, n}$ при $n\ge5$.
\end{note}

Тогда, можно сформулировать граничную теорему для алфавитов с 5-ю и более буквами в виде равенства
$|\D_{\mN, n}|=\infty$.

Ведём ещё несколько вспомогательных функций.

\begin{defn}
	Пусть $\V\subset\A_n^*$, а $w\in\A_n^*$, при этом $w$ и все слова в $\V$ имеют конечную длину.
	Тогда:
	\begin{itemize}
		\item
		$\p(w)=|w|-\per(w)$.
		Т.е. $\p(w)$ --- длина повтора слова $w$
		\item
		$\SUFF(w)=\{v\in\A_n^*: \exists u\in\A_n^* \text{ такое, что } uv=w\}$.
		Т.е. $\SUFF(w)$ --- множество всех суффиксов слова $w$.
		\item
		$\PREF(w)=\{u\in\A_n^*: \exists v\in\A_n^* \text{ такое, что } uv=w\}$.
		Т.е. $\PREF(w)$ --- множество всех префиксов слова $w$.
		\item
		$\suff_l(w)=v: v\in\SUFF(w), |v|=l$. Т.е. $\suff_l(w)$ --- суффикс слова $w$ длины $l$.
		\item
		$\pref_l(w)=u: u\in\PREF(w), |u|=l$. Т.е. $\suff_l(w)$ --- префикс слова $w$ длины $l$.
		\item
		$\l(\V)=\max\{m\in\mN: \pref_m(u)=\pref_m(v) \text{ где } u\ne v \text{ и } u, v\in\V\}$.
		Т.е. $\l(w)$ --- длина наибольшего общего префикса среди всех пар различных слов из $\V$.
		\item
		$\r(\V)=\max\{m\in\mN: \suff_m(u)=\suff_m(v) \text{ где } u\ne v \text{ и } u, v\in\V\}$.
		Т.е. $\r(w)$ --- длина наибольшего общего суфикса среди всех пар различных слов из $\V$.
		\item
		Пусть $a_1ua_2va_3ua_4\in w\in\A_n^*$, где $a_1,a_2,a_3,a_4\in\A_n$.
		Назовём повтор $u$ нерасширяемым (в, выделеных жирным, позициях $a_1\u a_2va_3\u a_4$),
		если $a_1\ne a_3$ и $a_2\ne a_4$.
		Повтор $u$ нерасширяем влево[вправо],
		если $a_1\ne a_3$[$a_2\ne a_4$].
	\end{itemize}
\end{defn}


\begin{defn}
	Множество всех общих факторов двух слов $u, v$ обозначим как
	$u\cap v=\{w: w\subset v,w\subset u\}$.
\end{defn}


\hyperlink{contents}{$\upuparrows$}

\subsubsection{Лемма 1}

Здесь мы разберём простую схему построения бесконечных граничных слов предпологая,
что уже существует бесконечное грничное слово над алфавитом с б$\acute{\text{о}}$льшим числом букв.
В данной работе мы называем эту схему спусковым морфизмом.

Данная ключевая лемма описанная здесь являтся авторской редакцией этой леммы из ДР1.
Доказательство леммы описанно с более подробными объяснениями каждого случая.
Серым текстом выделены важные или полезные дополнения, пояснения и разбор некоторых (тривиальных) случаев,
не рассмотренных в ДР1.
Поправки текста в ДР1 выделены красным.

\begin{lem}\label{l1}
	Пусть $L, n, k\in\mN$ и $n\ge5$.
	Рассмотрим набор из $n+k$ слов $\V\subset\A_n^L$ удовлетворяющий условиям:
	\begin{description}
		\item[(l1.c1)\label{l1c1}]
		$\per(v)=|v|(=L)$ для всех $v\in\V$;
		\item[(l1.c2)\label{l1c2}]
		$\lexp(uv)\le\frac{n}{n-1}$ для любых различных $u, v\in\V$;
		\item[(l1.c3)\label{l1c3}]
		$\max\{l+r: \suff_l(u)=\suff_l(v), \pref_r(v)=\pref_r(w)\}\le\frac{L}{n-1}$
		для любых различных (попарно) слов $u, v, w\in\V$.
	\end{description}
	Тогда $f(\dot{\omega})\in\T_n$ для любого бесконечного слова $\dot{\omega}\in\T_{n+k}$
	и биективнго морфизма $f: \A_{n+k}\leftrightarrow\V$.
\end{lem}
\begin{proof}
	Для начала заметим несколько свойств повторов:
	\begin{description}
		\item[(l1.p1)]
		Из \ref{l1c2} следует, что
		$\frac{L+\max\{\l(\V), \r(\V)\}}{L}\le\frac{n}{n-1}$.
		А значит
		\begin{equation}\tag{l1.p1}\label{l1p1}
		\max\{\l(\V), \r(\V)\}\le\frac{L}{n-1}
		\end{equation}
		\Big(
		Т.к. иначе получим противоречивую цепочку неравенств
		$\frac{n}{n-1}\ge\frac{L+\max\{\l(\V), \r(\V)\}}{L}>\frac{L+\frac{L}{n-1}}{L}=\frac{n}{n-1}$
		\Big)
		
		\item[(l1.p2)]
		Так же, из \ref{l1c2} следует, что $\forall v_1\ne v_2\in \V$ и $\forall v\in v_1\cap v_2$
		фактор $w\subset v_1v_2$ с перфиксом и суффиксом $v$ имеет ограничения на экспоненту
		$\frac{n}{n-1}\ge\exp(w)=\frac{|w|}{|w|-|v|}\ge\frac{2L}{2L-|v|}$.
		А значит
		\begin{equation}\tag{l1.p2}\label{l1p2}
		\forall v_1\ne v_2\in \V \text{ и } \forall v\in v_1\cap v_2 \text{ выполняется неравенство }
		|v|<\frac{2L}{n-1}
		\end{equation}
		\Big(
		Т.к. иначе получим противоречивую цепочку неравенств
		$\frac{n}{n-1}\ge\frac{2L}{2L-|v|}\ge\frac{2L}{2L-\frac{2L}{n-1}}=\frac{n-1}{n-2}>\frac{n}{n-1}$
		\Big)
		
		\item[(l1.p3)]
		Пусть $v_1, v_2, v_3\in \V$ попарно различные
		и $v\in v_1\cap v_2v_3$, где $v\not\subset v_2$ и $v\not\subset v_3$
		Тогда $\exists l_1, l_2>0$ такие,
		что $l_1+l_2=|v|$, $\suff_{l_1}(v_2)=\pref_{l_1}(v)$ и $\pref_{l_2}(v_3)=\suff_{l_2}(v)$.
		А значит, с учётом \ref{l1c2}
		\begin{equation}\tag{l1.p3}\label{l1p3}
		\text{ для попарно различных } v_1, v_2, v_3\in \V \text{ и } v\in v_1\cap v_2v_3,
		\text{ где } v\not\subset v_2 \text{ и } v\not\subset v_3 \text{ выполняется }
		|v|<\frac{2L}{n-1}
		\end{equation}
		\Big(
		Т.к. иначе, с учётом что $\max\{l_1, l_2\}\ge\frac{|v|}{2}$,
		получим противоречивую цепочку неравенств\\
		$\frac{n}{n-1}\ge\max\{\lexp(v_2v_1), \lexp(v_1v_3)\}\ge
		\max\Big\{\frac{L+|v|-l_2}{L-l_2}, \frac{L+|v|-l_1}{L-l_1}\Big\}>
		\max\Big\{\frac{L+l_1}{L}, \frac{L+l_2}{L}\Big\}\ge
		\frac{L+\frac{|v|}{2}}{L}\ge\frac{L+\frac{L}{n-1}}{L}=\frac{n}{n-1}$
		\Big)
		
		\item[(l1.p4)]
		Пусть $v_1, v_2, v_3, v_4\in \V$ попарно различные и $v\in v_1v_2\cap v_3v_4$.
		Тогда из \ref{l1c2} следует
		\begin{equation}\tag{l1.p4}\label{l1p4}
		\text{ для попарно различных } v_1, v_2, v_3, v_4\in \V \text{ и }
		v\in v_1v_2\cap v_3v_4 \text{ выполняется неравенство }
		|v|\le\frac{2L}{n-1}
		\end{equation}
		Докажем.
		Случаи, когда $v$ полностью лежит, хотябы, в одном из слов $v_1, v_2, v_3, v_4$
		уже рассмотрены в 2-х предыдущих свойствах повторов.
		При этом, неравенства строгие.
		
		Тогда рассмотрим $v$ детальнее.
		Пусть $v=u_1u_2u_3$,
		при этом БОО\footnote{Без ограничения общности} $u_1u_2\in\SUFF(v_1)$, $u_3\in\PREF(v_2)$ и
		$u_1\in\SUFF(v_3)$, $u_2u_3\in\PREF(v_4)$.
		В случае $|u_2|=0$ наше неравенство вытекает сразу из (\ref{l1p1}),
		где $|u_1|, |u_3|$ ограничены сверху числом $\frac{L}{n-1}$.
		При этом, равенство достигается только в случае равенства $|u_1|=|u_3|=\frac{L}{n-1}$.
		
		Значит достаточно рассмотреть только случай, когда $|u_1|, |u_2|, |u_3|>0$, что невозможно
		т.к. иначе получим противоречивое по \ref{l1c2} неравенство
		$\frac{5}{4}\ge\frac{n}{n-1}\ge\lexp(v_1v_4)\ge\exp(u_2u_2)\ge2$.
	\end{description}

	Докажем некоторые свойства для нашего морфизма $f(\dot{\omega})$:
	\begin{description}
		\item[(l1.p5)\label{l1p5}]
		Пусть $u, v, w\in\V$ и $u\ne v\ne w$,
		тогда по \ref{l1c2} $v\not\in uw$
		т.к. иначе $\frac{5}{4}\ge\frac{n}{n-1}\ge\lexp(uv)\ge2$.
		Тогда
		\begin{equation}\tag{l1.p5.e1}\label{l1p5e1}
		\forall v\in\V \text{ выполняется импликация }
		\big(v=\subw_m^L(f(\dot{\omega}))\big) \to
		\big(m\equiv1\bmod{L}\big)
		\footnote{Для простоты можно считать,
			что $\dot{\omega}$ бесконечно вправо,
			а нумерация букв и в $\dot{\omega}$ и в $f(\dot{\omega})$ начинается с 1.
			В общем случае, можно начинать отсчёт с любой буквы в $\dot{\omega}$,
			а в $f(\dot{\omega})$ с первой буквы образа выбранной буквы в $\dot{\omega}$.
		}
		\end{equation}
		Т.е. любое слово из $\V$, найденое в слове $f(\dot{\omega})$
		может быть только образом одной буквы из $\dot{\omega}$.
		
		Тогда, с учётом (\ref{l1p4})
		\begin{equation}\tag{l1.p5.e2}\label{l1p5e2}
		\forall v\subset f(\dot{\omega}) \text{ выполняется импликация }
		\Big(|v|-\per(v)\ge\frac{2L}{n-1}\Big)\to
		(\per(v)\equiv0\bmod L)
		\end{equation}
		
		Докажем это.
		Пусть $\p(v)\ge\frac{2L}{n-1}$ (напомним, что $\p(v)=|v|-\per(v)$), тогда, с учётом (\ref{l1p4}), следует,
		что повтор слова $v$ нерасширяем в $f(\dot{\omega})$:
		\begin{itemize}
			\item
			Либо \big(в случае $\p(v)>\frac{2L}{n-1}$\big) содержит хотябы 1 элемент из $\V$.
			Но тогда, по (\ref{l1p5e1}), этот элемент начинается с позиции $1\bmod L$ в слове $f(\dot{\omega})$.
			Тогда и левый и правый повторы слова $v$ начинаются с одинакового смещения по $\bmod L$,
			из чего и следует равенство $\per(v)\equiv0\bmod L$.
			\item
			Либо \big(в случае $\p(v)=\frac{2L}{n-1}$\big) по (\ref{l1p2}) и (\ref{l1p3})
			повтор слова $v$ не содержится ни в оном слове из $\V$,
			и по (\ref{l1p4}) представ$\acute{\text{и}}$м как $v_1v_2$,
			где $|v_1|=|v_2|$ и $v_1\in\SUFF(u)$, $v_2\in\PREF(w)$ для некоторых $v_1, v_2\in\V$.
			Тогда и в левом и вправом повторе слова $v$
			фактор $v_2$ начинается с позиции $1\bmod L$ в слове $f(\dot{\omega})$.
			И, аналогично предыдущему случаю, получаем равенство $\per(v)\equiv0\bmod L$.
		\end{itemize}
		
		Т.е. по (\ref{l1p5e2}), если длина повтора фактора из $f(\dot{\omega})$ больше
		$\frac{2L}{n-1}\ge\frac{L}{2}$,
		то он (повтор) представим в виде конкатенации слов из $\V$
		с добавленными по краям короткими словами длины (по (\ref{l1p1})) не больше
		$\max\{\l(\V), \r(\V)\}\le\frac{L}{n-1}\le\frac{L}{4}$ (каждое).
		Отсюда, в частности, следует,
		что не существует нерасширяемых повторов с длиной из интервала $\big(\frac{2L}{n-1}, L\big)$.
		
		Из (\ref{l1p5e2}) получаем,
		что для слов с повторами длины $\tau{\color{red}{\ge}}\frac{2L}{n-1}$ эквивалентны проверки условий
		\footnote{В оригинале (т.е. в дипломной работе 2011г.) было ошибочное условие $\tau\le\frac{2L}{n-1}$,
			где $\tau$ обозначалось как <<$lr$>>.
		}
	\begin{equation}\tag{l1.p5}\label{l1p5e3}
		f(\dot{\omega})\in\D_{\{p>0:\ \frac{p+\tau}{p}\le\frac{n}{n-1}\}, n}
		\leftrightarrow
		f(\dot{\omega})\in\D_{\{p>0:\ \frac{p+\tau}{p}\le\frac{n}{n-1}, p\equiv0\bmod L\}, n}
		\end{equation}
		Т.е. чтобы проверить принадлежность $f(\dot{\omega})$ к $\D_{p, n}$ при таких $p>0$,
		что $\frac{p+\tau}{p}\le\frac{n}{n-1}$ и $\tau\ge\frac{2L}{n-1}$,
		достаточно ограничиться случаями когда длина корня $p\equiv0\bmod L$.
		
		Упростим условие для длин корней.
		Множество длин корней $p$, для которых $\frac{p+\tau}{p}\le\frac{n}{n-1}$
		при всевозможных $\tau\ge\frac{2L}{n-1}$,
		эквивалентно множеству всех $p$, когда $p\ge(n-1)\tau\ge(n-1)\frac{2L}{n-1}=2L$.
		Тогда перепишем эквивалентность
		\begin{equation}\tag{l1.p5'}\label{l1p5e4}
		f(\dot{\omega})\in\D_{\{p>0:\ \frac{p+\tau}{p}\le\frac{n}{n-1}, \tau\ge\frac{2L}{n-1}\}, n}
		\leftrightarrow
		f(\dot{\omega})\in\D_{\{p\ge2L\}, n}
		\leftrightarrow
		f(\dot{\omega})\in\D_{p\ge2L: p\equiv0\bmod L\}, n}
		\end{equation}
		
		При этом, если $0\le\tau\le\frac{2L}{n-1}$ и $p\ge2L$, то гарантированно выполняется
		$\frac{p+\tau}{p}\le\frac{2L+\tau}{2L}\le\frac{2L+\frac{2L}{n-1}}{2L}=\frac{n}{n-1}=\RT(n)$.
		Т.е. $\forall v\subset f(\dot{\omega})$ условие $\exp(v)\le\RT(n)$ автоматически выполняется при $\per(v)\ge2L$ и $\p(v)\le\frac{2L}{n-1}$.
		А значит, если условие $\exp(v)\le\RT(n)$ нарушается при $\per(v)\ge2L$,
		то $\p(v)>\frac{2L}{n-1}$, а значит, по (\ref{l1p5e2}), $\per(v)\equiv0\bmod L$.
		Проще говоря, если найдётся $v\subset f(\dot{\omega})$ с корнем $\per(v)\ge2L$ и экспонентой $\exp(v)>\RT(n)$,
		то он имеет и длину повтора $\p(v)>\frac{2L}{n-1}$ и корень кратный $L$.

	\end{description}
	
	\sout{\color{red}
		{
			Значит для проверки $f(\dot{\omega})\in\D_{\{p>0: \frac{p+\l(V)+\r(\V)}{p}\le\frac{n}{n-1}\}, n}$
			достаточно убедиться, что $f(\dot{\omega})\in\D_{\{p>0: p\equiv0\bmod L\}, n}$???
		}
	}
	\footnote{В оригинале данное некорректное утверждение не используется, поэтому его можно игнорировать.
	}
	
	Обозначим
	\begin{itemize}
		\item
		$w_m^M=subw_m^M(f(\dot{\omega}))$;
		\item
		$l_m^M=\max\{l: \per(w_{m-l}^M)=\per(w_m^M)\}$ (т.е. $l_m^M$ это максимальное расширение $w_m^M$ влево в $f(\dot{\omega})$);
		\item
		$r_m^M=\max\{r: \per(w_m^{M+r})=\per(w_m^M)\}$ (т.е. $r_m^M$ это максимальное расширение $w_m^M$ вправо в $f(\dot{\omega})$);
		\item
		$W_m^M=w_{m-l_m^M}^{{\color{red} l_m^M}+M+r_m^M}$.
		\footnote{В оригинале к длине не добавлено левое расширение $l_m^M$.
		}
	\end{itemize}
	Здесь $W_m^M$ это нерасширяемый фактор в $f(\dot{\omega})$,
	содержащий фактор $w_m^M$ с общим корнем т.е. $\per(W_m^M)=\per(w_m^M)$.
	БОО можем считать, что $w_m^M$ уже не расширяем в $f(\dot{\omega})$ т.е. $w_m^M=W_m^M$ и $l_m^M=r_m^M=0$.
	\footnote{В оригинале об этом допущении не сказано, но нигде больше не используются ни $l_m^M$ ни $r_m^M$.
	}
	
	\begin{description}
	\item[-]
		Разберём случаи,
		когда длины корней факторов из $f(\dot{\omega})$ не меньше $p\ge2L$.
		Нарушение условия $f(\dot{\omega})\in\D_{\{p\ge2L\}, n}$,
		как мы уже выяснили, возможно только тогда,
		когда длина повтора не меньше $\frac{2L}{n-1}$.
		Поэтому, будем рассматривать только такие факторы из $f(\dot{\omega})$.
		
		Для б$\acute{\text{о}}$льшей определённости возьмём произвольный нерасширяемый фактор $\dot{w}$ в $\dot{\omega}$.
		Возмём $m$ и $M$ такие, что $w_m^M=f(\dot{w})$.
		Тогда по свойству морфизма $f$ получим $|w_m^M|=|\dot{w}|L\equiv0\bmod L$.
		Используя свойство (\ref{l1p5e1}) получим $m\equiv1\bmod L$,
		а по свойству (\ref{l1p5e2}) получим $\per(w_m^M)\equiv0\bmod L$.
		
		Обозначим $K_1=|\dot{w}|, K_2=\per(\dot{w})$.
		Т.к. $\dot{w}$ нерасширяем в $\dot{\omega}$ влево [вправо],
		то $w_m^M$ расширяемо влево [вправо] в $f(\dot{\omega})$ не более чем на $\l(\V)$ [на $\r(\V)$].
		
		Тогда выполняется неравенство $|W_m^M|=M\le\r(\V)+K_1L+\l(\V)$.
		А так же $\per(W_m^M)=\per(w_m^M)=K_2L$.
		
		\begin{description}
			\item[(l1.p6)\label{l1p6}]
			Случай, когда $K_2\ge c(n+k-1)$, для вещественных $ c\ge2$
			(здесь $c$ можно рассматривать как длину повтора,
			и достаточно проверить для целых $c$, но это удлинит доказательство из-за необходимости доказать,
			что достаточно рассмотреть получаемое разреженное множество длин корней $K_2$.
			Такая возможность разредить $K_2$ обуславливается тем,
			что, с ростом длины повторов на 1, минимальное требуемое $K_2$ увеличивается с шагом $n+k-1$).
			
			По условию $\dot{\omega}\in\T_{n+k}$, тогда оценим сверху $K_1\le K_2\frac{n+k}{n+k-1}$.
			Т.е. чтобы проверить, что слово $f(\dot{\omega})$ из множества $\D_{K_2L, n}$ достаточно проверить,
			что произвольное слово $W_m^M$ с корнем $\per(W_m^M)=K_2L$,
			и ограниченной сверху длиной $K_2\frac{n+k}{n+k-1}L+\l(\V)+\r(\V)$,
			имеет экспоненту не больше $\RT(n)$.
			Тогда, с учётом (\ref{l1p1}) и $k\ge1$
			$$
			\frac{|W_m^M|}{\per(W_m^M)}\le
			\frac{K_1L+\l(\V)+\r(\V)}{K_2L}\le
			\frac{K_2L\frac{n+k}{n+k-1}+\frac{2L}{n-1}}{K_2L}\le
			\frac{c(n+k)+\frac{2}{n-1}}{c(n+k-1)}\le
			\frac{2(n+k)+\frac{2k}{n-1}}{2(n+k-1)}=
			\frac{n}{n-1}
			$$
			\begin{equation}\tag{l1.p6}\label{l1p6e}
			\text{Из чего следует, что } f(\dot{\omega})\in\D_{\{p:\ p\ge2L(n+k-1)\}, n}.
			\text{ В частности } \fbox{$f(\dot{\omega})\in\D_{\{L\cdot p:\ p\ge2(n+k-1)\}, n}$}.
			\end{equation}
			
			\item[(l1.p7)\label{l1p7}]
			Рассмотрим случай $K_2\ge n+k$ и $K_1=K_2+1$.
			И снова, с учётом (\ref{l1p1}) и $k\ge1$
			$$
			\frac{|W_m^M|}{\per(W_m^M)}\le
			\frac{K_1L+\l(\V)+\r(\V)}{K_2L}\le
			\frac{K_2+1}{K_2}+\frac{2L}{(n-1)K_2L}\le
			\frac{n+k+1}{n+k}+\frac{k+1}{(n-1)(n+k)}=
			\frac{n}{n-1}
			$$
			Т.е. для $\dot{w}$, с длиной повтора $1$ и корнях $\ge n+k$,
			его расширеный образ $W_m^M$ имеет корректную экспоненту.
			При этом, $\dot{w}$, в силу условия $\dot{w}\in\T_{n+k}$,
			не может иметь длину повтора больше $1$ при $K_2<2(n+k-1)$.
			\begin{equation}\tag{l1.p7}\label{l1p7e}
			\text{Это доказывает, что } \fbox{$f(\dot{\omega})\in\D_{\{L\cdot p:\ n+k\le p<2(n+k-1)\}, n}$}.
			\end{equation}
			
			\item[(l1.p8)\label{l1p8}]
			Теперь, пусть $K_2=n+k-1$, тогда по условию $\dot{\omega}\in\T_{n+k}$ получим $K_1\le n+k$.
			Заметим, что это последний нерассмотренный случай,
			когда у нашего прообраза $\dot{w}$ имеется непустой повтор.
			При этом, длина повтора равна $1$, и все буквы между повторами различны.
			
			Тогда, с учётом (\ref{l1p5e1}), для прообраза фактора $w_m^M$ найдётся позиция $i$,
			что $m=(i-1)L+1$ т.е. $\dot{w}=\dot{\omega}[i..i+n+k]$. 
			Тогда $\dot{\omega}[i]=\dot{\omega}[i+n+k]$ и $\dot{\omega}[i-1]=\dot{\omega}[i+n+k+1]$
			(докажите это).
			
			Обозначим $a=\dot{\omega}[i-1], b=\dot{\omega}[i], c=\dot{\omega}[i+1], d=\dot{\omega}[i+n+k-1]$.
			Тогда $w_m^M=bc\dot{\omega}[i+2...i+n+k-2]db$ и $W_m^M\subset f(abc\dot{\omega}[i+2...i+n+k-2]dba)$
			при различных $a, b, c, d\in\A_{n+k}$ (докажите, что они различны).
			
			Тогда, для попарно различных $a, c, d\in\A_{n+k}$ получим различные $f(a), f(c), f(d)\in\V$.
			Тогда оценим $|W_m^M|\le K_1L+\max\{l+r: \pref_l(f(c))=\pref_l(f(a)), \suff_r(f(a))=\suff_r(f(d))\}$
			Тогда, с учётом \ref{l1c3} и $k\ge1$
			\footnote{Заметьте, что при $k\ge2$, условие \ref{l1c3} можно не использовать для этого случая ($K_2=n+k-1$){\color{red}}
			}
			$$
			\frac{|W_m^M|}{\per(W_m^M)}\le
			\frac{K_1L+\frac{kL}{n-1}}{K_2L}=
			\frac{(n+k)(n-1)+k}{(n+k-1)(n-1)}=
			\frac{n}{n-1}
			$$
			\begin{equation}\tag{l1.p8}\label{l1p8e}
			\text{Тогда получаем, что } \fbox{$f(\dot{\omega})\in\D_{L(n+k-1), n}$}.
			\end{equation}
			
			\item[(l1.p9)\label{l1p9}]
			Рассмотрим оставшиеся случаи когда корни $p\ge2L$. Т.е $2\le K_2<n+k-1$,
			тогда, по условию $\dot{w}\in\T_{n+k}$, получаем $K_1=K_2$.
			{\cGr Тогда из (\ref{l1p4}) слдует, что $|W_m^M|-\per(W_m^M)\le\frac{2L}{n-1}$.}
			Тогда оценим $\exp(W_m^M)$
			$$
			\frac{|W_m^M|}{\per(W_m^M)}\le
			\frac{K_1L+\frac{2L}{n-1}}{K_2L}=
			\frac{K_2(n-1)+2}{K_2(n-1)}\le
			\frac{2(n-1)+2}{2(n-1)}=
			\frac{n}{n-1}
			$$
			\begin{equation}\tag{l1.p9}\label{l1p9e}
			\text{Получаем } \fbox{$f(\dot{\omega})\in\D_{\{L\cdot p:\ 2\le p<n+k-1\}, n}$}
			\end{equation}
			
			В дополнение к этому случаю приведём ещё одну идею доказательства.
			Как мы уже выяснили, у нерасширяемого $W_m^M$ нет повторов с длиной из интервала $(\frac{2L}{n-1}, L)$
			а при длине повтора у $W_m^M$
			не больше  $\frac{2L}{n-1}$ и длине корня не меньше $2L$ всегда $\exp(W_m^M)\le\RT(n)$.
			При этом, в силу неравенства $\l(\V)+\r(\V)<L$,
			случай, когда длина повтора у $W_m^M$ не меньше $L$,
			невозможен т.к. он требует чтобы $K_2<K_1$, что противоречит условию $\dot{w}\in\T_{n+k}$.
		\end{description}
		
		Из (\ref{l1p6e}), (\ref{l1p7e}), (\ref{l1p8e}) и (\ref{l1p9e}) слеует
		$f(\dot{\omega})\in\D_{\{L\cdot p:\ p\ge2\}, n}$,
		\begin{equation}\tag{l1.p10}\label{l1p10e}
		\text{тогда по (\ref{l1p5e4}) получим } \fbox{\fbox{$f(\dot{\omega})\in\D_{\{p:\ p\ge2L\}, n}$}}
		\end{equation}
		
		\item[-]
		Теперь осталось разобрать случаи,
		когда длины корней факторов из $f(\dot{\omega})$ меньше $2L$.
		
		Пусть $p_m^M=\per(W_m^M)<2L$.
		По условию $\dot{w}\in\T_{n+k}$ получим $|w_m^M|\le p_m^M\cdot\RT(n+k)$, тогда, с учётом (\ref{l1p1}) и $n\ge5$
		$$|W_m^M|\le|w_m^M|+\l(\V)+\r(\V)\le
		p_m^M\frac{n+k}{n+k-1}+\frac{2L}{n-1}<
		2L\bigg(1+\frac{1}{n+k-1}+\frac{1}{n-1}\bigg)<
		2L\bigg(1+\frac{1}{4}+\frac{1}{4}\bigg)=
		3L$$
		Т.е. $W_m^M$ лежит в конкатенации не более чем 4-х слов $v_1, v_2, v_3, v_4\in\V$.
		БОО пусть $W_m^M\subseteq v_1v_2v_3v_4\subset f(\dot{\omega})$.
		При этом, т.к. $n+k>4$ и прообраз фактора $v_1v_2v_3v_4$ из $\T_{n+k}$,
		то эти слова попарно различны.
		{\cGr Случай, когда $W_m^M$ не полностью лежит в конкатенации 3-х из этих слов исключается т.к. иначе,
		в силу $p_m^M<2L$ и условия \ref{l1c2},
		получим противоречие $\RT(n)\ge\lexp(v_3v_2)\ge2$.}
		\footnote{В оригинале об этом не сказано, но неявно этот случай исключается, при рассмотрении $m_L+M-1>2L$ (в оригинале).
		}
		
		Обозначим $m_L\equiv((m{\color{red}-1})\bmod L){\color{red}+1}$.
		\footnote{В оригинале не учтено смещение на 1 для индексов, начинающихся с 1.}
		Т.е. $m_L$ это позиция первой буквы $W_m^M$ некоторого $\V${--}образа.
		
		Если $(m_L-1)+M\le2L$, то найдутся попарно различные  $u, v\in\V$, что $W_m^M\subseteq uv$.
		Тогда, по условию \ref{l1c2} получим $\lexp(W_m^M)\le\lexp(uv)\le\RT(n)$.
		
		Тогда считаем, что $(m_L-1)+M>2L$.
		Тогда найдутся попарно различные $u, v, w\in\V$ такие, что $W_m^M\subseteq uvw$.
		При этом, первая буква $W_m^M$ лежит в $u$, а последняя в $w$.
		{\cGr БОО считаем, что $uvw=x'W_m^Mx''$, где $|x''|\le|x'|=m_L-1$.}
		\footnote{В оригинале не сказано, но из 2-х вариантов только такой полностью соответствует дальнейшим рассуждениям.
		}
		Тогда докажем оставшиеся случаи для длин корней:
		\begin{description}
			\item[(l1.p11)\label{l1p11}]
			$p_m^M>L$ ($p_m^M<2L$).
			Тогда $m_L+p_m^M>2L$,
			т.е. (с учётом $|x'|\ge|x''|$) оба повтора слова $W_m^M$ не могут одновременно пересекаться с $v$
			(т.к. иначе получим $\lexp(uw)\ge2$, что противоречит \ref{l1c2}).
			Случай, когда оба повтора из $W_m^M$ не пересекаются с $v$ тривиален т.к. в этом случае
			$\exp(W_m^M)<\lexp(uw)\le\RT(n)$.
			\footnote{В оригинале этот случай неявно рассмотрен при разборе случая $b\le0$,
				где $b$ определённо далее.
			}
			
			Тогда можно представить $u.v.w=x'y_0.y_1z_0.z_1y_0y_1x''$,
			где $y_0y_1$- повтор слова $W_m^M$, а $|y_0|, |y_1|>0$.
			
			Обозначим:
			\footnote{В оригинале автор забыл пояснить смысл параметров $a,b,c,d$, но,
				используя контекст их использования (в неравенствах),
				здесь мы их свяжем с правильно разбитым $u.v.w$ на факторы.}
			\begin{itemize}
				\item
				$a=L-(m_L-1)>0$.
				Т.е. $a=|x'y_0|-|x'|=|y_0|>0$.
				\item
				$b=M-p_m^M-a>0$.
				Т.е. $b=\p(W_m^M)-|y_0|=|y_0y_1|-|y_0|=|y_1|>0$.
				\item
				$c=m_L+p_m^M-2L>0$.
				Т.е. $c=|x'|+per(W_m^M)-2L=|x'y_0y_1z_0z_1|-2L=|z_1|$.
				При этом $|z_1|>0$ т.к. иначе $\lexp(uw)=\lexp(x'y_0y_0y_1x'')\ge2$, что противоречит \ref{l1c2}.
				\item
				$d=L-a-b-c\ge0$.
				Т.е. $d=L-(a+b+c)=L-(M-p_m^M+m_L+p_m^M-2L)=3L-(m_L+M)=3L-|x'y_0y_1z_0z_1y_0y_1|=|x''|\ge0$.
			\end{itemize}
		
			Т.к. $lexp(uw)\le\RT(n)$, то $\frac{2a+c}{a+c}\le\frac{n}{n-1}$, тогда $a\le\frac{c}{n-2}$.
			Т.к. $lexp(wv)\le\RT(n)$, то $\frac{2b+d}{b+d}\le\frac{n}{n-1}$, тогда $b\le\frac{d}{n-2}$.

			Тогда $L=a+b+c+d\le\frac{c+d}{n-2}+c+d=\frac{(c+d)(n-1)}{n-2}$.
			Откуда $c+d\ge L\frac{n-2}{n-1}$.
			
			Тогда длина повтора $M-p_m^M=a+b=L-(c+d)\le L-L\frac{n-2}{n-1}=\frac{L}{n-1}$.
			Из чего, с учётом $p_m^M>L$, получаем
			
			$$
			\exp(W_m^M)=\frac{p_m^M+a+b}{p_m^M}
			=\frac{n-1+\frac{(a+b)(n-1)}{p_m^M}}{n-1}
			{\color{red}\le}\frac{n-1+\frac{L}{p_m^M}}{n-1}
			<\frac{n-1+1}{n-1}
			=\RT(n).
			\footnote{В оригинале вместо одного из неравенств ошибочно стоит равенство.
			}
			$$
			\begin{equation}\tag{l1.p11}\label{l1p11e}
			\text{Получаем } \fbox{$f(\dot{\omega})\in\D_{\{p:\ L<p<2L\}, n}$}
			\end{equation}
			
			\item[(l1.p12)\label{l1p12}]
			$p_m^M=L$. 
			Тогда левый и правый повторы $W_m^M$ начинаются в одинаковых смещениях $l$ от правых концов в $u$ и $v$ соответственно.
			Аналогично, заканчиваются они в одинаковых смещениях $r$ от левых концов в $v$ и $w$.
			А значит длина повтора слова $W_m^M$ равна $l+r$.
			Тогда, с учётом \ref{l1c3}, для наших, попарно различных $u, v, w$
			$$
			\exp(W_m^M)
			=\frac{L+l+r}{L}
			\le
			1+
			\max_{
				\begin{array}{c}
				\scriptstyle u,v,w\in\V,\\
				\scriptstyle|\{u,v,w\}|=3
				\end{array}
			}
			\bigg\{
			\frac{l+r}{L}:
			\begin{array}{c}
			\suff_l(u)=\suff_l(v),\\
			\pref_r(v)=\pref_r(w)
			\end{array}
			\bigg\}
			\le\frac{L+\frac{L}{n-1}}{L}
			=\frac{n}{n-1}
			\footnote{В 25 ноябрьском оригинале не полностью скопирована функция $\max\{\dots\}$.
				Но в 24 июньском оригинале она корректна.
			}
			$$
			\begin{equation}\tag{l1.p12}\label{l1p12e}
			\text{Получаем } \fbox{$f(\dot{\omega})\in\D_{L, n}$}
			\end{equation}
			
			\item[(l1.p13)\label{l1p13}]
			$0<p_m^M<L$.
			Пусть $l$, $r$ --- длина пересечения $W_m^M$ с $u$ и $w$ соответственно,
			а $l'$ --- длина пересечения правого повтора в $W_m^M$ с $v$.
			Тогда $l<l'$ т.к. левый конец (самая левая буква) левого повтора лежит в $u$,
			а расстояние между левыми концами левого и правого повторов в $W_m^M$ равно $p_m^M<L=|v|$.
			
			Тогда, длина общего префикса и суффикса в $v$ равна $\p(W_m^M)-(l+r)=(l'+r)-(l+r)=l'-l>0$.
			Но это невозможо т.к. иначе нарушается условие \ref{l1c1}.
			\begin{equation}\tag{l1.p13}\label{l1p13e}
			\text{Значит } \fbox{$f(\dot{\omega})\in\D_{\{p:\ 0<p<L\}, n}$}
			\end{equation}
			
		\end{description}
	\end{description}
	
	Объединяя факты (\ref{l1p10e}), (\ref{l1p11e}), (\ref{l1p12e}) и (\ref{l1p13e}),
	получим $f(\dot{\omega})\in\D_{\mN, n}$, что и требовалось доказать.
\end{proof}

\subsubsection{Замечания и дополнения}
Лемма \ref{l1} может быть легко обобщена,
если допустить возможность отображать букву не в единственный образ из $\V$,
а некоторый альтернативный не равный ни одному из зарезервированных образов для $n+k$ букв.
Другими словами, число образов одной буквы может быть больше 1.
Это обобщение позволит проще доказывать корректность экспоненциального роста ГС с дополнительным образом.

\hyperlink{contents}{$\upuparrows$}

\subsection{Лемма 2. Метод контекстнозависимой (КЗ) подстановки}

\subsubsection{Дополнительные определения}

Усилим условия граничности слов

\begin{defn}
	Пусть $p, n\in\mN$, и $\eps\in\mR$.
	Обозначим за $\D_{p, n}^\eps$ множество слов
	таких, что любой их фактор с корнем $p$ имеет длину не больше $p\cdot n/(n-1)-\eps$ т.е.
	$$\D_{p, n}^\eps = \bigg\{w\in\A_n^*: \forall v\subset w,
	\text{ выполняется импликация } per(v)=p\to\frac{|v|+\eps}{\per(v)}\le\frac{n}{n-1}\bigg\}$$
\end{defn}

\begin{defn}
	Пусть $n\in\mN$, $P\subset\mN$, и $\eps\in\mR$.
	Обозначим за $\D_{P, n}^\eps$ множество слов таких,
	что любой их фактор с корнем $p$ из $P$ имеет длину не больше $p\cdot n/(n-1)-\eps$ т.е.
	$$\D_{P, n}^\eps = \bigcap_{p\in P}\D_{p, n}^\eps$$
\end{defn}

\begin{defn}
	Пусть $v'\subset f(uavbw)$, такое что $f(v)\subseteq v'$, $f(vb)\not\subseteq v'$ и $f(av)\not\subseteq v'$,
	тогда $v$ назовём целым прообразом слова $v'$.
\end{defn}

\begin{defn}
	$\exp^\eps_P(v)=\frac{|u|+\eps}{\per(u)}$ при $\per(u)\in P$
	и $\exp^\eps_p(v)=\frac{|u|}{\per(u)}$ в остальных случаях.
	По умолчанию $P=\{p\ge3n-3\}$.
\end{defn}

\begin{defn}
	$\lexp_P^\eps(v)=\max\Big\{\lexp(v),
	\sup\Big\{\exp^\eps(v):\ u\subseteq v, \per(u)\in P\Big\}\Big\}$.
	По умолчанию $P=\{p\ge3n-3\}$.
\end{defn}

\begin{note}\label{n1}
	$\lexp_P^\eps(v)\le\RT(n)\leftrightarrow v\in\D_{P,n}^\eps\cap\T_n$.
	В частности $\lexp_{\ge3n-1}^\eps(v)\le\RT(n)\leftrightarrow v\in\D_{\ge3n-1,n}^\eps\cap\T_n$.
\end{note}

\begin{note}\label{n2}
	$\exp^\eps_P(v')\le\lexp_P^\eps(v)$, для любых $v'\subseteq v$.
\end{note}

\begin{defn}
	Произвольное слово из $\V$ назовём $\V${--}образом.
\end{defn}

\begin{note}\label{nt_3nm1}
В ГС не может быть 2 повтора с периодом менее $2n-1$ и 3 повтора с периодом менее $3n-1$ (для $n\ge5$).
\end{note}

Что очевино т.к. для 2 повтора не может быть периода $2n-2$ т.к. для этого нужно
чтобы 2 раза 2 соседние буквы сдвигались влево кодом Пансьё, но это нарушит граничность.
Аналогично, для 3 повтора --- не может быть 3 повтора с периодом $3n-2$ нужно,
чтобы 3 соседние буквы хотябы по 2 раза смещались влево за 3 хода, это неизбежно приведёт (по принципу дирихле)
к необходимости за 1 ход сместить 2 соседние буквы влево, а это нарушит граничность.
Т.е. получаем

\begin{note}\label{nt_D3nm1}
	$\D_{\ge3n-3,n}^\eps\cap\T_n=\D_{\ge3n-1,n}^\eps\cap\T_n$.
\end{note}

Заметьте, что граница $\eps$ сверху растёт линейно от длины периода при одинаковой длине повтора.
Или в общем --- при любой длине повтора $r$ для слов с длинами периодов $r(n-1)+c_1$ и $r(n-1)+c_2$
граница для $\eps$ сверху зависит только от разницы $c_2-c_1$, при этом линейно (т.е. не зависимо от $r$).

Более того, при более длинных повторах, даже если достигается граница повторяемости,
то количество букв в повторе между повторами должно быть меньше на 1 числа кратного 3-м (для нарушения свойства граничности).
Поэтому, получаем естественную многозначную подстановку, предложенную в ДР1 (и тексте лекций).


\begin{sgn}
	Обозначим $\D_{\ge r(n-1),n}^\eps\cap\T_n$ как $\D_{r,n}^\eps${--}ГС.
\end{sgn}

\hyperlink{contents}{$\upuparrows$}

\subsubsection{Лемма 2}

\paragraph{\large Неформальное предисловие.}

Обычный (т.е. контекстносвободный (КС)) одинаково удлиняющий
(т.е. длины образов букв одинаковы и больше 1) морфизм $f:\A_n\to\V$,
скорее всего, не позволяет построить
граничное слово (ГС) на основе другого ГС длины не менее $n+1$, даже с более сильными свойствами.
Невозможность объясняется наличием хотябы 1-го $n${--}фактора с повтором длины 1 в любой строке из $n+1$ буквы,
что гарантирует достижение границы повторяемости $\RT(n)$ в достаточно длинных словах.
А значит, из-за добавления общих префиксов или суффиксов (если они не пусты)
от образов соседних букв по краям, будет нарушаться граничность.
Как избежать этой проблемы?

Одно из направлений --- сокращение общих префиксов и суффиксов образов.
Например, как-то сократить $\l(\V)$ и $\r(\V)$ до 0, что не так просто найти, хотябы для каких-то пар в наборе (если, вообще, существуют).
Понятно, что, если в наборах больше $n$ слов, то неизбежно $\l(\V),\r(\V)>0$.
Но по нашему методу построения в Лемме \ref{l3} всегда $\l(\V),\r(\V)\ge n$.

Другой подход более реалистичен --- разбиения длинных повторов в образе
за счёт периодической подмены образа, хотябы, одной буквы в повторе прообраза.

В оригинальной работе формулировка подходящей подстановки оказалась не полностью корректной и, даже, неоднозначной.
Но на 1-й нашей лекции семинара в марте 2011г. алгоритм КЗ подстановки была изложена достаточно полно т.к.
вся лекция была посвящена этой лемме (на сколько помнит автор).
В добавок, на 2-й лекции один из присутствующих (Сенчёнок Т.А., не присутствовшая на 1-й лекции) 
достаточно ясно пересказал нам этот алгоритм для уточнения --- верно ли люди поняли.
В общем-то тогда и автор сам начал понимать, что он рассказывал на 1-й лекции 

Формальное определение нашей КЗ подстановки оказалось требует б$\acute{\text{о}}$льшей аккуратности.
Покажем некоторые варианты из них.
Например, определим КЗ подстановку для конечных слов (действующую справа на лево):
$$
f(ua)=f(u)f(a)=f(u)v_{a, i}, \text{ где }
i=\big(\big(|u|_a-1\big)\bmod{3}\big)+1, v_{a, i}\in\V
$$

Определим его для бесконечных справа слов:
$$
f_u(av)=v_{a,i}f_{ua}(v), \text{ где }
i=\big(\big(|ua|_a-1\big)\bmod{3}\big)+1, v_{a, i}\in\V,
f_w(\lambda)=\lambda,
\forall w\in \A_n^*
$$
Индекс $u$ в обозначении $f_u$, можно рассматривать как первый аргумент подстановки.
Он нужен только для сохранения контекста --- количество пройденых букв $a$ слева, для каждого $a\in\A_n$.
Тогда, для вычисления образа слова $\dot{\omega}$ нужно взять $f_\lambda(\dot{\omega})$.

Для определения подстановки над бесконечным в обе стороны словом,
достаточно обозначить промежуток между любыми 2-мя буквами как точку отсчёта,
и определить нашу подстановку для правого подслова от точки отсчёта как для бесконечных справа слов. 
А для левого так же (можно обе подстановки обозначать разными переменными),
только отсчёт должен быть в обратную сторону (т.е. отрицательные числа).
Только индекс $i$ всегда определять в $\{1,2,3\}$ с сохранением класса эквивалентности по $(\bmod{3})$.

Для полноты добавим ещё определение через КЗ грамматики:

Для каждого $a\in\A_n$ единожды поставим по 3 своих нетерминальных символа $S_{a,i}$, $\forall i\in\{1,2,3\}$.
Тогда зададим КЗ подстановку $f:\A_n^*\to\A_n^*$ над словом $w\in\A_n^*$.

1) $f(w)\to VS_{s_1,1}S_{s_2,1}...S_{s_n,1}w$, где $s_1,s_2,...,s_n\in\A_n$ и все различны

2) $S_{a,i}a\to a_iS_{a,(i\pmod 3) + 1}$, $\forall a\in\A_n$

3) $S_{a,i}\bar{a}\to \bar{a}S_{a,i}$, $\forall\bar{a}\not=a$ 

4) $S_{a,i}\lambda\to\lambda$

5) $Va_i\to v_{a,i}V$, $\forall a\in\A_n$, $\forall i\in\{1,2,3\}$, где $v_{a,i}$ --- $i$-й $\V${--}образ буквы $a$

6) $V\lambda\to\lambda$

Заметим, что наша подстановка строит слово так, что начала 2-х одинаковых $\V${--}образов $u$
находятся на расстоянии не менее $3(n-1)L$
(где $L$ --- длина каждого $\V${--}образа).
При этом, если $u$ содержится в конкатенации 2-х других $\V${--}образов $v, w$,
то $\max\{\lexp(vu), \lexp(uw)\}\ge2>\RT(n)$ при $n>2$.
Т.е. чтобы избегать случаев $u\subset vw$ для факторов с периодом менее $3(n-1)L$,
достаточно добавить условие граничности
для любой пары различных $\V${--}образов.
Таким образом, добавляя конечное число проверок для нашего множества $\V$,
мы облегчаем задачу проверки граничности факторов с короткими периоами, полученных нашей подстановкой.

Отметим, что в оригинальном доказательстве для длин периодов фактора образа менее $3(n-1)L$
не отмечено замечание, что повторы образа не содержат в себе $\V${--}образов
(но во всех подслучаях рассматриваются только такие повторы).
Это одно из основных свойств, которое разделяет факторы с длинными повторами и короткими
по аналогии с замечанием в Лемме \ref{l1} об отсутствии слов с повторами определённых длин.
Это связано с тем, что повторы (образа) имеют значения, близкие к числам, кратным $L$.

В оправдание этого упущения добавим, что (наверняка) на 1-й лекции
достаточно подробно были рассмотрены все нетривиальные моменты в доказательстве.
Но и этот случай, наверняка был разобран, хоть он и не сложный но рутинный.
По воспоминаниям, у автора ещё небыло короткой интерпретации этого замечания, %
но его доказательство вытекало разными способами ---
например, через равенство периода и длины целого прообраза произвольного фактора $v\subset f_\lambda(w)$ при $\per(v)<3L(n-1)$.
Но, разбор этого простого свойства был довольно рутинным (по смутным воспоминаниям).


Следующая лемма использует более сильные ограничения на $\V${--}образы и,
возможно поэтому, проще в доказательстве чем Лемма \ref{l1}.
Так же, нужно отметить, что стрелки в ДР1 разного типа имели несколько различное назначение ---
двойная стрелка <<$\Rightarrow$>> использовалась как полноценный вывод из условий лемм,
<<$\rightarrow$>> как импликация с доп.условием.
Но это правило автором было принято
не сразу,
поэтому в первых формализациях <<$\rightarrow$>> использовалась как глобальное следствие (т.е. как <<$\Rightarrow$>>).
Доказательство Леммы \ref{l2} формализовалось первым (по видимому) т.к. это самая важная лемма для наших целей.

\hyperlink{contents}{$\upuparrows$}

\begin{lem}\label{l2}
	Пусть $L, n, k\in\mN$, $n\ge5, L\ge6(n-1)$ --- константы.
	
	Рассмотрим набор из $3n$ различных слов $\V\subset\A_n^L$ удовлетворяющий условиям \Big(обозначим $\eps=\frac{\l(\V)+\r(\V)}{L-1}$\Big):
	\begin{description}
		\item[(l2.c1)\label{l2c1}]
		$\per(v)=|v|(=L)$ для всех $v\in\V$;
		\footnote{Это условие эквивалентно условию $\pref_l(v_i)\not=\suff_l(v_i), \forall l,i\in\mN$ как в оригинальной работе.}
		\item[(l2.c2)\label{l2c2}]
		$uv\in\D_{3, n}^{\eps}\cap\T_n$ для любых различных $u, v\in\V$;
		\footnote{Напомним, что $\D_{3, n}^\eps=\D_{\{p:\ p\ge3n-3\}, n}^\eps$.}
		\item[(l2.c3)\label{l2c3}]
		$\eps+\max\{l+r: \pref_l(u)=\pref_l(v), \suff_r({\cRe v'})=\suff_r(w)\}\le\frac{L}{n-1}$
		для $\forall u\ne w\in\V\setminus\V_a$ и $v, {\cRe v'}\in\V_a$, $\forall a\in\A_n$.
		\footnote{В 25 ноябрьском оригинале не полностью скопирована/отображена функция $\max\{\dots\}$.
			Но в 24 июньском оригинале она отображена.
			Но мы несколько усилим это условие.
		}
	\end{description}

	Пусть слово $u\in\A_n^*$ и КЗ подстановка $f_*:\A_n^*\to\A_n^*$ определена правилом:
	
	$$
	\begin{array}{lcr}
		\forall a\in\A_n, u,v\in \A_n^* &f_u(av)=v_{a,i}f_{ua}(v), &\text{ где }
		i=\big(\big(|ua|_a-1\big)\bmod{3}\big)+1,
		v_{a, i}\in\V\\
		\forall w\in \A_n^* &f_w(\lambda)=\lambda &
	\end{array}	
	$$
	
	Т.е. на каждую букву $a\in\A_n$ найдётся ровно 3 (уникалных) образа $v_{a,1},v_{a,2},v_{a,3}\in\V$,
	которые используются КЗ подстановкой $f_*$ в порядке циклической очереди.
	При этом, разные буквы не имеют общих образов.

	Тогда, для любого слова $\dot{w}\in\D_{3, n}^\eps\cap\T_n$
	КЗ подстановка $f_\lambda(\dot{w})\in\D_{3, n}^\eps\cap\T_n$
\end{lem}
\begin{proof}
	Считаем, что $\p(v)>0$ (т.к. случай $\p(v)=0$ тривиален).
	
	Пусть $u$ и $v$ различные $\V${--}образы такие, что их максимальный общий префикс равен $\l(\V)$.
	Т.к. $uv$ гранично по \ref{l2c2} и $|u|=|v|=L\ge3(n-1)$.
	Тогда, используя условие \ref{l2c2} $uv\in\D_{3, n}^{\eps}$
	(ограничения на экспоненты факторов слова $uv$ с корнями $u$ и $v$), установим неравенства
	$\frac{L+\l(\V)+\eps}{L}\le\frac{n}{n-1}$.
	Аналогично установим, что $\frac{L+\r(\V)+\eps}{L}\le\frac{n}{n-1}$.
	Тогда оценим
	
	\begin{equation}\tag{l2.p0}\label{l2p0e}
	\frac{2L+\eps L}{L}
	=\frac{L+\frac{\l(\V)L}{L-1}}{L}
		+\frac{L+\frac{\r(\V)L}{L-1}}{L}
	=\frac{L+\l(\V)+\frac{\l(\V)}{L-1}}{L}
		+\frac{L+\r(\V)+\frac{\r(\V)}{L-1}}{L}
	\le\frac{L+\l(\V)+\eps}{L}
		+\frac{L+\r(\V)+\eps}{L}
	\le\frac{2n}{n-1}
	\end{equation}
	
	Откуда $\eps\le\frac{2}{n-1}$. Т.е. в лемме допустимо ограничение $\eps\in[0,\frac{2}{n-1}]$ без потери общности.
	
	Возмём произвольный фактор $v\subseteq f_\lambda(\dot{w})$.
	Тогда, для доказательства леммы достаточно доказать, что $\exp(v)\le\RT(n)$
	и, при $\per(v)\ge3n-3$, выполняется $\exp^\eps(v)\le\RT(n)$.
	БОО считаем, что $v$ нерасширяем в $f_\lambda(\dot{w})$.
	
	Докажем для всех значений корней:
	\begin{description}
		\item[(l2.p1)\label{l2p1}]
		Случай $\per(v)\ge3(n-1)L$.
		
		Тогда найдутся такие $a, b$, что $aL\le|v|\le aL+\l(\V)+\r(\V)$ и $\per(v)=b\ge3(n-1)$.
		\footnote{В оригинале забыто огрничение на $|v|$ снизу.
		в оригинальной формулировке, видимо, имеется ввиду:
		какими бы ни были корень и повтор <<целого>> прообраза $v$, всегда выполняется неравенство.
		Тогда автор ещё не использовал понятия <<корень>> и <<повтор>> фактора, а использовал только <<период>> и <<длина>> подслова,
		поэтому приходилось делать утверждения для всех периодов.
		}
		Понятно, что $a$ и $b$ --- длина корня и длина повтора целого прообраза $v$ соответственно.
		Тогда по свойству прообраза $\frac{a+\eps}{b}\le\frac{n}{n-1}$.
		\footnote{В оригинале, видимо, стрелкой <<$\to$>> обозначено просто следствие,
			когда импликация в Лемме \ref{l1} обозначалась в скобках с этой стрелкой.
		}
		Откуда получаем, требуемое для этого случая, неравенство
		\begin{equation}\tag{l2.p1}\label{l2p1e}
		\frac{|v|+\eps}{\per(v)}
		\le\frac{aL+\l(\V)+\r(\V)+\frac{\l(\V)+\r(\V)}{L-1}}{bL}
		\le\frac{aL+\eps L}{bL}
		\le\frac{n}{n-1}
		\end{equation}
		
	\end{description}
	
	{\cGr
		По правилу нашей КЗ подстановки понятно,
		что в любом факторе $f(\dot{w})$ из $3n-3$ $\V${--}образов все $\V${--}образы различны
		
		Докажем, что повторы в $v$ не содержат целого $\V${--}образа для остальных случаев т.е. для $\per(v)<3(n-1)L$.
		\footnote{В оригинале об этом замечании этот случай опущен,
			но во всех подслучаях рассматриваются только такие повторы.
		}

		{\color{teal}
			
		}

		ОП. Пусть $v_i, v_j$ --- левейшие $\V${--}образы в левом и правом повторах.
		Тогда (непустой) суффикс одного из них равен префиксу другого.
		Если у них различные смещения от левых краёв своих повторов, то получим 
		либо $\lexp(v_iv_j)\ge2$, либо $\lexp(v_jv_i)\ge2$ \contr\ \ref{l2c2}.
		Значит смещения одинаковы, а значит $v_i, v_j$ равны.
		Тогда $\per(v)$ кратен $L$, а значит $\per(v)\le(3n-4)L$,
		а значит $v_i, v_j$ входят в фактор из $3n-3$ $\V${--}образов,
		а значит не могут быть равны (противоречие).
		\Big[Т.е. $\p(v)<2L$, а значит случай,
		когда $\per(v)=2nL\ge(3n-5)L$ не интересен, т.к. в этом случае экспонента хорошая
		$$
		\frac{|v|+\eps}{\per(v)}
		=\frac{\per(v)+\p(v)+\eps}{\per(v)}
		<\frac{2nL+2L+\eps}{2nL}
		\le\frac{(n+1)L+\frac{\l+\r}{2L-2}}{nL}
		\le\frac{n+1+\frac{1}{2(L-1)(n-1)}}{n}
		=\frac{n^2-1+\frac{1}{2(L-1)}}{n(n-1)}
		<\frac{n}{n-1}
		$$
		А значит, можно считать, что $v$ пересекает только различные $\V${--}образы в $f(\dot{w})$.
		\Big]
		\footnote{Это дополнительное свойство для лучшего понимания, но мы им не воспользуемся.
		}
	}

	\begin{description}
		\item[(l2.p2)\label{l2p2}]
		Случай $2L\le\per(v)<3(n-1)L$.
		Докажем отдельно для подслучаев:
		
		\begin{description}
			\item[-]\label{l2p2.1}
			Случай, когда оба повтора в $v$ полностью лежат в некоторых $v_i, v_j\in\V$.
			
			Если $v_i\ne v_j$, то по \ref{l2c2} $v_iv_j, v_jv_i\in\D_{\ge 3n-3, n}^\eps$.
			Тогда, с учётом что $L\ge3(n-1)$ и замечания \ref{n1}, получим
			\begin{equation}\tag{l2.p2}\label{l2p2e}
			\frac{|v|+\eps}{\per(v)}
			\le\max\{\lexp_{\ge3n-3}^\eps(v_iv_j), \lexp_{\ge3n-3}^\eps(v_jv_i)\}
			\le\frac{n}{n-1}
			\footnote{В оригинале $\lexp_{\ge3n-3}^\eps$ обозначалось другим способом, но без определения.
				Но по контексту понять можно.
			}
			\end{equation}
			Если же $v_i=v_j$, то если повторы в разных позициях $v_i$,
			то, используя замечания \ref{n1} и \ref{n2}, условие \ref{l2c2} и то что $|v|\ge\per(v)\ge2L>|v_i'|$
			для $v_i'\subseteq v_i$ с этими повторами, получим
			$\exp^\eps(v)\le\exp^\eps(v_i')\le\RT(n)$.
			Если же повторы в одинаковых позициях, то, по условию нерасширяемости $v$, повтор должен содержать и сам $\V${--}образ $v_i$,
			что невозможно.
			\footnote{В оригинале случай $v_i=v_j$ не рассмотрен.
			}
			
			\item[-]\label{l2p2.2}
			Случай, когда повтор в $v$ не лежит полностью ни в каком $\V${--}образе.
			
			Тогда пусть попарно различные $v_i,v_j,v_k,v_l$ такие, что $v[1...\p(v)]\subset v_iv_j$ и
			$v[\per(v)+1...\per(v)+\p(v)]\subset v_kv_l$.
			
			Тогда $\per(v)\equiv0(\bmod{L})$ 
			(т.к. иначе, либо $\lexp(v_iv_l)\ge2$, либо $\lexp(v_kv_j)\ge2$, \contr\ \ref{l2c2}).
			
			Тогда, используя (\ref{l2p0e}) получим требуемое неравенство
			\begin{equation}\tag{l2.p3}\label{l2p3e}
			\frac{|v|+\eps}{\per(v)}
			\le\frac{\per(v)+\l(\V)+\r(\V)+\eps}{\per(v)}
			\le\frac{\per(v)+\eps L}{\per(v)}
			\le\frac{2L+\eps L}{2L}
			\le\frac{n}{n-1}
			\end{equation}
			
			\item[-]\label{l2p2.3}
			Оставшиеся случаи сведём к проверке $v$ при $L<\per(v)<2L$.
			
			Пусть попарно различные $v_i, v_j, v_k\in\V$.
			БОО считаем, что левый повтор $v$ содержится в $v_i$, а правый в $v_jv_k$.
			Выделим факторы $u_1\subset v_iv_jv_k$ и $u_2\subset v_jv_kv_i$ с точно теми же повторами, что и у $v$.
			Тогда $\per(u_1)+\per(u_2)=3L$.
			Тогда, либо $\per(u_1)<2L$, либо $\per(u_2)<2L$.
			
			БОО пусть $\per(u_1)<2L$.
			А значит, по общности нашей подстановки
			(т.е. любые перестановки из $n$ образов из $\V$ могут оказаться в образе),
			можем считать, что $v_iv_jv_k\subseteq f(\dot{w})$.
			
			Т.к. $\per(u_1)\le2L<\per(v)$ и $\p(u)\ge\p(v)$,
			то $\exp^\eps(u)\ge\exp^\eps(v)$.
			
			Т.е. для проверки $\exp^\eps(v)\le\RT(n)$ достаточно проверить $\exp^\eps(u)\le\RT(n)$
			при $\per(u)<2L$.
			Заметим, что левый повтор $u_1$
			начинается с более сильным смещением от правого края $v_i$ чем правый повтор от правого края $v_j$,
			а значит $\per(u_1)>L$.
			\footnote{Это избыточный факт, но для формальности проверено утверждение из оригинала.
			}
			
		\end{description}
		
		Пусть $v'$ --- повтор слова $v$.
		
		Для остальных случаев (т.е. при $\per(v)<2L$) установим, что 
		$v\subset v_iv_jv_k$ для некоторых различных $v_i,v_j,v_k\in\V$.
		\footnote{В оригинале это утверждение написано в форме импликации.
			Доказательство как в Лемме \ref{l1} для случая $p_m^M<2L$.
		}
		
		\item[(l2.p3)\label{l2p3}]
		$L<\per(v)<2L$.
		
		Хотябы один повтор слова одержится в некотором $\V${--}образе.
		Т.е. либо $v'\sse v_i$, либо $v'\sse v_k$.
		Докажем это от противного --- предположим, что $v'\not\subseteq v_i$ и $v'\not\subseteq v_k$.
		Понятно, что ни один повтор не содержит $v_j$ т.к. иначе,
		они пересекался бы и тогда $\lexp(v_jv_k)>2$ или $\lexp(v_iv_j)>2$ \contr\ \ref{l2c2}.
		Тогда $v'\ss v_iv_j$ и $v'\ss v_jv_k$.
		Т.к. $\per(v)>L$, то пересечение левого повтора с $v_i$ длиннее пересечения правого повтора с $v_j$,
		а значит часть суффикса в $v_i$ совпадает с некоторым префиксом в $v_k$.
		Т.е. $\lexp(v_iv_k)\ge2$ \contr\ \ref{l2c2}.
		
		БОО пусть $v'\subseteq v_i$ (т.е. левый повтор полностью лежит в $v_i$).
		
		Пусть $a$ и $b$ --- длины пересечения правого повтора с $v_j$, $v_k$ соответственно.
		Тогда $a+b=|v'|$.
		Пусть $a_1$ и $b_1$ --- максимальные длины префикса и суффикса в $v_i$,
		не пересекающиеся с левым повтором $v$.
		Тогда $a_1+b_1=|v_i|-|v'|$.
		
		Получаем $a+a_1+b+b_1=|v_i|=L$.
		Тогда, либо $a_1+a\ge3(n-1)$, либо $b_1+b\ge3(n-1)$.
		Тогда:
		\begin{description}
			\item[-]
			либо, т.к. $v_jv_i\in\D_{\ge3n-3,n}^\eps$, получим $\frac{a_1+2a+\eps}{a_1+a}\le\frac{n}{n-1}$,
			\big(т.е. $\frac{a+\eps}{a_1+a}\le\frac{n}{n-1}-1=\frac{1}{n-1}$\big),
			откуда $a\le\frac{a_1+a}{n-1}-\eps$.
			\item[-]
			либо, т.к. $v_iv_k\in\D_{\ge3n-3,n}^\eps$, аналогично получаем $b\le\frac{b_1+b}{n-1}-\eps$.
		\end{description}
		
		При этом, по \ref{l2c2} $v_jv_i\in\T_n$ получим $\frac{a_1+2a}{a_1+a}\le\frac{n}{n-1}$,
		откуда $a\le\frac{a_1+a}{n-1}$.
		Аналогично получаем для $b\le\frac{b_1+b}{n-1}$.
		
		Тогда в обоих случаях $a+b\le\frac{a_1+a+b_1+b}{n-1}-\eps=\frac{L}{n-1}-\eps$.
		
		Наконец, получаем требуемое
		\begin{equation}\tag{l2.p4}\label{l2p4e}
		\frac{|v|+\eps}{\per(v)}
		=\frac{\per(v)+a+b+\eps}{\per(v)}
		\le\frac{\per(v)+\frac{L}{n-1}-\eps+\eps}{\per(v)}
		<\frac{L+\frac{L}{n-1}}{L}
		=\frac{n}{n-1}
		\end{equation}
		
		\item[(l2.p4)\label{l2p4}]
		$\per(v)=L$.
		Если $v$ пересекает все 3 $\V${--}образа.
		Тогда оба повтора в $v$ пересекают и $v_j$.
		Тогда, пусть $l,r$ --- длины пересечения левого повтора с $v_i,v_j$ соответственно.
		Тогда, учитывая условие \ref{l2c3} получим требуемое
		\begin{equation}\tag{l2.p5}\label{l2p5e}
		\frac{|v|+\eps}{\per(v)}
		=\frac{\per(v)+l+r+\eps}{L}
		\le
			1
			+\max_{
				\begin{array}{c}
				\scriptstyle u\ne w\in\V\setminus\V_a,\\
				\scriptstyle v,v'\in\V_a, a\in\A_n
				\end{array}
			}
				\bigg\{
				\frac{l+r}{L}: 
				\begin{array}{lcr}
				\pref_l(u) & = & \pref_l(v),\\
				\suff_r(v') & = & \suff_r(w)
				\end{array}
				\bigg\}
			+\frac{\eps}{L}
		\le\frac{L+\frac{L}{n-1}}{L}
		=\frac{n}{n-1}
		\end{equation}
		
		Если же $v\sse v_iv_j$ или $v\sse v_jv_k$, то по \ref{l2c2} получаем требуемое $v\in\D_{\ge3n-3, n}^\eps\cap\T_n$.
		\footnote{В оригинале этот случай не описан.
		}
		
		\item[(l2.p5)\label{l2p5}]
		$\per(v)<L$.
		Тогда, сразу получаем требуемое,
		либо $v\ss v_iv_j\in\D_{\ge3n-3, n}^\eps\cap\T_n$,
		либо $v\ss v_jv_k\in\D_{\ge3n-3, n}^\eps\cap\T_n$.
		Т.к. иначе, по аналогии со случаем \ref{l1p13} в Лемме \ref{l1},
		получим противоречие с \ref{l2c1}.
		
		А значит $v\in\D_{\ge3n-3, n}^\eps\cap\T_n$ т.е. $v$ является ГС
		и, при $\per(v)\ge3n-3$, имеет $\exp^\eps(v)\le\RT(n)$.
	\end{description}
	В результате, получаем, что любой фактор $v\subseteq f(\dot{w})$ имеет $\exp(v)\le\RT(n)$
	и, при $\per(v)\ge3n-3$, имеет $\exp^\eps(v)\le\RT(n)$.
	
	Откуда и получаем требуемое $f(\dot{w})\D_{\ge3n-3, n}^\eps\cap\T_n$.
\end{proof}

\hyperlink{contents}{$\upuparrows$}

\subsubsection{Замечания и дополнения}

Для лучшего понимания поясним, что $\eps$ в \ref{l2c2} характеризует некоторую свободу (резерв
\footnote{Или <<зазор>> как было предложено нашим НР на лекции автора})
в экспоненте для применения нашей КЗ подстановки.
Точнее, определяет минимальную требуемую добавку
к экспоненте любого фактора с периодом не менее $3n-3$, чтобы превысить $\RT(n)$.
Значение этой добавки зависит от длины периода фактора и $\eps$.
В нашем случае (для леммы \ref{l2}) достаточно, чтобы $\eps$ было константой.

\begin{note}
	Если лемма \ref{l2} верна при $\eps=C$ в \ref{l2c2}, то она верна и при любом $\eps\in[0, C]$.
\end{note}

Т.е., в силу монотонности $\eps_p$ от периода $p$, достаточно взять максимальное занчение $\eps=\frac{2}{n-1}$.
Т.е. если для $\eps=C$ лемма работает, то она работает и для любого $\eps\in[0,C]$.

Заметьте, что условие \ref{l2c3} может быть легко ослаблено,
где максимум достаточно выбрать только среди троек различных $\V${--}образов,
не являющихся образами общей буквы.
Это возможно т.к. это условие необходимо только в пункте \ref{l2p5},
когда оцениваются образы 3-х соседних букв, которые всегда различны при $n\ge4$.
Условие \ref{l2c3} 
При достаточно большом отношении $L/n$ можно добиться достаточно малых $\l(\V)$ и $\r(\V)$,
благодаря чему можно добиться выполнения условия \ref{l2c3}.

%
%
%

\hyperlink{contents}{$\upuparrows$}

\subsubsection{Снятие ограничения на длину $\V${--}образов с помощью условий \ref{l1c2} и $|\V|\ge3n$ при $n\ge5$}
Снятие ограничения на длину $\V${--}образа не обязательно
\footnote{Стоит отметить, что первым это сделал наш НР,
	в передоказанных (т.е. отредактированных) основных авторских Леммах из ДР1.}
т.к., очевидно, что проверить длину можно за $O(6(n-1))$ т.е. за полином.
Но для красоты решения добавим это свойство, которое вытекает даже из одного условия \ref{l1c2} Леммы \ref{l1}.

{\cGr
Ниже решение (в лоб) не перепровереное автором. В следующей версии этого манускрипта планируется переделать.
Идея грубого решения --- по принципу Дирихле из $n$ вариантов бинарных последовательностей (т.е. $b$ и $c$)
найдётся пара с общей длиной (как минимум) порядка $bc${--}префикса log(n) с округлением вверх до целого.
Но, так же по принципу Дирихле, либо достаточно много префиксов с более длинным $c$ кодом, либо слишком много с $b$ кодом,
но это требует более детального разбора.
}

Докажем, что $L\ge6(n-1)$.
По принципу Дирихле, среди $3n$ $\V${--}образов над $n$ буквами,
найдутся 2 образа $v_1,v_2\in\V$, с одинаковой первой буквой $a_1$.
Тогда по условию \ref{l1c2} $v_1v_2\in\T_n$, а значит $|v_1|$ не менее $n-1$.

Если предположить, что длина ровно $n-1$, то наши слова состоят только из разных букв.
Тогда есть только 2 буквы. продолжающие последовательность ГС --- либо $a_1$, либо недостающая --- пусть $a_n$.
Т.е. каждый $\V${--}образ начинается с одной из 2-х возможных букв.

Возмём большинство (как минимум $\lceil(3n-1)/2\rceil$) $\V${--}образов с общей первой буквой.
\begin{itemize}
	\item
	Если это буква $a_1$, то эти слова начинаются с префикса $a_1a_n$ (т.е. они начинаются с $b$ в $bc${--}коде)
	Тогда, существует хотябы 2 $\V${--}образа с общим префиксом длины 2.
	А значит, по условию граничности пары этих слов, их длина $L$ не менее $2(n-1)$.
	Тогда как минимум у $\lceil\lceil(3n-1)/2\rceil/2\rceil$ $\V${--}образов (т.е. минимум у 4-х при $n\ge5$)
	\begin{itemize}
		\item
		с $b$ или $c$ кодом в префиксе (определение $bc${--}кодов здесь \ref{l3_def}).
		
		Тогда, как минимум у 4, общий префикс длины 4. ... т.е. $L\ge4(n-1)$
		Аналогично на след-м уровне получим $L\ge6(n-1)$ (хоть при $b$ хоть при $c$ коде)
	\end{itemize}
	\item 
	Если же это буква $a_n$, то следующий целый $bc${--}код начинается со второй буквы.
	\\
	Тогда как минимум у $\lceil\lceil(3n-1)/2\rceil/2\rceil$ $\V${--}образов (т.е. минимум у 4-х при $n\ge5$)
	\begin{itemize}
		\item
		с $b$ кодом.
		
		Тогда либо, как минимум у 5 (с $b$ кодом), общий префикс длины 3. ... т.е. $L\ge3(n-1)$.
		
		Либо, как минимум у 3 (с $c$ кодом), общий префикс длины 3. ... т.е. $L\ge4(n-1)$.
		
		Тогда на след-м уровне,
		либо у 3 ... $L\ge5(n-1)$,
		либо у 2 ... $L\ge6(n-1)$.
		
		Остаётся 1-й вариант --- у 2 ... $L\ge7(n-1)$.
		\item
		с $c$ кодом.
		
		Тогда, как минимум у 4, общий префикс длины 4. ... т.е. $L\ge4(n-1)$
		
		Аналогично на след-м уровне получим $L\ge6(n-1)$ (хоть при $b$ хоть при $c$ коде)
	\end{itemize}
\end{itemize}

\hyperlink{contents}{$\upuparrows$}

\subsection{Лемма \ref{l3}. Циркулярные/кольцевые $D_{3,n}^\eps${--}ГС ($D_{3,n}^\eps${--}ЦГС)}

\subsubsection{Дополнительные определения и замечания}\label{l3_def}

Здесь мы введём вспомогательные конструкции, которые были использованы в ДР1 и лекциях.
И опишем их несложные фундаментальные свойства,
чего не было сделано ни в ДР1 ни на лекциях, кроме некоторых по необходимости на лекциях (когда требовалось).
Автор считал это очевидными свойствами для специалистов в комбинаторике слов.

\begin{defn}
	$l${--}суффикс[$l${--}префикс,$l${--}фактор] слова это суффикс[префикс,фактор] длины $l$ этого слова.
\end{defn}

$A_n$ --- слово длины $n$ содержащее все буквы из $\A_n$, либо $n-1$ букву из $\A_n$,
где одинаковые буквы стоят по краям (в первой и последней позициях) в $A_n$.
Для простоты, можно считать, что буквы в $A_n$ упорядочены в лексикографическом порядке.

\begin{defn}
	Слова, у которых все циклические сдвиги это $D_{r_p,n}^{\eps_p}${--}ГС, назовём
	циклическими/циркулярными/\\кольцевыми $D_{r_p,n}^{\eps_p}${--}ГС и обозначим как $D_{r_p,n}^{\eps_p}${--}ЦГС.
\end{defn}

Введём аналог кода Пансьё, где <<\mi>> и <<\pl>> совпадают с <<0>> в коде Пансьё, а <<\ze>> с <<1>>:

Под расстоянием между буквами подразумевается разница позиций этих букв (или расстояние между геом.центрами).

\begin{defn}\ 
	Коды <<\ze>>, <<\mi>>, <<\pl>> уточняют код Пансьё.
	В позиции стоит знак в зависимости от ближайшей соответствующей буквы в слове слева:
	\\
	<<\mi>> если в соответствующей позиции слова стоит буква, у которой ближайшая слева такая же находится на расстоянии $n{-}1$;
	\\
	<<\pl>> если в соответствующей позиции слова стоит буква, у которой ближайшая слева такая же находится на расстоянии $n{+}1$;
	\\
	<<\ze>> если в соответствующей позиции слова стоит буква, у которой ближайшая слева такая же находится на расстоянии $n$.
\end{defn}

Далее, под кодом подразумевается уточнёный код Пансьё.
Задание для читателя доказать, что в коде БГС в обе стороны перед <<\pl>> стоит <<\mi>>, а после <<\mi>> стоит <<\pl>>.
Там же докажите, что перед <<\ze>> стоит <<\pl>>, а после <<\ze>> стоит <<\mi>>.

\begin{defn}\ 
	
	[Право] $\rbc${--}кодом назовём последовательность букв <<$\rb$>> и <<$\rc$>>.
	
	[Центро] $\nbc${--}кодом назовём последовательность букв <<$\nb$>> и <<$\nc$>>.
	
	[Лево] $\lbc${--}кодами назовём последовательность букв <<$\lb$>> и <<$\lc$>>.
	%
\end{defn}

\begin{defn}
	Подстановка $\phi$ над $\rbc$, $\nbc$ и $\lbc${--} кодами определяется:
	
$\phi(\rb)=\mi\pl$,
$\phi(\rc)=\mi\pl\ze$. Назовём эти подстановки образами $\rbc${--}кода или $\rbc${--}образами;

$\phi(\nb)=\pl\mi$,
$\phi(\nc)=\pl\ze\mi$. Назовём эти подстановки образами $\nbc${--}кода или $\nbc${--}образами;

$\phi(\lb)=\mi\pl$,
$\phi(\lc)=\ze\mi\pl$. Назовём эти подстановки образами $\lbc${--}кода или $\lbc${--}образами.
\end{defn}

Примечание:
Будем записывать $(\phi(w))^k$ как $\phi(w)^k$.

\begin{note}\label{l3:nt:bc:open}
	$\phi(w^k)=\phi(w)^k$.
\end{note}

\begin{defn}
	Код назовём замощаемым $\rbc$[$\nbc$,$\lbc$]{--}образами, если он представим конкатенацией $\rbc$[$\nbc$,$\lbc$]{--}образов.
\end{defn}

\begin{defn}\ 
	$bc${--}кодом назовём код, если он замощаем $\rbc$, $\nbc$ или $\lbc${--}образами (хотябы одним из них).
	
	
	
	
	Целым $bc${--}кодом назовём замощаемый $\rbc$ или $\lbc${--}образами код.
	
	Нецелым $bc${--}кодом назовём код, замощаемый только $\nbc${--}образами.
	
	Целым $bc${--}сдвигом назовём циклический сдвиг кода, являющийся целым $bc${--}кодом.
	
	Нецелым $bc${--}сдвигом назовём циклический сдвиг кода, являющийся нецелым $bc${--}кодом.
\end{defn}

{\bf
Чтобы не путаться, что $bc${--}код (сдвиг и др.) состоит из кодов Пансьё или расширенных кодов, а не из букв <<$b$>> и <<$c$>>
достаточно смотреть на ниличие знака над буквами.
Т.е. у названия кода (не образа) с буквами <<$b$>> и <<$c$>> всегда есть знак вектора или черты
(т.к. <<$b$>> и <<$c$>> требуют конкретизации для подстановки $\phi$ ).
А для расширенного кода Пансьё не требуются такие знаки.
}

Заметьте, что целые $\rbc$ и $\lbc${--}коды покрывают не все {\cRe целые коды}
т.к. код начинающийся и заканчивающийся на \ze\ не может быть замощён ни $\rbc$ ни $\lbc${--}образами.
Но некоторые такие коды (с началом и концом \ze) могут замощаться кодами \{\ze\mi\pl, \mi\pl, \mi\pl\ze\},
при чём все, кроме \ze(\mi\pl\ze)*.
Но такие коды нам не потребуются.

\begin{note}\label{BcC-NotBcC}
	Сдвиг, хотябы в одну сторону (влево или вправо) целого $bc${--}кода на 1 букву делает его не целым $bc${--}кодом.
	
	Сдвиг, в любую сторону не целого $bc${--}кода на 1 букву делает его целым $bc${--}кодом.
	
	Но для любого целого $bc${--}кода либо 2-х либо 3-х буквенный сдвиг (в любую сторону) является целым $bc${--}кодом.
\end{note}

\begin{note}\label{l3:nt:bc:sh}
	\begin{description}\ 
		\item[(1)]
		Любой циклический сдвиг $bc${--}кода является $bc${--}сдвигом.
		\item[(2)\label{2sh}]
		Среди 2-х соседних циклических сдвигов $bc${--}кода найдётся хотябы один целый $bc${--}код.
		\item[(3)\label{3sh}]
		Среди 3-х соседних циклических сдвигов $bc${--}кода найдутся замощаемые образами,
		для каждого из $\lbc$, $\nbc$ и $\rbc${--}образов.
\end{description}
\end{note}

Задание для читателя доказать
\begin{note}
	$bc${--}код гарантирует граничность факторов с повторами менее $3$, а так же с периодами менее $3n-3$.
\end{note}
Это удобное замечание для множества $D_{3,n}^\eps${--}ГС (т.е. из $\D_{\ge3n-3, n}^\eps\cap\T_n$).

\paragraph{Предпосылки.}
{
	Т.к. $n${--}префикс (если он есть) любого ГС может иметь несколько возможных кодов,
	то введём дополнительные параметры для однозначности.
	
	Пусть $w$ --- ГС длины более $n$, а $u$ --- некоторый код этого слова такой же длины (т.е. $|u|=|w|$).
	Существует единственное $n${--}слово $A_n$ такое, что действие (т.е. перестановка) $n${--}префикса
	в коде $u$ на $A_n$ однозначно определяет $n${--}префикс в $w$.
	
	Такое $n${--}слово $A_n$ назовём предпосылкой слова $w$ с кодом $u$.
	
Предпосылка $A_n$ для кода --- предшествующее коду $n${--}слово,
по которому кодом определяется $n$ букв в $(n{+}1)${--}префиксе ГС.
Т.е. первые $n-1$ букв (всегда различны) в ГС однозначно определяются этой предпосылкой и $n${--}префиксом кода.
}

Предпосылка может быть 2 видов:
\begin{description}
\item[(\Anz)\label{An0}]
Все $n$ букв различны.
Критерием этого случая в коде Пансьё --- чётность числа <<0>> между предпосылкой и первой (слева) <<1>> в коде;
\item[(\Anm)\label{An-}]
Только 2 буквы совпадают и только по краям $n${--}слова.
Критерием этого случая в коде Пансьё --- нечётность числа <<0>> между предпосылкой и первой (слева) <<1>> в коде.
\end{description}

\begin{note}\label{An-C-NotC}
	\Anm{--}предпосылки применяются если и только, если код начинается с <<\pl>>.
	
	Как следствие, \Anz{--}предпосылки применяются если и только, если код начинается с <<\mi>> или <<\ze>>.
	
\end{note}
\begin{proof}
Перед <<\pl>> необходимо применять \Anm{--}предпосылку
(даже, любой $n${--}фактор перед <<\pl>> должен иметь вид \Anm)
т.к. в коде перед <<\pl>> всегда <<\mi>>, а это последняя буква предпосылки.
А значит последняя буква в \Anm{--}предпосылке совпадает с первой.

А первый <<\pl>> в коде необходим для этой предпосылки т.к. только <<\pl>> <<достаёт>> недостающую букву перед предпосылкой.
\end{proof}

Как следствие

\begin{note}\label{An-BcC-NotBcC}
	\Anz{--}предпосылки применяются только к $\rbc$ и $\lbc${--}кодам.
	\Anm{--}предпосылки только к $\nbc${--}кодам.
	Т.е у целого $bc${--}кода ($bc${--}сдвига) только \Anz{--}предпосылки,
	а у не целого $bc${--}кода ($bc${--}сдвига) только \Anm{--}предпосылки.
\end{note}

Первые $n-1$ букв в ГС определяются неоднозначным кодом Пансьё т.к. предпосылки могут быть различными для построения ГС.
Но по предпосылке они определены однозначно.

Т.о. длину Пансьё кода слова можно сравнять с длиной самого слова, однозначно определив по предпосылке.

\begin{defn}
Код Пансьё $u$ будем называть {\cRe$A_n${--}кодом} ГС $w$, если $A_n$ это предпосылка по которой строится $w$ по коду $u$.
Т.е. длина {\cRe$A_n${--}кода} ГС совпадает с длиной самого ГС.
\end{defn}

\begin{defn}
Подстановка $f_{A_n}:\{\mi,\pl,\ze\}^*\to \A_n^*$ переводит {\cRe$A_n${--}код} в слово по предпосылке $A_n$
\end{defn}

\begin{defn}
	Пусть $m\in\mN_0, l,k\in\mN$ и $u'$ --- код Пансьё, где $m,l\le|u'|$.
	Пусть $u'_1u'_2=u'$ такие, что $|u'_1|=m$,
	То обозначим $\w_{u}(m,l)=\pref_l\big(f_{\cRe A'_n}\big(u'_2u'_1\big)\big)$
	\footnote{В ДР1 ошибочно взят суффикс вместо префикса.}.
	Т.е. $\w_{u}(m,l)$ --- фактор в $f_{A_n}(u'_1u'_2u'_1)$ длины $l$ с циклическим сдвигом на $m$ влево,
	где $A'_n$ --- $n${--}суффикс слова $A_n.f_{A_n}(u'_1)$.
	
\end{defn}

Заметьте, что можно задать смещение $m>|\phi(u)|$,
это будет эквивалентно смещению всего кода $\phi(u^k)$ влево на $m$.
Т.е. $\w_{u^k}(m,l)\sim\w_{u^k}(m\pmod{|\phi(u)|},l)$

\begin{sgn}
	Фактор слова $w$ в позициях(индексы с нуля) от $i$ до $j$ (включительно)
	обозначим $w[i:j+1]$ как в языке программирования Python.
	
	Так же отрицательным индексом обозначается отступ от конца слова как в Python.
	Т.е. когда аргумент отрицателен, но его модуль не превосходит $|w|$, то к аргументу добавляется $|w|$.
	
	По умолчанию пустой аргумент в начале (перед <<$:$>>) означает 0, а в конце (после <<$:$>>) длину слова.

	Т.е. $w[:-n]$ означает префикс в $w$ длины $|w|-n$ (т.е. занимает позиции $0, ..., |w|-n-1$).
	
	В частности, $w[:]=w$, $w[-n:]$ --- суффикс длины $n$.
\end{sgn}

\begin{defn}
	$bc${--}корнем слова $w$ назовём корень его $bc${--}кода (если он есть т.е. определение не только для ГС).
\end{defn}

\begin{defn}
	Пусть $k\in\mN$.
	Слово $w$ с $bc${--}кодом назовём $k${-}$bc${--}корневым,
	если $bc${--}код слова $w$ является $k$-й степенью его $bc${--}корня,
	и $n${--}суффикс слова $w$ равен предпосылке его $bc${--}кода.
	При этом, $k$ минимально.
\end{defn}

\begin{defn}
	Слово $w$ с $bc${--}кодом назовём $bc${--}корневым,
	если при некотором $k\in\mN$ оно является $k${-}$bc${--}корневым.
\end{defn}

\begin{note}\label{l3:nt_k-bc-all}
	Любой сдвиг $k${-}$bc${--}корневого слова так же является $k${-}$bc${--}корневым словом.
\end{note}

{\cGr
Избыточное определение, но для тренировки понимания
\begin{defn}
	Пусть $n\in\mN_{\ge2}$, $u\in\{b,c\}^*$, $|\phi(u)|\ge n$.
	Если $w$ --- $D_{r_p,n}^{\eps_p}${--}ГС, чей код представим в виде $\phi(u)^k$, при некотором $k\in\mN_{\ge2}$.
	Т.е. $w=f_{A_n}(\phi(u)^k)$ при некоторой предпосылке $A_n$.
	При этом, $A_n\in\SUFF(w)$.
	То такое $D_{r_p,n}^{\eps_p}${--}ГС будем называть $k${-}$bc${--}корневым.
\end{defn}
}

\paragraph{Свойства кодов:}\ 

(1) Если код начинается с <<\mi>> или <<\ze>>, то предпосылка состоит из $n$ различных букв.

(1.1) ... и заканчивается на <<\pl>> или <<\ze>>, то код называется целым.

(1.1.1) ... начинается или заканчивается не на <<\ze>>, то код замощается $\rbc$ или $\lbc${--}образами.

(1.1.1.1) ... не начинается и не заканчивается на <<\ze>>, то код замощается и $\rbc$ и $\lbc${--}образами.

(1.2) ... и заканчивается на <<\mi>>, то код не циклический.

(2) Если код начинается с <<\pl>>, то предпосылка состоит из $n-1$ различной буквы (т.е. \Anm{--}предпосылка).
Такой код не замощается ни $\rbc$ ни $\lbc${--}образами т.е. не является целым $bc${--}кодом.

\paragraph{Свойства предпосылок:}\ 

(1) $\phi(u)$ действует на предпосылку как перестановка, при $u\in\{b,c\}^*, |\phi(u)|\ge n$.
Т.е. $n${--}суффикс слова $f_{A_n}(\phi(u))$ получается из $A_n$ некоторой перестановкой $\pi$.
В частности $n${--}суффикс слова $f_{A_n}(\phi(u)^k)$ получается из $A_n$ перестановкой $\pi^k$.

\begin{note}\label{l3:nt_bcSuf}
	Пусть $w$ --- $k${-}$bc${--}корневое $D_{r_p,n}^{\eps_p}${--}ГС над $\A_n$, при некоторых $n,k\ge3$ и неубывающем $\eps_p\ge0$.
	$u$ --- $bc${--}корень слова $w$.
	Тогда $A_n\not\in\SUFF(f_{A_n}(\phi(u)^{k'}))$ при любых $0<k'<k$
	т.к. иначе $w$ содержит квадрат, что нарушает условие граничности слова $w$ над $n$ буквенным алфавитом
	(а так же, нарушает минимальность $k$ в определении $k${-}$bc${--}корневого слова $w$).
\end{note}

\begin{note}\label{l3:nt_bcTWfor1,2rep}
	Пусть $w$ --- $D_{3,n}^\eps${--}ГС над $\A_n$, при некоторых $n,k\ge3, \eps\ge0$.
	И код слова $w$ является $bc${--}кодом.
	Тогда любой фактор $v$ слова $w$ при $\per(v)\le3n-3$ имеет экспоненту $\exp(v)\le\RT(n)$
\end{note}

\begin{note}\label{l3:nt:bc:fctr}
	Пусть $w$ слово, построенное по $bc${--}коду, и его фактор $v$ имеет $\exp(v)>1$, то существуют:
	
	\begin{description}
		\item[(1)\label{sh2}]
		и целые и нецелые $bc${--}сдвиги, содержащие (целый) фактор $v$ при любом $n\ge2$.
		\item[(2)\label{sh3}]
		и $\lbc$ и $\nbc$ и $\rbc${--}сдвиги, содержащие (целый) фактор $v$ при любом $n\ge3$;
\end{description}
\end{note}
\begin{proof}
	Т.к. $\exp(v)>1$, то в $v$ есть повторы.
	Между любыми 2-мя повторами не меньше $n-1$ букв, т.к. слово $w$ строится $bc${--}кодом.
	Любой сдвиг $bc${--}кода --- $bc${--}код.
	Представив $w$ в виде кольца, где с противоположной стороны от $v$ есть
	ещё фактор с $n-1$ буквой с $n$ соседними местами для разреза (вне $v$) кольца.
	
	Тогда, по замечанию \ref{l3:nt:bc:sh}.\ref{2sh} при $n\ge2$
	выбираем нужный из 2-х допустимых соседних сдвигов (т.е. не разрезая $v$) для любого (т.е. целого или нецелого) $bc${--}сдвига.
	
	А по замечанию \ref{l3:nt:bc:sh}.\ref{3sh} при $n\ge3$,
	выбираем нужный из 3-х допустимых сдвигов для любого $bc${--}сдвига из \ref{sh3}.
\end{proof}


\begin{note}[Монотонность]
	Допустимая длина повтора в $D_{r_p,n}^{\eps_p}${--}ГС неувеличивается с уменьшением периода $p$.
	Но при условии неуменьшения $\eps_p$ (т.е. неувеличения $\eps_p$ при уменьшении $p$).
\end{note}

\hyperlink{contents}{$\upuparrows$}

\paragraph{\large Неформальное предисловие.}

Как обычно, автор доказывал устно ключевые этапы доказательства (условия и промежуточные выводы).
Но когда доходило до формализации доказательств этих этапов, из-за спешки автор не замечал мелких (например арифметических) ошибок.
Поэтому, здесь в рамках обведены ключевые этапы (утверждения, условия и допущения) из ДР1,
а остальное это рассуждения как к ним прийти, ошибки и недочёты которых несложно (по мнению автора) исправляются.

Как оказалось, в ДР1 обобщенние доказательство для $\eps$ оказалось несколько небрежным,
но оно тривиально исправляется, и
идея обобщения видна и может быть использована (после исправления арифм-й ошибки) для более общего случая $\eps\ge0$.
О возможности доказать для всех $\eps\ge0$ автор говорил своему НР,
что такое возможно за счёт импликации --- формулировка леммы написана в форме импликации,
и, если условие (т.е. предпосылка в импликации) не выполняется при слишком больших $\eps$,
то и в следствии нет смысла доказывать верность.
Проще говоря, все основные леммы (\ref{l1},\ref{l2},\ref{l3}) были сформулированы по инструкциям автора (с учётом лекций и дополнительных обсуждений <<мимоходом>>),
кроме снятия ограничений снизу на длину $L$ в Лемме \ref{l2}.
Как мы увидим в доказательстве следующей леммы, обобщение легко вытекает из доказательства
(с простыми исправлениями арифметических недочётов) предложенных в ДР1.

Во всей работе по умолчанию рассматриваются только ГС бесконечно расширяемые в обе стороны.
Поэтому, по умолчанию подразумеваются только такие случаи.

Изначально предыдущие леммы (\ref{l1}, \ref{l2}) разработаны автором для более общего случая,
когда длины $\V${--}образов не константны
(формального доказательства автор не писал, а только проверял устно все нетривиальные случаи).
Для этого случая была придумана схема доказательства (но не формализована) для случая,
когда упорядоченные длины слов в наборе могли быть ограничены полиномом порядка $\Theta(n^{p})$
(т.е. сверху и снизу), тогда достаточно было чтобы в наборе было порядка $\Theta(n^{p+1})$ слов.
Это нетрудно понять, если рассмотреть худший случай, когда между длиннейшим $\V${--}образом
стоит $n+k-1$ кратчайших $\V${--}образов для Леммы \ref{l1},
или $3(n-1)$ кратчайших $\V${--}образов для Леммы \ref{l2}.
В частности в Лемме \ref{l1} длины слов $L\in\Theta(n^{0})$,
а количество слов $n+k\in\Theta(n^{1})$ при константном $k$.
Но с идеей с циклическими словами, необходимость в более общем случае отпала.


В ДР1 некорректное обозначение степени слов $w_1$ --
$w^{k|w|}$ имеется ввиду префикс длины $k|w|$ бесконечной копии $w^\infty$
(но об этом сказано в обозначениях в начале 3-й части),
т.е. в корректной интерпретации подразумевается $w^k$.

В ДР1 усиливается разбор для $\D_{3,n}^\eps${--}ГС (с арифметическими недочётами, но не ломающими общий ход доказательства),
где вместо $\eps$ использовалось $lr$ (в программировании переменые можно обозначать словами).

Но в ДР1 есть 3 недочёта в формулировке:
1) Усиление неявно обозначено штрихом в неравенстве с $\lexp'$ подразумевая $\lexp^\eps_{\ge3n-3}(w)\le\RT(n)$.
Понять, что это значит можно только по контексту (явного описания автор не нашёл) ---
по формулировке $\D_{\ge3n-3,n}^\eps$ предыдущих лемм (с замечанием \ref{n1}) и в доказательстве данной леммы --- когда используется переменная <<$lr$>>.
2) Не уточняется диапазон значений $\eps$,
но он тривиально (с очевидными поправками) вытекает при разборе всех ключевых утверждений в ДР1,
что мы выясним в авторском разборе чернового доказательства данной леммы об $\D_{3,n}^\eps${--}ЦГС.
3) Не добавлено в формулировку, но сказанное на лекции, условие целого смещения кода Пансьё.
Но это условие можно опустить, добавив второй(полностью дополняющий) вариант суффикса длины $n$
ровно с $n-1$ различной буквой (что и планировал автор решения).

Здесь мы выясним (докажем), что метод доказательства в ДР1 так же работает для всех $\eps\ge0$.
Даже, опустим ограничение снизу до $-\frac{n-2}{n-1}$ (и до $-1$).


Здесь переменная повтора $p$ из ДР1 заменена на $r$, а $lr$ на $\eps$.

На лекциях доказывалось для случая без $\eps$ (т.е. когда $\eps=0$),
но $\eps$ не усложняет доказательство вхождения $l'+2r$ целиком в 3 $bc${--}корня
т.к. увеличение $\eps$ только сокращает длину допустимого повтора $r$.
Что только облегчает доказательство рассматриваемых неравенств.

Так же на 2-й лекции уточнялось, что $c$ в $bc${--}коде может быть и \mi\pl\ze\ и \ze\mi\pl\ 
(благодаря вопросу от Самсонова А.В.). 
Но не уточнялось (на сколько помнит автор), что при сдвигах допускается перезамощаемость $u_2u_1$ альтернативным $c$.
Хотя, это очевидно и подразумевалось т.к. в обоих случаях используется \Anz{--}предпосылка.
Поэтому, используя замечание \ref{l3:nt:bc:fctr}, можно доказывать сразу в нескольких режимах --
для замощаемости и $\rbc$, и $\lbc$, и в общем случае $bc${--}кодом.

Так же, нужно дополнить, что в условиях леммы \ref{l3} из целостности кода следует и целостность $bc${--}кода.
Т.к., если слово ЦГС, то его код (и код любого его сдвига) не может (одновременно) начинаться и заканчиваться на \ze\
(легко понять ОП).

\hyperlink{contents}{$\upuparrows$}

\subsubsection{Лемма \ref{l3}}

В доказательстве леммы основной текст (подразумеваемый в ДР1 с исправлениями) написан обычным чёрным текстом.
Серым написаны дополнения и пояснения,
а так же альтернативные доказательства и усиления леммы, замеченые автором в процессе разбора ДР1.
Серый текст можно игнорировать (особенно в формулах) при чтении основного доказательства (написанного чёрным цветом).
\pmb{Таким шрифтом} выделен текст, написанный в доказательстве Леммы 3 ДР1 (переформулированный), а обычным пояснения.
{\cRe Красным цветом} выделены исправления ошибок в жирном шрифте
(все они легко поправимы и не нарушают предложенный в ДР1 способ доказательства) и уточнения.
Рамками выделены ключевые условия и выводы в доказательстве из ДР1.

\begin{lem}\label{l3}
	Пусть $n, k\in\mN$, $n\ge5, k\ge3$, $u$ --- $bc${--}код и $\eps\in[0,\frac{2}{n-1}]$
	\footnote{Ограничения на $\eps$ (т.е. на $lr$) в ДР1 не прописано,
		но по условию Леммы \ref{l2} это необходимый и достаточный минимум.
		Это ограничение, по случайности или нет, совпало с вытекающим в конце (при максимальном использовании остальных условий),
		но в <<сером тексте>> мы это ограничение расширим.
		Так же, в ДР1 подразумевается целостность $bc${--}кода.
		Здесь мы неявно это обусловим \Anz{--}предпосылкой \ref{l3c3}.
	}.
	{\color{gray}Даже при $\eps\ge\frac{2-n}{n-1}$, а при чуть более глубоком наблюдении и при $\eps\ge-1$.
	Более того, можно задать монотонно неубывающую последовательность $\eps_p$ для любого $p\in\mN$,
	где факторы с более длинным периодом $p$ имеют
	не меньший резерв $\eps_p$ чем факторы с меньшим периодом (это ещё одно возможное обобщение лемм \ref{l2} и \ref{l3}).}
	И пусть $A_n$ --- \Anz{--}предпосылка
	\footnote{Это условие эквивалентно условию целых сдвигов, что автор и уточнил на 2-й лекции благодаря вопросу от Самсонова А.В.
	}
	{\color{gray}или \Anm{--}предпосылка}.
	Тогда, если выполняются:
	\begin{description}
		\item[(l3.c1)\label{l3c1}]
		$|\phi(u)|\ge n$;
		\item[(l3.c2)\label{l3c2}]
		$f_{A_n}(\phi((u)^k))$ имеет суффикс $A_n$
		(или, что то же самое, $n${--}суффикс равен предпосылке $A_n$);
		\item[(l3.c3)\label{l3c3}]
		$f_{A_n}(\phi((u)^k))\in\D_{\ge3n-3,n}^\eps\cap\T_n$.
	\end{description}
	То для любого $bc${--}сдвига $u_2u_1$
	\footnote{
		На лекции разбирались целые сдвиги. Здесь же (в сером тексте) мы рассматриваем и остальные $bc${--}сдвиги.
		Все $bc${--}сдвиги это все циклические сдвиги $bc${--}кода.
	}
	$bc${--}кода $u$, где $u_1u_2=u$, выполняются:
	\begin{description}
		\item[(l3.f1)\label{l3f1}]
		$f_{A_n}(\phi((u_2u_1)^k))$ имеет суффикс $A_n$ {\cGr(при чём не зависимо от \ref{l3c3})};
		\item[(l3.f2)\label{l3f2}]
		$f_{A_n}(\phi((u_2u_1)^k))\in\D_{\ge3n-3,n}^\eps\cap\T_n$.
		\footnote{Для лучшего понимания поясним, что $\eps$ характеризует минимальную требуемую добавку (<<свободу>>)
			к экспоненте любого фактора с периодом не менее $3n-3$, чтобы превысить $\RT(n)$
			(т.е. максимальную/достаточную <<свободу>>, чтобы НЕ превысить $\RT(n)$).
			Значение этой добавки зависит от длины периода фактора и $\eps$.
			В нашем случае (для леммы \ref{l2}) достаточно, чтобы $\eps$ было константой.
			Т.е., в силу монотонности $\eps_p$ от периода $p$ и линейности всех неравенств от $\eps$ (используемых в доказательстве), достаточно взять максимальное занчение $\eps=\frac{2}{n-1}$
			(с учётом разбора этой леммы для $\eps=0$ на лекциях).
			Т.е. если для $\eps=0$ и $\eps=C$ лемма работает, то она работает и для любого $\eps\in[0,C]$.
		}
	\end{description}
\end{lem}
\begin{proof}
	
	{\cGr
		Для случая, когда факторы имеют период меньше $3n-3$
		можно воспользоваться замечанием \ref{l3:nt_bcTWfor1,2rep} для этого случая,
		а можно и этим же доказательством при $\eps=0$
		\footnote{На лекции разбирался этот случай}.
		Поэтому, для $\eps>0$ считаем что период контрпримеров не меньше $3n-3$.
		
		Когда мы сокращаем длину контрпримера на 1, вместе с этим сокращается и повтор на 1,
		но период остаётся тем же.
		Но, когда мы берём пример с теми же повторами с противоположной стороны <<кольца>>,
		то период меняется при неизменной длине повтора,
		а с ним может измениться и $\eps$ в условии усиленной граничности.
		
		Обозначим $u'=\phi(u)$.
		
		Заметьте, что при произвольном циклическом $bc${--}сдвиге $u_2u_1$ коды $u_1$ и $u_2$ могут не быть $bc${--}кодами.
		Поэтому, для них (по отдельности) не определена подстановка $\phi$.
		Поэтому, определим $u'_1$ и $u'_2$ таким образом,
		что $u'_1u'_2=\phi(u_1u_2)$ и $u'_2u'_1=\phi(u_2u_1)$.
		Т.е. $u'_2u'_1$ это и есть наш произвольный сдвиг.
	}
	\begin{description}
		\item[(l3.f1)]
		Заметим, что $\phi((u_1u_2)^k)$ представима тождественной перестановкой для $A_n$
		т.к. по \ref{l3c2}
		\pmb{$f_{A_n}(\phi((u_1u_2)^k))$ имеет} \pmb{суффикс равный}
		{\it$bc${--}предпосылке}
		\pmb{$A_n$}.
		Тогда $\phi((u_1u_2)^{2k})$ так же тождественная перестановка.
		\\
		Значит $f_{A_n}(\phi((u_1u_2)^{2k}))=f_{A_n}(\phi((u_1u_2u_1u_2)^k))$ так же имеет суффикс $A_n$.
		
		%
		%
		%
		
		Т.к. $\phi(u_1u_2)^k=\phi((u_1u_2)^k)$ тождественна, то
		\pmb{$f_{A_n}\big(\phi((u_1u_2)^{2k})\big)
		=f_{A_n}\big(\phi(u_1u_2)^k\phi(u_1u_2)^k\big)
		=f_{A_n}\big(\phi(u_1u_2)^k\big)^2$}
		
		Т.е. $f_{A_n}\big(\phi((u_1u_2)^{2k})\big)$ это квадрат.
		Тогда $f_{A_n}\big(\phi(u_1u_2)^k\big)^2$ можно представить как
		
		$$f_{A_n}(u'_1).f_{A'_n}(\phi(u_2u_1)^k).f_{A'_n}(\phi(u_2u_1)^{k-1}).f_{A_n}(u'_2)$$
		
		где $A'_n$ --- $n${--}суффикс в $A_n.f_{A_n}(u'_1)$, а точки это конкатенации.
		{\cGr
			Заметьте, что при целых $bc${--}сдвигах в $A'_n$ нет одинаковых букв
			т.к. смещение происходит на целый $bc${--}код т.е. $\phi(u_2u_1)^k$
			заканчивается либо на \ze, либо на \pl
			\footnote{
				Спасибо Самсонову А.В. за вопрос.
				Автор собирался обобщить для всех смещений
				т.е. когда суффикс длины $n$ может состоять из $n{-}1$ различной буквы (с одинаковыми буквами по краям)
				когда писал, но долго переписывать пришлось бы.
				Поэтому, ограничился <<протыми>> сдвигами, достаточными для
				доказательства основной (экспоненциальной) теоремы в частных случаях.
				В планах были и др. виды обобщений для ДР2, которыее видел автор,
				но некоторые события надолго отбили желание у автора продолжать развивать его собственные леммы.
			}.
			%
		}
		%
		
		Учитывая замечание \ref{l3:nt:bc:open} получаем, что
		в $f_{A'_n}(\phi((u_2u_1)^k))$ $n${--}суффикс совпадает с предпосылкой $A'_n$
		или, что то же самое, \pmb{$n${--}суффикс в $f_{A_n}(\phi((u_2u_1)^k))$ совпадает с предпосылкой $A_n$},
		что и доказывает \ref{l3f1}.
		Т.е. $\phi((u_2u_1)^k)$ представима тождественной перестановкой.
		А в случае целого сдвига, представима только тождественной перестановкой.
		
		{\cGr
			Для большей строгости доказательства общего случая (на основе доказательства целых сдвигов)
			можно использовать смещение на 1 влево и вправо.
			Тогда, по замечаниям \ref{BcC-NotBcC} и \ref{An-BcC-NotBcC}, можно понять,
			что в оставшемся случае --- при \Anm{--}предпосылке $A_n$ соседние сдвиги будут целыми и там перестановки однозначны
			т.е. (как мы выяснили) тождественны.
			Т.е. левая буква в предпосылке $A_n$ должна переходить в левую букву $n${--}суффикса под действием перестановки.
			Соответственно, правая в правую. Т.е. перестановка так же тождественна.
			
		}
		
		\item[(l3.f2)]
		{\cGr
			Идея доказательства проста --- оценить снизу/сверху некоторый параметр (например, длину повтора) кратчачйшего
			контрпримера и сверху/снизу этот же параметр в хорошем примере.
			Соотнеся оценки, показать возможность этого только при условиях, противоречащих ограничениям в нашей лемме.
			В процессе доказательства мы покажем альтернативные доказательства, но продолжим метод предложенный в ДР1.
			
			Сначала увидим, что кратчайший контрпример (запрещённый фактор) неизбежно слишком длинный при больших $k$
			(т.к. не должен целиком входить в $bc${--}сдвиг кратный $|u|$ т.е. когда $|u'_1|\bmod{|u'|}=|u'_2|\bmod{|u'|}=0$),
			а точнее длина кратчайшего контрпримера будет не менее $(k-1)|u'|+2$, что почти очевидно.
			Это неравенство позволит связать неравенства полученные для контрпримера и противоположного фактора
			(противоположный в кольцевом/циклическом сдвиге) с теми же повторами.
			Но тогда противоположный фактор будет с ещё большей экспонентой при слишком больших $n$ и $k$.
			Т.е., оценив сверху и снизу длину повтора контрпримера, мы поймём,
			что это невозможно при ограничениях в нашей лемме, кроме случая $n=5, k=3$.
			Для оценки общего случая (в частности и для $n=5, k=3$) мы докажем,
			что противоположный фактор с теми же повторами лежит в слове с кодом $\phi(u)^3$,
			что позволит достаточно сильно ограничить значение длины повтора.
			И, сравнивая границы его длины, покажем, что такое возможно только при $n<5$ или $k<3$.
			
			Для $\eps$ доказательство условно можно разбить на 3 части:
			
			(1) Не зависимо от $\eps$ (т.е. при $\eps=0$) оценим длины контрпримера $\x\y\x$ и его дополнения $\x\y'\x$
			(\ref{l3e1} и \ref{l3e2})
			в нашем кольцевом слове $\x\y\x\y'$ (т.е. некотором сдвиге нашего слова $f_{A_n}(\phi(u)^k)$).
			Сами неравенства \ref{l3e1} и \ref{l3e2} вытекают только из-за условий существования и минимальности контрпримера
			(код которого не входит в $k$ корней $\phi(u)$)
			т.е., даже независимо от свойств ГС, а значит и независимо от $\eps$.
			
			(2) Докажем \ref{l3e3} т.е. вхождение кода циклического дополнения к контрпримеру (с общими повторами) $\x\y'\x$
			в 3 повтора $bc${--}корня $\phi(u)$.
			Здесь, из доказательства для $\eps=0$ вытекает и для остальных случаев $\eps\ge0$
			(даже для любых $\eps\ge-2$ при $n\ge5$).
			Это очевидно т.к. увеличение $\eps$ только сокращает минимальную необходимую длину повтора для контрпримера,
			а с ней и длину самого слова при том же периоде.
			Но для следующей части мы вычислим неравенства в более общем виде для $\eps$,
			поэтому будем уточнять неравенства для $\eps$ и в этой части.
			
			(3) Через оценку длины повтора $\x$ снизу и сверху докажем невозможность контрпримера
			при одновременных $n\ge5$ и $k\ge3$.
			Здесь нам потребуется сравнение разных $\eps$ из-за разных длин периодов.
			Здесь, так же, снижение значения $\eps$
			(и меньшего и большего из них при условии неубывания $\eps_p$ по $p$) до $0$ усиливает вывод.
		}
		
		\fbox{\pmb{
			Докажем ОП.
			Пусть
			$\lexp^\eps(f_{A_n}(\phi((u_2u_1)^k)))>\RT(n)$}
		} т.е. $f_{A_n}(\phi((u_2u_1)^k))\not\in\D_{\ge3n-3,n}^\eps\cap\T_n$
		
		
		Из предположения следует, что существуют такие $m,l$, что $\exp^{\eps}(\w_{u^k}(m,l))>\RT(n)$.
		\footnote{В ДР1 при $exp$ забыт штрих.}
		
		Напомним, что \pmb{$\w_{u}(m,l)\sim\pref_l\big(f_{A_n}\big(u'_2u'_1\big)\big)$}
		т.е. это $l${--}префикс некоторого слова с кодом равным
		циклическому сдвигу кода $u'_1u'_2=u'=\phi(u)$ на $m=|u'_1|$ влево.
		\footnote{В ДР1 определение даётся через циклический сдвиг слова длины $k|u'|$.
			Здесь мы используем более гибкую формулировку}
		
		БОО пусть
		\fbox{\pmb{
		$l=\min_{0\le m<k|u'|}\big\{l: \exp^\eps(\w_{u^k}(m,l)>\RT(n))\big\}$}}.
		\footnote{Не забывайте, что в ДР1 мы оцениваем только целые {\it$bc${--}сдвиги}.
		Хоть это и не было сказано в ДР1 явно, но это было сказано на 2-й лекции.}
		Т.е. $\w_{u^k}(m,l)$ может быть расширяемым в контрпримере.
		
		Оценим $l$ снизу.
		Обозначим
		$m_2=m(\bmod{|u'|})$ т.е. позиция первой буквы внутри $bc${--}корня $u'$,
		$m_1=m-m_2$ т.е. индекс (с нуля) первой буквы самого $bc${--}корня $u'$ где начинается наш фактор.
		\footnote{В оригинале неаккуратно определено $m_2$.
			Здесь мы, в отличие от оригинала, приведём все их значения к индексам(адеса и смещения с нуля)
			во избежание путаницы. Это не изменит ключевые условия и вывод \ref{l3e1}.
		}
		Очевидно, что и $m_1$ и $m_2$ --- целые $bc${--}сдвиги
		т.к., равные по модулю длины $bc${--}корня $u'$, сдвиги сохраняют целостность.
		А $0$ и $m$ это целые $bc${--}сдвиги.
		Значит, все сдвиги эквивалентные $0$ или $m$ по $\bmod{|u'|}$ являются целыми.
		
		Т.к. (циклический) сдвиг кода Пансьё $u'^{2k}$ на длину $|u'|$ не меняет его,
		то
		$\w_{u^{2k}}(m_1,k|u'|)$ эквивалентно
		
		$\w_{u^{2k}}(m_1+|u'|,k|u'|)$, что эквивалентно
		$\w_{u^{2k}}(0,k|u'|)=\w_{u^k}(0,k|u'|)$ (т.к. $m_1\equiv_{|u'|}0$).
		Тогда можно считать, что $0\le m\le|u'|-1$
		\footnote{В ДР1 $m$ перепутан с позицией.
			Эта путаница несклько объясняет, почему в определении $m_2$ вычетается $1$.
		Видимо, считая, что эти 2 параметра разного типа, путал где какой.
		Но это не ломает основную мысль?идею и легко исправляется.}
		т.е. $m_2=m, m_1=0$.
		Это значит, что в слове $\w_{u^{2k}}(0,(k+1)|u'|)$ есть фактор эквивалентный $\w_{u^k}(m,l)$.
		
		Но, по \ref{l3c3} $\w_{u^k}(0,k|u'|)\in\D_{\ge3n-3,n}^\eps\cap\T_n$.
		А значит самая правая буква <<аналога>> нашего плохого фактора $\w_{u^k}(m,l)$
		должна лежать вне граничного слова $\w_{u^k}(0,k|u'|)$.
		Тогда
		\begin{equation}\tag{l3.e1}\label{l3e1}
		\fbox{\pmb{$l\ge(k-1)|u'|+2$}}
		\end{equation}
		Или проще,
		индекс первой буквы $m<|u'|$, а последней $m+l>k|u'|$,
		откуда $l\ge(k-1)|u'|+2$.
		\footnote{В ДР1 не все рассуждения верны, но ключевой вывод верный.
			Автор хотел строго формально доказать очевидное неравенство \ref{l3e1}.}
		
		Пусть \pmb{$r=l-per(\w_{u^k}(m,l))$} длина повтора нашего кратчайшего плохого фактора в $f_{A_n}\big(\phi(u)^{2k}\big)$
		с началом в $m$.
		\footnote{В оригинале $r$ обозначен как $p$}
		
		\pmb{Очвидно, что $\w_{u^{2k}}(0,k|u'|)=\w_{u^{2k}}(k|u'|,k|u'|)$.
		Тогда $\w_{u^{2k}}(m,r)=\w_{u^{2k}}(k|u'|+m,r)$}.
		При этом, $\w_{u^{2k}}(m,r)$ --- правый повтор плохого фактора, а $\w_{u^{2k}}(m+l-r,r)$ --- левый.
		{\color{gray}
			Отсюда уже видно, что асимптотически при больших $k$ наш фактор плохой только если повторы большие,
			даже больше $|u'|$ при том, что $l$ близко к $k|u'|$.
			Поэтому, для $k\ge4$ дальше доказывается легко, с учётом минимальности $l$.
			Для $k=3$ требуются дальнейшие свойства.
		}
		
		\pmb{Тогда $\w_{u^{2k}}(m+l-r,r)=\w_{u^{2k}}(k|u'|+m,r)$}.
		Т.к. правый повтор достаёт до индекса $m+l$ т.е. до $(k+1)$-го сектора $|u'|$,
		а $\w_{u^{2k}}(k|u'|+m,r)$ начинается с индекса $k|u'|+m$ т.е. в том же $(k+1)$-м секторе $|u'|$.
		{\color{gray}
			При этом, если $m+l>(k+1)|u'|$, то контрпример длиннее допустимого $k|u'|$ т.к. $m<|u'|$.
			Поэтому $m+l\le(k+1)|u'|$.
			Более того, из минимальности $l$ и $k\ge3$ сразу вытекает, что расстояние между правым концом левого повтора и
			левым концом правого будет меньше $|u'|$.
			А значит $r>\frac{|u'|}{2}$.
			Т.е. период противоположного фактора менее $|u'|+r$, откуда его экспонента
			$\frac{per+r}{per}>\frac{|u'|+r+r}{|u'|+r}>\frac{|u'|+|u'|}{|u'|+|u'|/2}=\frac{4}{3}>\RT(n)$ при $n\ge5$.
			Т.е. уже получаем противоречие т.к. противоположный фактор должен входить в $k$ целых секторов $u'$,
			что гарантирует ему граничность.
			Но мы продолжим доказательство по первому методу автора.
		}
		
		Обозначим $l'=k|u'|-l$
		\footnote{Обозначение $l'$ в ДР1 не показано явно,
			но видно по дальнейшему контексту в уравнении $k|u'|-l=l'$ в цепочке неравенств \ref{l3e2}.
		}, тогда $\w_{u^{2k}}(m+l-r, 2r+l')$ --- противоположный фактор (в нашем кольцевом слове) с нашими повторами.
		\pmb{Тогда $\per(\w_{u^{2k}}(m+l-r, 2r+l'))=l'+r$}.
		\footnote{
			По зам-ю \ref{l3:nt:bc:fctr} любой фактор с повторами в каком либо $bc${--}коде
			лежит хотябы в одном его целом $bc${--}сдвиге.
			Поэтому, в ДР1 можно считать, что фактор $\w_{u^k}(m,l)$ находится в целом $bc${--}сдвиге.
		}
		
		С учётом, что $l$ минимально для <<плохости>>, оценим длину повтора $r$.
		По \ref{l3c3} \fbox{\pmb{$\frac{l-1+\eps}{l-r}\le\frac{n}{n-1}$}}
		\footnote{С уменьшением длины сокращается и повтор, но период в знаменателе тот же (т.е. $\eps$ не меняется).
			Так же, важно, что более короткое слово по зам-ю \ref{l3:nt:bc:fctr}
			входит в целый сдвиг $bc${--}кода как и сам контрпример.
			При этом, нет необходимости уменьшать $\eps$ для меньшего повтора с тем же периодом.
			Т.е. можно считать, что $\eps$ максимален для данного периода и повтора так,
			что уеличение повтора (или числителя в экспоненте как $\eps$) на любое положительное число $\in\R$ делает экспоненту запрещённой т.е. $>\RT(n)$.
		},
		откуда 
		$nl-l-n+1+\eps(n-1)\le n(l-r)$, т.е. $\frac{-l-n+1+\eps(n-1)}{n}\le-r$
		или \pmb{$1+\frac{l-1+{\cRe(1-n)}\eps}{n}\ge r$}.
		\footnote{В ДР1 забыт множитель $-(n-1)$ при $\eps$.
		Но неравенство верно, просто слабее.
		Этот недочёт очевиден и исправляется легко до следующего ключевого вывода \ref{l3e3}.
		Этот вывод очевиден, учитывая, что на лекциях это доказывалось при $\eps=0$,
		наверняка (в тексте лекций это и разбиралось),
		а $\eps>0$ только укорачивает допустимые повторы.}
		
		Тогда $1-r\ge\frac{-k|u'|+1+(n-1)\eps}{n}$
		(т.к. $k|u'|\ge l$).
		По \ref{l3e1}, с учётом минимальности $l$, оценим
		\begin{equation}\tag{l3.e2}\label{l3e2}
		\fbox{\pmb{
			$|u'|-2
			\ge k|u'|-l
			=l'
			\ge l-2r
			\ge(k-1)|u'|+2-2r$}
		}
		{\cGr
			\ge\bigg(
			k-1-\frac{2k}{n}+\frac{2+2(n-1)\eps}{n}
			\bigg)|u'|
		}
		\end{equation}
		$$
		{\cGr
			\ge\bigg(
			k-1-\frac{2k-2}{n}
			\bigg)|u'|
			=\bigg(
			(k-1)\frac{n-2}{n}
			\bigg)|u'|
			\ge
			\frac{6}{5}|u'|
			>|u'|-2
		}
		$$
		
		{\color{gray}
		Как видно, уточнение приводит к противоречию, что так же доказывает нашу лемму при любых $\eps\ge0$.
		Но автором
		это не было замечено т.к. видел более красивое свойство,
		для достижения которого и выведена эта цепочка неравенств.
		Но при более грубом неравенстве из этой цепочки $|u'|-2>(k-1-\frac{2k}{n})|u'|$,
		можно проще доказать почти для всех $n\ge5, k\ge3$ кроме случая $n=5,k=3$,
		как это было сделано на 2-й лекции семинара.
		Например так:
		
		Найдём $k, n$, при которых выполняется $(k-1-\frac{2k}{n})|u'|\ge|u'|-2$
		(т.е. когда цепочка неравенств противоречива, для этих значений уже будет доказана наша Лемма).
		Добавим $|u'|$ и получим $2|u'|-2\le\big(k-\frac{2k}{n}\big)|u'|=k\frac{n-2}{n}|u'|$.
		
		Получаем достаточное условие для противоречия $k\frac{n-2}{n}\ge2$.
		Видно, что $\frac{n-2}{n}\ge\frac{1}{2}$ при $n\ge4$.
		Тогда получим $k\ge4$.
		
		Осталось доказать для $k=3$.
		Подставляя в неравенство $|u'|-2\le(2-6/n)|u'|$, получим $-2\le|u'|(1-6/n)$.
		Что выполняется при $n\ge6$.
		Т.е. осталось доказать только для $n=5, k=3$.
		}
		\footnote{На 2-й лекции автор показывал доказанные случаи на графике (на доске)
		}
		
		Т.к. \pmb{$l=k|u'|-l'$} и по (\ref{l3e2}) $l'\ge(k-1)|u'|+2-2r$, то 
		
		$$\pmb{
			r\le\frac{k|u'|-l'-1+\eps{\cRe(1-n)}}{n}+1
			\le\frac{k|u'|-(k-1)|u'|+2r-2-1+\eps{\cRe(1-n)}}{n}+1
			=\frac{|u'|+2r-3+\eps{\cRe(1-n)}}{n}+1
		}$$
		
		Тогда перенесём $r$ влево и учтём, что $n\ge5$
		$$\pmb{
			r\le\frac{|u'|+n-3+\eps{\cRe(1-n)}}{n-2}
			<\frac{|u'|}{3}+1+\frac{\eps{\cRe(1-n)}}{n-2}
			<|u'|+1+\frac{\eps{\cRe(1-n)}}{n-2}
		}$$
		{\color{gray}
		Т.е. из целочисленности $r\le|u'|$, но мы этим не воспользуемся
		\footnote{Оценка могла быть намного точнее.
		И нужна аккуратность при использовании свойства целочисленности}
			Из 1-го неравенства (в цепочке неравенств), уже можно установить
			$r\le\frac{|u'|}{n-2}+\frac{n-(3+\eps(n-1))}{n-2}<\frac{|u'|}{n-2}+1$,
			но мы ограничим ещё точнее.
		}
		Откуда, индекс левой буквы слова $\w_{u^{2k}}(m+l-r,l'+2r)$ ограничен
		\pmb{$m+l-r\ge k|u'|+1-r\ge(k-1)|u'|-\frac{\eps{\cRe(1-n)}}{n-2}$}
		и, с учётом ограничения на $l'\le|u'|-2$ в (\ref{l3e2}), $m\le|u'|$ и $l+l'=k|u'|$,
		индекс правой буквы
		\pmb{$m+l+l'+r-1<(k+2)|u'|+1+\frac{\eps{\cRe(1-n)}}{n-2}-1$}
		\footnote{В ДР1 использована эта грубая оценка с недостающим коэффицентом $(1-n)$ при $\eps$.
			Просто, вывод уже очевиден, особенно, если использовать более точные выведеные ограничения на $r<|u'|/(n-2)+1\dots$.}.
		Теперь видно, что $\w_{u^{2k}}(m+l-r,l'+2r)$ находится в позициях $\big[(k-1)|u'|,...,(k+2)|u'|-1\big]$ слова
		$f_{A_n}(u'^{2k})$ при любых $\eps\ge0$.
		Значит
		\begin{equation}\tag{l3.e3}\label{l3e3}
		\fbox{\pmb{слово $\w_{u^{2k}}(m+l-r,l'+2r)$ эквивалентно некоторому фактору
		в $\w_{u^{3}}(0,3|u'|)\subseteq f_{A_n}(u'^{k})$.}}
		\end{equation}
		
		{\color{gray}
			Понятно, что при $\eps\ge0$ выполняются достаточные условия вхождения фактора длины $l'+2r$ в позициях 3 $bc${--}корней $u'$.
			Но можно расширить допустимые значения для удовлетворения этих условий.
			
			Оценим точнее левую границу
			$$
			m+l-r
			\ge k|u'| + 1 + \frac{-|u'|-(n-3)+\eps(n-1)}{n-2}\ge(k-1)|u'|
			$$
			С учётом \ref{l3c1} последнее неравенство выполняется при
			$\frac{n-1}{n-2}\eps\ge(\frac{1}{n-2}-1)|u'|+\frac{n-3}{n-2}-1\ge\frac{n-3}{n-2}(1-|u'|)-1$
			Т.е. при $\eps\ge\frac{n-3}{n-1}(1-|u'|)-\frac{n-2}{n-1}\ge3-n-\frac{n-2}{n-1}$
			левая буква находится правее $k-1$ $bc${--}корня $u'$
			
			Оценим точнее правую границу
			$$
			m+l+l'+r-1\le(k+1)|u'| + r - 1
			\le\Big(k+\frac{n-1}{n-2}\Big)|u'| + \frac{n-3}{n-2} - \frac{\eps(n-1)}{n-2} - 1
			\le(k+2)|u'| - 1
			$$
			С учётом \ref{l3c1} последнее неравенство выполняется при
			$\frac{n-1}{n-2}\eps\ge(\frac{n-1}{n-2}-2)|u'|+\frac{n-3}{n-2}
			\ge\frac{n-3}{n-2}-\frac{n-3}{n-2}|u'|=\frac{n-3}{n-2}(1-|u'|)$.
			Т.е. при $\eps\ge\frac{n-3}{n-1}(1-|u'|)\ge3-n$ правая буква находится в границах первых $k+2$ $bc${--}корней $u'$
			(т.е. корней $\phi(u)=u'$).
		}
		
		Т.е. слово $\w_{u^{2k}}(m+l-r,l'+2r)$ эквивалентно некоторому фактору в $f_{A_n}(\phi(u)^k)$.
		
		По (\ref{l3e2}) видно, что $l'\ge l-2r$, тогда $l'+r\ge l-r$.
		Т.е. период фактора $\w_{u^{2k}}(m,l)$ (=$l-r$ и меньше он быть не может в силу минимальности $l$ в определении)
		короче периода в $\w_{u^{2k}}(m+l-r,l'+2r)$ (т.е. $\le l'+r$).
		\footnote{
			Можно понять и устно.
			По замечанию \ref{l3:nt:bc:fctr}
			противоположный фактор длины $l'+2r$ с повторами длины $r$
			лежит в некотором смещении исходного слова.
			Тогда по минимальности $l$ тот не может быть короче контрпримера
			т.к. иначе он так же является контрпримером (из-за таких же повторов и короткой длины).
			Но контрпримером он быть не может т.к. его эквивалент входит в исходное слово $f_{A_n}(\phi(u)^k)$
		}
		Значит (по условию неубывания $\eps_p$ по периоду $p$) $\eps$ в усиленной экспоненте для первого фактора не больше $\eps'$ для второго.
		\footnote{
			На лекциях рассматривался случай $\eps=\eps'=0$.
			И т.к. для доказательства вхождения контрпримера в 3 $bc${--}корня для $\eps\ge0$ достаточно доказать для $\eps=0$,
			то в ДР1 для $\eps$ уже достаточно было найти допустимые значения при $\eps'=\eps$
			(т.к. $\eps'$ используется для оценки $r$ сверху т.е. чем больше $\eps'$
			тем строже получаются неравенства, что только усиливает вывод в доказательстве).
			Но т.к. в дальнейшем мы не учитываем $\eps'$ (т.е. как бы $\eps'=0$, игнорируя условие $\eps\le\eps'$),
			то это только добавляет ограничений для $\eps$.
			Но, даже так, <<свободы>> для $\eps$ хватает для применения Леммы \ref{l2}.
		}
		
		Теперь мы можем применить \ref{l3c3} к $\w_{u^{2k}}(m+l-r,l'+2r)$ для оценки $r$ сверху.
		Т.е. \pmb{$\frac{l'+2r+\eps{\cRe'}}{l'+2r-r}\le\frac{n}{n-1}$},
		для некоторого $\eps'\ge\eps$.
		Откуда $-l'-2r+(n-1)\eps'\le-n\cdot r$ или \pmb{$l'-(n-1)\eps{\cRe'}\ge(n-2)r$}.
		{\color{gray}Тогда по \ref{l3e2} $r\le\frac{|u'|}{n-2}-\frac{(n-1)\eps'}{n-2}$.}
		\footnote{Эта точность избыточна при $0\le\eps'\le\frac{2}{n-1}$ как мы увидим при ограничении $r$ снизу,
			но мы решим в более общем случае для $\eps'\ge-1$.
		}
		Тогда
		\begin{equation}\tag{l3.e4}\label{l3e4}
		\fbox{\pmb{$r<\frac{|u'|}{n-2}$}}
		\text{ или }
		\frac{r\cdot n}{|u'|}\le\frac{n}{n-2}{\color{gray}-\frac{n(n-1)\eps'}{(n-2)|u'|}}
		\end{equation}
		
		Чтобы перейти к последнему неравенству в 3-й лемме ДР1 \big(т.е. $k-1<\frac{n}{n-2}$\big),
		достаточно найти ограничение для $r$ снизу.
		\footnote{На сколько помнит автор, на лекции этот переход не был объяснён.
			Но все необъяснёные переходы автора имеют несложные объяснения, обычно связано с
			методом ОП как в первых утверждениях в Лемме \ref{l1} в ДР1.
		}
		Воспользуемся, наконец, ограничением на $r$ снизу
		используя допущение запрещённости фактора длины $l$.
		
		Тогда
		$\frac{l+\eps}{l-r}>\frac{n}{n-1}$,
		или $-l+(n-1)\eps>-nr$,
		откуда $r>\frac{l-(n-1)\eps}{n}$.
		Используя (\ref{l3e1}), получим
		$\frac{(k-1)|u'|+2-(n-1)\eps}{n}<r$, тогда
		$(k-1)+\frac{2-(n-1)\eps}{|u'|}<\frac{r\cdot n}{|u'|}$.
		\footnote{Заметьте, что при $0\le\eps\le\frac{2}{n-1}$ сразу вытекает неравенство $k-1<\frac{n}{n-2}$,
			 которого достаточно для Леммы \ref{l2}}
		Тогда, применяя (\ref{l3e4}), получим
		
		$$k-1+\frac{2-(n-1)\eps}{|u'|}
		<\frac{n}{n-2}{\color{gray}-\frac{n(n-1)\eps'}{(n-2)|u'|}}$$
		
		\begin{equation}\tag{l3.e5}\label{l3e5}
		k-1
		<\frac{n}{n-2}{\cGr-\frac{n(n-1)\eps'}{(n-2)|u'|}}+\frac{(n-1)\eps-2}{|u'|}
		\le\frac{n}{n-2}
		\end{equation}
		
		{\color{gray}
		Уточним, при каких $\eps$ и $\eps'$ выполняется неравенство $-\frac{n(n-1)\eps'}{(n-2)|u'|}+\frac{(n-1)\eps-2}{|u'|}\le0$.
		Учитывая что $|u'|\ge n>2$ получим
		
		$$
		\frac{(n-2)\eps-2\frac{n-2}{n-1}}{|u'|}\le\frac{n\eps'}{|u'|}
		\Longleftrightarrow
		(n-2)\eps-2\frac{n-2}{n-1}\le n\eps'
		\Longleftrightarrow
		-2\frac{n-2}{n-1}\le 2\eps+n(\eps'-\eps)
		\Longleftrightarrow
		\eps\ge-\frac{n-2}{n-1}-dn
		$$
		}
		
		Т.о. мы установили, что контрпример возможен только при
		\fbox{\pmb{$k-1<\frac{n}{n-2}$}} при любых $0\le\eps\le\eps'\le\frac{2}{n-1}$
		{\color{gray}(даже при любых $\eps'\ge\eps\ge max\{3-n, -\frac{n-2}{n-1}\}=-\frac{n-2}{n-1}$)}.
		Видно, что даже при $n\ge4$ необходимо чтобы $k-1<2$
		
		\fbox{\pmb{т.е. $k<3$ при $n\ge5$ \contr\ ограничению $k\ge3$ в условии нашей леммы}}.
		Это противоречие и доказывает нашу лемму.
		
		{\color{gray}
			А теперь уточним ещё сильнее допустимые значения $\eps$, используя \ref{l3e5}.
			Нам достаточно, чтобы $k-1<\frac{n-1}{n-3}$, тогда получаем то же противоречие условиям нашей леммы только теперь при $n\ge5$.
			Тогда условия на $\eps$ и $\eps'$
			$$
			\frac{n}{n-2}{\color{gray}-\frac{n(n-1)\eps'}{(n-2)|u'|}}+\frac{(n-1)\eps-2}{|u'|}
			\le\frac{n-1}{n-3}
			$$
			Перегруппируем слагаемые
			$$
			\frac{n}{n-2}-\frac{2}{|u'|}-\frac{n-1}{n-3}
			\le\frac{n(n-1)\eps'}{(n-2)|u'|}-\frac{(n-1)(n-2)\eps}{(n-2)|u'|}
			\Leftrightarrow
			\frac{-3n+n+2n-2}{(n-2)(n-3)}-\frac{2}{|u'|}
			\le\frac{(n-1)(n\eps'-(n-2)\eps)}{(n-2)|u'|}
			\Leftrightarrow
			$$
			$$
			\frac{-2|u'|}{(n-2)(n-3)}-2
			\le\frac{(n-1)(n(\eps'-\eps)+2\eps)}{(n-2)}
			\Leftrightarrow
			\frac{-2|u'|}{(n-1)(n-3)}-2\frac{n-2}{n-1}
			\le n(\eps'-\eps)+2\eps
			$$
			Найдём верхнюю границу возможного значения справа при $|u'|\ge n\ge5$.
			Максимальное значение достигается при $|u'|=n$, тогда
			$$
			\frac{-2|u'|}{(n-1)(n-3)}-2\frac{n-2}{n-1}
			=2\frac{-|u'|-(n^2-5n+6)}{(n-1)(n-3)}
			=-2\frac{n^2-(5n-|u'|)+6}{n^2-4n+3}
			\le-2\frac{n^2-4n+6}{n^2-4n+3}
			\le-2
			$$
			Разберём случай $\eps=\eps'$.
			Тогда, при $|u'|\in o(n^2)$ получим $\sup$ нижних границ для $\eps'$ равна $-1$ (т.е. когда $n\to\infty$).
			При $|u'|\in\Theta(n^2)$ получим снижение $\sup$ на некоторую константу.
			При $|u'|\in\omega(n^2)$ и $n\to\infty$ допустимая нижняя граница на $\eps'$ стремится к $-\infty$.
			Т.е. минимальные равные $\eps$ и $\eps'$ могут быть равны -1 и больше.
			
			
			Аналогично, в остальных случаях (напомним, что $\eps'$ привязан к более длинному периоду,
			а с ним и не менее короткому повтору)
			видно, что при $\eps'>\eps$ ограничение снизу на $\eps$ даже снижается, особенно с ростом $n$.
			
			Так же, можно заметить, что $\eps'$ может быть даже немного меньше $\eps$, 
			но с ростом $n$ эта допустимая разница уменьшается (при условии линейного роста $|u'|$ от $n$).
		}
	\end{description}
\end{proof}

\hyperlink{contents}{$\upuparrows$}

Заметим, что возможность отрицательных $\eps$ могут быть связаны с обобщением
циркулярных слов с почти граничными свойствами при $n\ge5$.
А так же с возможностью построения $\D_{3,n}^\eps${--}ЦГС для $n\in\{3,4\}$.

Так же, заметьте, что $\eps'$ могло быть немного, даже больше чем $\eps$,
что несколько позволит расширить условия нашей леммы.

\begin{quest}\label{q1}
	$k|u'|$ не меньше $n^2$? Найти нижнюю границу, хотябы асимптотическую.
\end{quest}

\begin{quest}\label{q2}
	Можно ли в эту лемму включить $k=2$, хотябы при каком-то $n$?
\end{quest}

Предположение автора (по вопросу 2), что нет т.к. даже при $n\to\infty$ нижняя граница $k$ только стремится к $2$.
Но тогда должен существовать контрпример т.е. когда ГС получено по квадрату непериодического $bc${--}кода,
но некоторый сдвиг содержит недопустимы фактор.

\paragraph{Задания и вопросы для читателя:}\ 

(1) Оценить длины повторов наших $\D_{3,n}^\eps${--}ЦГС.
\footnote{На этот вопрос сразу ответил Шур А.М. в конце 2-й лекции, на сколько помнит автор.
	Попробуйте и вы.}
%

(2) Существуют сопряжённые слова полученные целыми $bc${--}сдвигами?

По второму вопросу в конце 2-й лекции автора, когда озвучивал некоторые вопросы, которые ещё предстояло решить автору,
наш НР увидел красивое аналитическое решение,
но достаточно было свести ответ к полиномиальной проверке, что и показано в ДР1 (но не реализовано).
Самая сложная, по времени, задача --- проверить условие \ref{l2c2},
поэтому автор считал приоритетнее оптимизацию этой проверки,
для чего и был придуман алгоритм поиска запрещённого фактора за $n\log(n)$.


\hyperlink{contents}{$\upuparrows$}

\paragraph{$\D_{3,n}^\eps${--}ЦГС}

Из выполнения условий Леммы \ref{l3} (с учётом серого текста) следует существование $\D_{3,n}^\eps${--}ЦГС.

Можно заметить, что каждое $\D_{3,n}^\eps${--}ЦГС порождает
класс эквивалентности по циклическому сдвигу.
Т.е. любые 2 слова в нём получаются друг из друга через некоторое смещение и <<перестановку букв алфавита>>

\hyperlink{contents}{$\upuparrows$}

\subsubsection{Замечания и дополнения}

Здесь мы используем Лемму \ref{l3} в общем случае т.е. для любых сдвигов (не только для целых) и любых $n\ge5, k\ge3, \eps\ge-1$.

Назовём слова над $\A_n$ {\bf перестановочно эквивалентными} (ПЭ) словами (ПЭС), если одно получается из другого через перестановку букв их общего алфавита.
Это класс эквивалентности, состоящий из $n!$ различных слов.
Т.е. это обычные эквивалентные слова, но для определённости, здесь будем называть их как ПЭ.

Назовём слова над $\A_n$ {\bf циклически эквивалентными} (ЦЭ) словами (ЦЭС), 
если одно получается из другого циклическим сдвигом.
Слова Линдона (минимальные в лексикографическом порядке) в таких классах можно считать каноническими.
Т.е. класс ЦЭ слов к $w$ (содержащий $w$) состоит из не более чем $|w|$ различных слов.
Не более т.к. в этом классе слова могут быть самосопряжёнными
т.е. слово $w=uv$ из этого класса может быть равно $vu$ при $|u|,|v|>0$ (т.е. являться целой степенью $\ge2$).
Но для не самосопряжённого $w$ этот класс состоит ровно из $|w|$ слов.

\begin{note}\label{nt-k-bc-NSC}
	Любое $k${-}$bc${--}корневое слово несамосопряжено
	(что вытекает из минимальности корня и $k$).
\end{note}

Слова, являющиеся ЦЭ или ПЭ, назовём {\bf циклически-перестановочно эквивалентными} (Ц-ПЭ) словами (Ц-ПЭС).
Т.е. они так же создают класс эквивалентности.

Если слова $u$ и $v$ эквивалентны Ц/П/Ц-П,
то будем говорить, что $u$ ЦЭС/ПЭС/Ц-ПЭС к $v$.

Любой из этих классов эквиваленстности порождается из любого его элемента всевозможными сдвигами и/или перестановками букв алфавита.
Из Леммы \ref{l3} следует

\begin{note}\label{nt-CE}
	Пусть $n,k\in\mN, n\ge5, k\ge3, \eps\ge-1$
	$u_1u_2$ --- $bc${--}код,
	$|\phi(u_1u_2)|\ge n$ и
	$f_{A_n}(\phi(u_1u_2))$ имеет $n${--}суффикс $A_n$.
	Тогда, используя Лемму \ref{l3},
	$f_{A_n}(\phi((u_1u_2)^k))\in\D_{3, n}^\eps\T_n$
	$\Leftrightarrow$
	$f_{A_n}(\phi((u_2u_1)^k))\in\D_{3, n}^\eps\T_n$.
	
	Другими словами, выполняется одновременная $D_{3,n}^\eps${--}граничность ЦЭС над $\A_n$,
	если хотябы одно из них является $k${-}$bc${--}корневым при $k\ge3, n\ge5, \eps\ge-1$
	(в общем-то, тогда они все $k${-}$bc${--}корневые).
\end{note}

\begin{note}\label{nt-PE}
	Среди всех сдвиов любого $k${-}$bc${--}корневого слова $w$ есть ровно $k$ ПЭС к $w$.
\end{note}

Т.е. в классе ЦЭС, порождённого $k${-}$bc${--}корневым словом длины $k\cdot l$, ровно $l$ классов ПЭС.

\begin{note}
	Пусть класс Ц-ПЭС порождён $k${-}$bc${--}корневым словом $w$.
	Представим класс Ц-ПЭС как таблицу элементов $(n!)\times(k\cdot l)$,
	где по вертикали все последоваельные (циклические) сдвиги.
	Тогда очевидно, что:
	\begin{enumerate}
		\item Каждая строка таблицы это класс ПЭС порождённый ЦЭС к $w$.
		\item Через каждые $l$ строк $bc${--}корень слова $w$ повторяется в нашей таблице.
		А значит строки с индексами равными по модулю $l$ это равные классы ПЭС, порождённые ЦЭС к $w$.
		\item Операции сдвига и перестановки букв алфавита коммутативны.
		Откуда, любой столбец таблицы это класс ЦЭС, порождённый ПЭС к $w$.
	\end{enumerate}
\end{note}

Поэтому, различных элементов в классе Ц-ПЭС не более $n!\times l$.
Т.е. класс Ц-ПЭС, порождённый $k${-}$bc${--}корневым словом длины $k\cdot l$ по зам-ю \ref{nt-k-bc-NSC},
можно рассматривать как таблицу $n!\times l$,
состоящую из различных $k${-}$bc${--}корневых слов.

Набор слов, построеный по схеме в Лемме \ref{l3} обладает сразу несколькими очевидными свойствами
(кроме доказанных в Лемме), нужными нам для Леммы \ref{l2}:

1) Длины этих слов совпадают. Для этого свойства автором и была раработана схема построения $D_{3,n}^\eps${--}ЦГС.

2) Все они являются простыми т.е. их экспонента равна $1$ (т.е. их период совпадает с их длиной).
Асимптотически это очевидно, используя Лемму \ref{l3}.
Точнее, если предположить обратное,
то при достаточно маленьком повторе можно взять циклическое смещение и получить квадрат \contr\ граничности.
Если же не достаточно маленькие, то само слово не может быть граничным при достаточно большом $n$ (даже при $n\ge3$).

2-е свойство легко сводится к проверке за полином для конкретных конструкций,
как и следующие свойства для доказательства существования РРДГС,
при построении набора по схеме в этой Лемме.

Доказательство 2-го свойства и ответ на вопрос \ref{q1} для читателя.

\begin{note}\label{pr1}
	Пусть $n\ge5,k\ge3$ и $bc${--}код $u$ удовлетворяет условиям Леммы \ref{l3} и $f_{A_n}(\phi((u)^k))=v'vv'$.
	Тогда $|v'|=0$.
\end{note}
\begin{proof}
	Иначе сдвиг $v'v'v$ (с квадратом) --- ПЭС к $f_{A_n}(\phi((u_2u_1)^k))$
	(при некоторых $u_1u_2=u$)
	
	\hfill т.е. $f_{A_n}(\phi((u_2u_1)^k))$ тоже содержит квадрат \contr\ \ref{l3f2}.
	\footnote{
		Среди целых сдвигов так же найдётся с квадратом $v'v'$, используя замечание \ref{l3:nt:bc:sh}\ref{2sh}.
	}
	%
	%
	%
	%
\end{proof}
это замечание доказывает 2-е свойство
(оно не является необходимым для сведения проверки слов к полиномиальной
т.к. для конкретного набора слов легко, непосредственным анализом их конструкции, как мы покажем).

Следующее замечание взято из текста для лекций --- 1-е свойство (<<свойства слов $w=f_n(A_n,$...>>)

\begin{note}
	Пусть $w=\w_{u^k}(0,u'k)$. 
	Для поиска $\lexp^{\cGr\eps}(w)$ достаточно найти максимум $\exp^{\cGr\eps}(v)$ у всех $v\subseteq w$ 
	где начало $v$ лежит в первом сегменте ($bc${--}корне) $w$.
\end{note}
\begin{proof}
	Заметим, что для любых $0\le c'<k$, $0\le c<|u'|$ и $l>0$ таких, что $c'|u'|+c+l\le k|u'|$
	эквивалентны слова $\w_{u^k}(c,l)$ и $\w_{u^k}(c'|u'|+c,l)$.
	При этом, каждому фактору $v\in w$ найдутся такие $c',c,l$, что $v=\w_{u^k}(c'|u'|+c,l)$.
	
	А значит $\exp^{\cGr\eps}(v)=\exp^{\cGr\eps}(\w_{u^k}(c'|u'|+c,l))=\exp^{\cGr\eps}(\w_{u^k}(c,l))$.
\end{proof}

Из этого замечания следует, что для поиска $\lexp^\eps$ достаточно взять максимальную $\exp^\eps$ среди $|u'|^2k$ факторов.

\begin{note}[следствие зам-я \ref{l3:nt:bc:sh}]\label{nt:_bc_}
	Не каждый циклический сдвиг в $\D_{3,n}^\eps${--}ЦГС $w$ имеет целый $\rbc${--}корень,
	но каждый не более чем 3-й в общем случае имеет $\rbc${--}корень,
	и каждый не более чем 2-й имеет либо $\rbc$, либо $\lbc${--}корень.
	
	Каждое $\D_{3,n}^\eps${--}ЦГС имеет сдвиг с целым $bc${--}корнем,
	а так же, сдвиг с не целым $bc${--}корнем.
	
	Каждое $\D_{3,n}^\eps${--}ЦГС имеет длину не менее 2.
\end{note}

Так же можно заметить, что длина любого ЦГС $w$ не меньше $n$
т.к. корень в $w^2$ не больше $|w|$ и должен содержать все буквы из $\A_n$.
Но не так очевидно, можно ли ограничить снизу длину $bc${--}корня в $w$ числом $n$?

\hyperlink{contents}{$\upuparrows$}

\subsection{РРДГС --- Равномерно растущее дерево $\D_{3, n}^\eps${--}ГС}
Экспоненциальность тривиальным образом достигается при помощи дополнительного образа (ДО) --
такая методика была озвучена автором ещё во время лекций (на семинаре в марте 2011).
Но в ДР об этом не сказано явно, хотя, это очевидно.
Экспоненциальность легко представима через дерево,
как описано у нас в ДР1 с применением Леммы \ref{l1}.
Более того, в дереве удваивается количество веток на каждом ограниченном (хотябы какой-нибудь константой) его отрезке,
что потребуется нам для доказательства ещё более сильного утверждения.

Вариаций порождения экспоненциально растущего множества ГС на основе ДО можно придумать много,
где экспоненциальность вытекает из самой конструкции т.е. из линейного расстояния между соседними заменяемыми образами.
А доказательство граничности может быть рутинным.
[В худшем случае, можно повторить доказательство лемм при рассмотрении частных случаев,
аккуратно доказывая для случаев, когда одинаковые образы становятся различными,
как бы разбивая на 2 подслучая.]

Для решения с помощью Леммы \ref{l1} в ДР1 приведена схема,
красиво связанная схема с предложенным множеством из $3(n+1)$ ГС.
И этого уже достаточно для доказательства экспоненциальности
т.к. граничная теорема доказана.

Пусть $\V$ удовлетворяет условиям подстановки.
Тогда выделим две ветки интерпретации идеи <<дополнительного образа>> для порождения экспоненциальности
(egc --- exponential growth condition):
\begin{description}
	\item[(egc1)\label{ec1}]
	$\exists v'\not\in\V$ такой, что
	$\{v'\}\cup\V\setminus\{v\}$ удовлетворяет условиям подстановки,
	для некоторого $v\in\V$.
	\item[(egc2)\label{ec2}]
	$\exists v'\not\in\V$ такой, что
	$\{v'\}\cup\V\setminus\{v\}$ удовлетворяет условиям подстановки,
	для любого $v\in\V$.
\end{description}
Проще говоря, условие \ref{ec2} говорит,
что любое подмножество из $\{v'\}\cup\V$ мощности $|\V|$ должно удовлетворять условиям подстановки.
Понятно, что из условия \ref{ec2} вытекает и условие \ref{ec1}.

\hyperlink{contents}{$\upuparrows$}

\paragraph{\large Неформальное дополнение.}

Доказательство экспоненциальности в ДР1 получилось самым небрежным
из-за недостатка времени перед сдачей диплома.
Поэтому, доказательство в явной форме там не написано.
Но это очевидное следствие, не требующее детального описания.
И главная мысль была донесена --- получить экспоненциальный рост множества слов за счёт дополнительного образа (ДО).
В ДР1 об этой идее явного упоминания автор не нашёл,
но неявно это видно в формулировке при идее доказательства с использованием Леммы \ref{l2}
(которая оказалась копипастом из текста для лекций на семинаре).
А так же, явно эта идея была озвучена автором в конце 2-й лекции 2011 года (коротко и без подробностей т.к. идея очевидна).
И, на сколько помнит автор, несколько раз озвучивал эту идею перед НР (в частности, в конце 2-й лекции), так же без подробностей.


Так же в формулировке не уточнено, что дополнительный образ должен быть равноценен остальным
т.е. удовлетворять условиям соответствующих лемм (т.е. удовлетворять условию \ref{ec2}).
Точнее, множество слов $\V$ должно удовлетворять условию нужной леммы (1-й или 2-й),
но число элементов в нём должно быть, хотябы на 1 больше, чем требуется в соответствующей лемме.
Опять же, по контексту нетрудно до этого догадаться.
Так же можно несколько ослабить тем, что $\V$ может быть и требуемой мощности,
а дополнительный образ $v'$ должен подменять некоторый $v\in\V$ так,
чтобы $\{v'\}\cup\V\setminus\{v\}$ тоже удовлетворяло условию той же леммы (т.е. удовлетворять условию \ref{ec1}).



\hyperlink{contents}{$\upuparrows$}

\subsubsection{РРДГС (для частных случаев) при выполнении условий Лемм с доп.$\V${--}образом}

Более правильный подход для Леммы \ref{l2} --- доказать не конкретно с циклической подменой $\V${--}образа (с 3-мя образами в цикле),
а с произвольной подменой образа так, чтобы расстояние между буквами (между центрами их позиций, или между их началами)
с одинаковым $\V${--}образом было не менее $3n-3$.
Аналогично для Леммы \ref{l1}, но для неё нужно усилить условия \ref{l1c3} условием \ref{l2c3} но без $\eps$.
Тогда будет сохраняться граничность при нашем порождении экспоненциального множества с помощью ДО.
Эти идеи обобщения были изложены в дополнениях к этим леммам.

Но в данной версии работы мы приведём условное (выполнение условий лемм \ref{l1} и \ref{l2}) доказательство
существования РРДГС в той форме, которая подразумевалась в ДР1.
Доказательство с использованием ДО не сложное, но рутинное и требует аккуратности.
В данной версии доказательство несколько сырое (планируется доработать в следующих версиях).

%
%
%
%


\begin{stat}\label{exp:any_subseq}
	Если наше множество $\V_n$ удовлетворяет условиям Леммы \ref{l2} [Леммы \ref{l1} с доп.условием \ref{l2c3}]
	и имеет не менее $3n+1$ [$n+k+1$] $\V${--}образов,
	то любой образ, где заменяемый $\V${--}образ заменён <<альтернативным>> $\V${--}образом
	в произвольных местах его появления в образе,
	принадлежит $\D_{\ge3n-3, n}^\eps\cap\T_n$ [$\T_{n+k}$].

	Более строго --- пусть $i_1<i_2<...<i_m$ позиции всех букв прообраза $\w$ с одинаковым образом,
	которые замеяемы подстановкой на <<альтернативный>> $\V${--}образ.
	Тогда, если любое подмножество этих букв заменить на <<альтернативный>> $\V${--}образ,
	то образ $\w'\in\D_{\ge3n-3, n}^\eps\cap\T_n$ [$\T_{n+k}$].
	
	А так же, всевозможные такие замены порождают экспоненциально растущее дерево слов.
	При этом, расстояние между соседними вершинами (места бифукации, центры позиций вершин) дерева
	ограничены сверху и снизу константами (в обеих леммах).
\end{stat}
\begin{proof}
	Экспоненциально растущщее дерево очевидно.
	Расстояние между соседними вершинами в Лемме \ref{l2} не превосходит $3n+3$, а в Лемме \ref{l1} не превосходит $n+k+1$.
	А также, это же расстояние не менее $3n-3$ в Лемме \ref{l2} и $n+k-1$ в Лемме \ref{l1}.
	Т.о. равномерность роста доказана.
	
Пусть множество слов $W$ порождено экспоненциальной подстановкой $f:\A_m\to\V$ над бесконечным словом $\omega$,
где $|\V|=m+1, |v_i|=|v_j|>1 \forall v_i,v_j\in\V$.
$f(a_m)=v_m$ при нечётном номере появления буквы $a_m$ в слове $\omega$
и $f(a_m)=v_{m+1}$ в обратном случае.

Пусть $\V$ удовлетворяет условию x-леммы (где x --- 1 или 2, соответствующие лемме).
И $\{v'\}\cup\V\setminus\{v\}$ тоже удовлетворяло условию x-леммы.
Пусть бесконечное x-граничное слово $\omega$ порождено схемой по x-лемме.
Пусть КЗ подстановка $f:v\to v'$ порождает экспоненциальное множество слов.
Тогда любое полученное слово x-гранично.

ОП:
пусть фактор $u'x'u'\subset\omega'$ для некоторого $\omega'$, порождённого нашей КЗ подстановкой,
имеет максимально <<плохую>> экспоненту во всём слове $\omega'$.
Тогда повтор этого фактора пересекается либо с $v$ либо с $v'$.
Пусть $u'$ --- повтор слова $u'x'u'$.
Заменим обратно $v'$ на $v$ и возмём соответствующий фактор $uxu\in\omega$.
Заметим, что условия Леммы \ref{l1} вытекают из условий 2-й.


\begin{itemize}
	\item Пусть $|u'|\ge2L/(n-1)$, тогда по (\ref{l1p5e2}) $|u'x'|$ кратен $L$.
	Понятно, что те $v$, которые содержатся в оригинальном повторе $u$ целиком, они же или их дубликаты
	содержатся целиком в $u'$, т.е. не меняют экспоненту.
	Т.е. $u'x'u'$ стало запрещённым из-за частичного пересечения $v$ с $u$.
	А замена $v$ на $v'$ удлинила повтор до $|u'|$ за счёт более длинного пересечения $v'$ с другим $\V${--}образом.
	Но проверим, могло ли новое пересечение испортить экспоненту.
	Нам известно, что $\{v'\}\cup\V\setminus\{v\}$ удовлетворяло условию x-леммы,
	а значит замена $v$ на $v'$ даст хорошую экспоненту аналогичного фактора.
	А значит другой конец фактора тоже пересекает $v$, но в $\omega'$ он не поменян на $v'$.
	БОО пусть $v'$ справа, $v$ слева пересекают фактор $u'x'u'$.
	вот здесь и пригодится свойство худшего случая (длиннейшего общего суффикса или префикса).
	Эти ограничения общие и для $v'$ и для $v$.
	Т.е. можно считать что и для случая когда $v$ имеет максимально допустимый общий суффикс и префикс.
	Тогда и суффикс в $v'$ не превосходит допустимого, а значит и фактор не может быть с <<плохой>> экспонентой.
	
	Проверим это аккуратно.
	Здесь $v_{i_j}, v'_{i_j}$ zявляются $\V${--}образами.
	
	Детализируем левый $u'\subset v_{i_1}v_{i_2}...v_{m-1}v_{i_m}$, где $v_{i_2}...v_{i_{m-1}}\subseteq u'$.
	Ну, и правый $u'\subset v'_{i_1}v'_{i_2}...v'_{m-1}v'_{i_m}$, где $v'_{i_2}...v'_{i_{m-1}}\subseteq u'$.
	
	В силу кратности корня слова $u'x'u'$ длине $\V${--}образа, понятно, что $v_{i_2}...v_{i_{m-1}}=v'_{i_2}...v'_{i_{m-1}}$.
	
	Проще говоря, $v_{i_2}...v_{i_{m-1}}$ это образ целого прообраза повтора $u'$.
	А так же, понятно, что целый прообраз повтора $u'$ совпадает с целым прообразом повтора $u$.
	Т.к. по предположению $uxu$ является <<хорошим>> фактором, а $u'x'u'$ <<плохим>>,
	то разница прообразов $u$ и $u'$ в крайних <<неполных пробразах>>.

	Значит в них и находится наша буква (пусть $a$) с 2-мя образами $v$ и $v'$.
	При этом, если среди них буква $a$ отбражается только в одинаковый образ (пусть в $v$),
	то $u'x'u'$ не может быть <<плохим>>
	т.к. при полной замене $v'$ на $v$ в $\omega'$ мы восстановим <<хороший>> $\omega$,
	при этом, $v_{i_2}...v_{i_{m-1}}$ и $v'_{i_2}...v'_{i_{m-1}}$ так же останутся равными,
	а значит экспонента в $u'x'u'$ не поменяется.
	(дополним, что левые(правые) крайние <<неполные пробразы>> левого и правого $u'$ не могут совпадать
	т.к. иначе будет противоречие максимальности экспоненты фактора $u'x'u'$).
	
	Тогда должны существовать хотябы 2 различных $\V${--}образа (т.е. и $v$ и $v'$) среди крайних
	<<неполных пробразов>> $a$ (в контексте $\omega'$).
	При этом, т.к. их всего 4, а среди левых и правых должно быть ровно по 1-му прообразу $a$,
	то БОО можно считать, что оба образа от $a$ пересекаются с левым повтором $u'$
	т.к. они увеличивают экспоненту только за счет величины пересечения $\V${--}образов
	не зависимо от относительного их расположения (т.е. неважно кто слева а кто справа).
	
	Т.е. достаточно оценить длину $u'$, у которого $a$ --- крайние <<неполные пробразы>>,
	чьи $\V${--}образы --- дубликаты $v$ и $v'$.
	
	Обозначим $\V'=\{v'\}\cup\V\setminus\{v\}$.
	По условию \ref{ec1} $\V'$ обладает свойством $\l(\V')\le\frac{L}{n-1}$ и $\r(\V')\le\frac{L}{n-1}$.
	Так же, обоначим за $p$ длину целого прообраза повтора $u'$.
	
	Заметим, что парные (только левые или только правые) крайние образы в $u$
	не могут быть образами одинаковой буквы (точнее буквы $a$ т.к. мы уже выяснили,
	что среди левых(правых) есть хотябы по одному образу прообраза буквы $a$)
	т.к. иначе мы могли бы удлиннить в оригинале $uvu$ длину повтора целым $\V${--}бразом буквы $a$
	(а мы БОО считаем, что крайние образы не цельно лежат в $uvu$).
	Значит, можно считать, что длина повтора $u'$ ограничена сверху
	$pL+\max\{\l(\V), \l(\V')\}+\max\{\r(\V), \r(\V')\}\le pL+\frac{2L}{n-1}$
	\begin{itemize}
		\item
		Случай, когда корень длины не менее $3(n-1)L$
		достаточен для Леммы \ref{l2}
		т.к. одинаковые буквы в факторе прообраза встречаются с корнем длины не менее $3n-3$.
		Повторяя рассуждения случая \ref{l2p1} Леммы \ref{l2},
		получим доказательство для случая $|u'|\ge2L/(n-1)$
		(принадлежности нашего фактора к $\D_{\ge3(n-1)L,n}^\eps$) для схемы в Лемме \ref{l2}.
		\item
		Случай, когда корень длины не менее $n$
		достаточен для Леммы \ref{l1}
		т.к. в предпосылке прообраз это ГС над $n+k\ge n+1$ буквами,
		т.е. корень целого прообраза фактора длины не менее $n$.
		
		Подслучаи:
		\begin{itemize}
			\item
			Случай, когда корень $|u'v'|\ge2nL$ аналогичен \ref{l1p6}.
			По предположению неграничности $u'v'u'$, используя (\ref{l1p1}), получим
			(с учётом равенства корней $u'v'u'$ и $uvu$ и их кратности $L$) противоречие
			$$
			\frac{n}{n-1}
			<\frac{|u'v'u'|}{|u'v'|}
			=\frac{|uv|+|u'|}{|uv|}
			\le\frac{2nL+|u'|}{2nL}
			\le\frac{2nL+pL+\frac{2L}{n-1}}{2nL}
			\le\frac{n+1}{n}+\frac{1}{n(n-1)}
			=\frac{n^2}{n(n-1)}
			=\frac{n}{n-1}
			$$

			%
			
			\item
			Оставшиеся случаи, т.е. когда корень $|u'v'|\ge nL$ (и $<2nL$, а значит $\le(n+2)L$ для ГС над  $\A_{n+k}$),
			аналогичны \ref{l1p7} и \ref{l1p8}.
			В этих случаях $p\le1$.
			
			В случае $|u'v'|\ge n+k$ и $|u'v'u'|=n+k+p$ граничность доказывается как в \ref{l1p7}.
			Здесь необходимое неравенство проверяется, используя только независимо симметрирчные свойства (условия),
			как, например, $\l(\V)$, $\l(\V')$, $\r(\V)$ и $\r(\V')$
			ограничиваются независимо друг от друга (в худшем случае).
			
			Оставшийся случай $|u'v'|=n+k-1$ и $|u'v'u'|\le n+k$ доказывается как в \ref{l1p8}.
			Здесь же, для проверки неравенства на граничность используется симметричное, но зависимое условие $l+r$ в \ref{l1c3}.
			Например, при максимально допустимом $l$ не возможно использовать максимальное $r$.
			
			Поэтому, проверим оставшийся случай детально.
			
			Но усилим условие \ref{l1c3}
			Пусть у буквы $a$ есть 2 образа $v_{a,1}$ и $v_{a,2}$ у остальных хотябы по 1-му.
			Тогда для любых $u\ne w\in\V\setminus\{v_{a,1}, v_{a,2}\}$ и
			$v,v'\in\{v_{a,1}, v_{a,2}\}$ ($v$ и $v'$ могут совпадать)
			$$\max\{l+r: \suff_l(u)=\suff_l(v), \pref_r(v')=\pref_r(w)\}\le\frac{L}{n-1}$$
			
		\end{itemize}
	\end{itemize}
	
	\item Пусть $|u'|<2L/(n-1)$.
	Заметим, что по (\ref{l2p0e}) $\eps\le1/L(n-1)$.
	Тогда, чтобы нарушалось свойство, необходим достаточно короткий фактор $u'x'u'$ т.е.
	{
	}
	$$
	\frac{|u'x'u'|+\eps}{|u'x'|}
	>\frac{n}{n-1}
	\Longleftrightarrow
	n|u'x'u'|-|u'x'u'|+\eps>n|u'x'|
	\Longleftrightarrow
	|u'x'u'|<n|u'|+\eps
	<\frac{3Ln}{n-1}+\eps
	=3L+\frac{2L}{n-1}+\eps
	$$
	Тогда $|u'x'u'|<4L+\eps$ т.е. $|u'x'u'|\le4L$ даже при $n\ge3$, а значит <<плохой>> фактор входит в 5 соседних $\V${--}образов.
	Такие образы различны при $n\ge5$, при этом, даже для Леммы \ref{l1} ($\forall k\ge1$).
	Но, даже по условию \ref{ec1} любой такой фактор <<хороший>> (для любой из 2-х лемм), а значит и все его факторы тоже
	\contr\ $u'x'u'$ --- <<плохой>> фактор.
\end{itemize}
\end{proof}

Как видно, при слишком коротких повторах сразу доказывается корректность факторов в порождённых словах.
Этот случай почти очвиден т.к. доказывается даже грубыми неравенствами.
И очевиден асимптотически при больших $n$.
А в остальных случаях сразу синхронизируется длина корней с $0$ по модулю $L$.


\hyperlink{contents}{$\upuparrows$}

\newpage

\def\per{\sf per}
\def\Exp{\sf exp}
\def\lexp{\sf lexp}
\def\rt{\sf RT}
\def\gcd{\sf gcd}
\def\lcm{\sf lcm}
\def\lcs{\sf lcs}
\def\lcp{\sf lcp}
\def\lcps{\sf lcps}
\def\F{\bf F}
\def\Su{\sf S}
\def\Pr{\sf P}
\def\ex{\sf ex}
\def\C{\bf C}
\def\E{\mathbf E}
\def\I{\mathbf I}
\def\V{\mathbf V}
\def\U{\mathbf U}
\def\W{\mathbf W}
\def\sE{\mathbb E}
\def\sN{\mathbb N}
\def\sV{\mathbb V}
\def\sZ{\mathbb Z}
\def\0{\mathbf 0}
\def\-{\mathbf -}
\def\+{\mathbf +}

\def\Sg{\Sigma}
\def\Dt{\Delta}
\def\letterText{\textbf}
\def\a{\letterText{a}}
\def\b{\letterText{b}}
\def\ce{\letterText{c}}
\def\d{\letterText{d}}
\def\s{\letterText{s}}
\def\ee{\letterText{e}}

\def\u{\letterText{u}}
\def\v{\letterText{v}}
\def\w{\letterText{w}}
\def\x{\letterText{x}}
\def\y{\letterText{y}}
\def\z{\letterText{z}}
\def\tu{\tilde{\letterText{u}}}
\def\tv{\tilde{\letterText{v}}}
\def\tw{\tilde{\letterText{w}}}
\def\tx{\tilde{\letterText{x}}}
\def\ty{\tilde{\letterText{y}}}
\def\tz{\tilde{\letterText{z}}}

\def\e{\textit{e}}
\def\f{\textit{f}}
\def\g{\textit{g}}
\def\h{\textit{h}}
\def\i{\textit{i}}
\def\j{\textit{j}}
\def\k{\textit{k}}
\def\l{\textit{l}}
\def\m{\textit{m}}
\def\n{\textit{n}}
\def\p{\textit{p}}
\def\q{\textit{q}}
\def\r{\textit{r}}
\def\t{\textit{t}}
\def\Ce{\textit{C}}
\def\Eq{\textit{E}}
\def\Ie{\textit I}
\def\J{\textit{J}}
\def\L{\textit{L}}
\def\N{\textit{N}}
\def\R{\textit{R}}
\def\S{\textit{S}}
\def\al{\alpha}
\def\be{\beta}
\def\ga{\gamma}
\def\dt{\delta}
\def\bfdt{{\pmb\delta}}
\def\sg{\sigma}
\def\om{\omega}
\def\*{\cdot}

\section{Разбор ДР2 (2013) --- БГС с почти единичной экспонентой всех длинных факторов}\label{SW2}

В данной части автор, почти без изменений копирует некоторые части текста из ДР2, но с некоторыми пояснениями.
Важные изменения будут отмечены каким либо образом (например, в сносках или выделением красным/серым цветом или др. способом).
Точнее, красным будут исправления, а серым дополнения.
Дополнительные пояснения представлены обычным текстом.
Обозначения букв и слов будут отличаться от правил обозначения предыдущей части.
Здесь, как и в разборе ДР1 будет некоторый рефакторинг текста
(т.е. изменения без потери смысла, может с некоторой редакцией того, что подразумевается).

\paragraph{\large Неформальное дополнение.}

Главные леммы (называемые теоремами в ДР2) не перепроверены т.к. в их верности у автора нет сомнений.
В следующих версиях автор планирует разобрать эти леммы с пояснениями и обобщениями.

Немного философии обобщения утверждений.
Обобщением утверждения является такое утверждение, из которого вытекает изначальное.
Т.к. любое утверждение можно рассматривать как импликацию (как минимум условиями являются аксиомы),
то обобщением является как ослабление условий так и добавление (усиление) следствий.
В данной работе представлены оба вида обобщений.

\subsection {Определения и обозначения}

\subsubsection{Общие обозначения}

Курсивными буквами $e$, $f$, $i$, $j$, $m$, $n$, $k$, $p$, $s$, $L$ и $E$ (иногда с индексами) обозначим целые числа. $\eps$ --- рациональное число.
Жирными $\a$, $\b$, $\ce$, $\d$ и $\s$ обозначаются (не пустые) буквы.
$\u, \v, \w, \x, \y, \z$ обозанчают слова.
Символом $\lambda$ обозначим пустое слово.
Символы $\E$, $\V$, $\sV$ и $\Dt$ (иногда с индексами) обозначают некоторые множества или семейства множеств.
$\bfdt$ с индексом будет функцией определяющей множество, $\dt_{-}$ или $\dt$ без индекса --- число.
$\phi$ и $\varphi$ --- отображения.

Буква слова $\w$, стоящая в позиции с номером (или просто в позиции) $\i$, обозначается $\w[\i]$.
Очевидное свойство $|\x|+|\y|=|\x\y|$.
Количество букв $\a$ в слове $\w$ обозначается как $|\w|_\a$.
$\ex(\w)=|\w|-\per(\w)$.
Слова $\u$ и $\v$ будем называть эквивалентными и обозначать $\u\sim\v$,
если существует морфизм (или перестановка букв алфавита) $\varphi\colon\A_\n\to\A_\n$, что $\u=\varphi(\v)$.

\hyperlink{contents}{$\upuparrows$}

\subsubsection {Обобщённые $\D_{3, n}^\eps${--}ГС и их свойства --- $(\L, \eps, \E)${--}ГС}
Пусть $\n\ge5$, $\eps\ge 0$, $\L>1$, $\E\subseteq\sN$.
Индексы в данной части по умолчанию будут начинаться с 1.

\paragraph{$(\L, \eps, \E)${--}граничное слово:}
Обозначим характеристическую функцию множества $\mathbb{N}\setminus\E$ через $\chi_\x=1-\chi_\E(|\x|)$.
Граничное слово $\w\in\A_\n^*$ назовём $(\L, \eps, \E)${--}граничным (обозначим как $(\eps, \E)${--}ГС), если для любого фактора $\x\y\x$ слова $\w$ выполняется следующая импликация
\begin{gather} \label{epsE}
|\x|\notin\E\longrightarrow
\tfrac{|\x\y\x|+\eps}{|\x\y|}\le\tfrac{\n}{\n-1}
\text{ или, что тоже самое, }
\tfrac{|\x\y\x|+\eps\*\chi_\x}{|\x\y|}\le\tfrac{\n}{\n-1}.
\end{gather}

\begin{note}\label{1inE}
	Если $(\eps, \E)${--}ГС $\w$ при $\eps>0$ имеет длину $|\w|>n$, то $1\in\E$ т.к. в таком слове существует подслово имеет длину $\n$ и период $\n-1$.
\end{note}

\begin{note}\label{epsE:per>rep(n-1)}
	Пусть $\w\in\T_n$, $\eps\le\frac{1}{n-1}$ и $\x\y\x$ --- фактор $\w$, тогда, если $\frac{|\x\y\x|+\eps}{|\x\y|}>\rt(\n)$, то $|\x\y|=|\x|\*(\n-1)$.
\end{note}
\begin{proof}
	Т.к. $\w$ --- гранично, то $|\x\y|\ge|\x|\*(\n-1)$.
	Тогда достаточно доказать, что при $|\x\y|\ge|\x|\*(\n-1)+1$ выполняется неравенство $\frac{|\x\y\x|+\eps}{|\x\y|}\le\rt(\n)$.
	$$
	\frac{|\x\y\x|+\eps}{|\x\y|}\le\frac{(\n-1)|\x|+1+|\x|+\frac{1}{\n-1}}{|\x|(\n-1)+1}=\frac{\n|\x|(\n-1)+\n}{(\n-1)(|\x|(\n-1)+1)}=\rt(\n).
	$$
\end{proof}

\begin{note}\label{epsE:E=1}
	Пусть $\w\in\T_n$, $\eps\le\frac{1}{n-1}$ и $\x\y\x$ --- фактор $\w$, $|\x|=2$, тогда $\frac{|\x\y\x|+\eps}{|\x\y|}\le\rt(\n)$.
\end{note}
\begin{proof}
	При $|\x|=2$ любое граничное слово (при $\n\ge5$) имеет период $|\x\y|\ge2(\n-1)+1$.
	Откуда по Замечанию~\ref{epsE:per>rep(n-1)} получаем требуемое.
\end{proof}

Замечание~\ref{epsE:E=1} означает, что множество всех $(\eps, \E)${--}ГС совпадает с множеством
всех $(\eps, \E\setminus\{2\})${--}ГС, при $\eps\le\frac{1}{\n-1}$.

\paragraph{Расширяемый фактор:}
Фактор $\x\a\y\b\x$ с корнем $\x\a\y\b$ назовём расширяемым в $\w$, если
$\b\x\a\y\b\x\subseteq\w$ или $\x\a\y\b\x\a\subseteq\w$.

\paragraph{$(|\x|, \eps)${--}экспоненциальное слово:}
Фактор $\x\y\x$ слова $\w\in\T_n$, 
для которого выполняется неравенство
$\frac{|\x\y\x|+\eps}{|\x\y|}\le\frac{\n}{\n-1}$,
будем называть $(|\x|, \eps)${--}экспоненциальным (обозначать $(|\x|, \eps)${--}ЭС).

\begin{note}\label{def:equiv}
	Если фактор $\x\y\x\subseteq\w\in\T_n$ не $(|\x|, \eps)${--}ЭС, где $\eps\le\frac{1}{\n-1}$, тогда\par
	$1$ $\per(\x\y\x)=|\x\y|$ т.е. $\ex(\x\y\x)=|\x|$;\par
	$2$ $\x\y\x$ не расширяем в $\w$ {\cGr(даже при любых $\eps\le1$)};\par
	$3$ $|\x\y|=|\x|\*(\n-1)$;\par
	$4$ Если слово $\u\v\u\in\T_n$ не $(|\u|, \eps)${--}ЭС и $|\x|=|\u|$, то $|\y|=|\v|$.
	\begin {proof}
	Утверждения 1 и 2 легко проверить от противного.
	3-е утверждение вытекает из Замечания~\ref{epsE:per>rep(n-1)}.
	Для доказательства 4 посчитаем разницу длин $\y$ и $\v$. Б.О.О. считаем, что $|\y|\ge|\v|$.
	Пользуясь равенством $|\x|$ и $|\u|$, условием не $(|\y|, \eps)${--}ЭС $\x\y\x$ и граничностью слова $\u\v\u$ получим
	\begin{gather} \notag
	|\y|-|\v|=|\x\y|-|\u\v|<\frac{|\x\y\x|+\eps}{\rt(\n)}-\frac{|\u\v\u|}{\rt(\n)}\le\frac{|\y|-|\v|+\tfrac{1}{\n-1}}{\rt(\n)}
	\end{gather}
	откуда $|\y|-|\v|<1$ т.е. $|\y|=|\v|$.
	
	{\cGr
		Дополним доказательством для пункта (2) ОП.
		БОО пусть $\x\y\x$ расширяем в $\w$, но по условию $\w\in\T_n$ тогда
		$$
		\rt(n)
		\ge\frac{|\x\y\x|+1}{|\x\y|}
		\ge\frac{|\x\y\x|+\eps}{|\x\y|}
		$$
		значит $\x\y\x$ --- $(|\x|,\eps)${--}ЭС, что противоречит условию.
		Т.е. для пункта (2) достаточно чтобы $\eps\le1$.
		
		Доказательство для пункта (1) ОП.
		БОО пусть $\ex(\x\y\x)>|\x|\ge1$ т.е. $\per(\x\y\x)<|\x\y|$, но по условию $\w\in\T_n$ тогда
		необходимо $\per(\x\y\x)\le|\x\y|-|\x|=|\y|$ для $n\ge2$ т.к. иначе $\lexp(\x\y\x)>2$.
		Тогда получим противоречивое неравенство
		$$
		\rt(n)
		\ge\frac{|\x\y\x|}{\per(\x\y\x)}
		\ge\frac{|\x\y\x|}{|\y|}
		>\frac{|\x\y\x|+|\x|}{|\x\y|}
		\ge\frac{|\x\y\x|+\eps}{|\x\y|}
		>\rt(n)
		$$
		
		В доказательстве пункта (4) достаточно обусловить $\eps\le(\rt(n)-1)$, чтобы обобщить его для $n<5$
	}
	\end {proof}
\end{note}

{\cGr
	Для более ясного понимания, что значит $\u$ не $(\ex(\u),\eps)${--}ЭС в контексте $\w\in\T_n$ при $\eps\le\rt(n)-1=\frac{1}{n-1}$,
	из замечания \ref{def:equiv}(3) можно извлечь,
	что, если слово $\v$ с периодом меньше хотябы на 1 (т.е. $\per(\v)\le\per(\u)-1$)
	с тем же повтором (т.е. $\ex(\v)=\ex(\u)\ge1$ и $|\v|\le|\u|-1$),
	то $\v\not\in\T_n$ (т.е. в $\w$ таких факторов нет).
	
	%
	%
	
	Если отсортировать факторы по длине корня при одинаковом повторе,
	тогда будет ровно $n$ классов факторов с общей длиной повтора $e$, являющихся граничными.
	Т.е. классы можно отсортировать по длинам факторов и выписать первые $n$ из них
	$$
	en, en+1, en+2,..., en+(n-1)
	$$
	Первый класс из них имеет экспоненту $\frac{en}{en-e}=\rt(n)$.
	Каждый следующий класс имеет длину на $1$ больше предыдцщего при фиксированном $n$.
	Тогда и экспонента, из-за общей длины корня имеет линейный рост с линейным ростом длин факторов,
	увеличиваясь на константу $\frac{1}{n-1}$ в каждом следующем классе.
	
	Заметьте, что следующий $en+n$ это уже наименьшая длина граничного фактора, допускающего длину повтора $e+1$.
	Проверим: $\frac{en+n}{en+n-(e+1)}=\frac{n(e+1)}{(e+1)(n+1)}=\rt(n)$.
	
	{\bf Т.о. мы можем выделить $n$ классов граничных факторов.
		Будем их называть по номерам, где первый класс имеет длину $en$,
		для любого фактора с повтором длины $e$ при любом $e\in\mN_0$.}
	
	{Т.е. если взять $\u\subset\w\in\T_n$ и $0<\eps\le\rt(n)-1$,
		то фактор $\u$ не $(\ex(\u),\eps)${--}ЭС если, и только если $\exp(\u)=\rt(n)$.}
	
	В более общем случае 
	{\bf если $k\in\mN$,
		то фактор $\u$ не $(\ex(\u),k\eps)${--}ЭС если, и только если он из первых $k$ классов.}
	
}

\paragraph{Множество $\E_\w(\eps, \E, m)$ слов:}
Множество факторов $\x\y\x$ слова $\w$, не являющихся $(|\x|, \eps)${--}ЭС,
у которых множество позиций правого повтора $\x$ содержит $m$,
при этом $|\x|\in\E$, обозначим $\E_\w(\L, \eps, \E, m)$ т.е.
\begin{multline*}
\E_\w(\eps, \E, m)=\{\x\y\x=\w[i{+}1, ..., i{+}|\x\y\x|]\colon|\x|{\in}\E,\ \x\y\x\text{ не }(|\x|, \eps)\text{{--}ЭС},\ 
i, m{\in}\sN_0,\ i{+}|\x\y|<m\le i{+}|\x\y\x|\le|\w|\}
\end{multline*}

{\cGr
	Это множство над $(\eps,\E)${--}ГС для любого $m$ не превосходит $|\E|$.
	Об этом будет более подробное 
}

\begin{prop}\label{def:unic}
	Пусть $\w\in\T_n$ и $\eps\le\tfrac{1}{n-1}$,
	тогда для любого $e\in\E$ существует не более одного элемента $\u\v\u\in\E_\w(\eps, \E, \m)\cup\E_\w(\eps, \E, m+1)$ такого,
	что $|\u|=e$ и $\u\v\u$ не $(|\u|, \eps)${--}ЭС.
	\footnote{
		Условие <<$\u\v\u$ не $(|\u|, \eps)${--}ЭС>> можно избыточно 
		т.к. вытекает из определения $\E_\w(\eps, \E, \m)$ и $\E_\w(\eps, \E, m+1)$.
		А условие $|\u|=e$ можно заменить на условие $|\u\v|=\per(\u\v\u)$,
		но и это избыточно т.к. только единственным образом фактор $\u\v\u\in\T_n$ может быть не $(|\u|, \eps)${--}ГС
		(по зам-ю \ref{def:equiv}.1).
	}
\end{prop}
\begin{proof}
	Допустим существование двух различных элементов $\u_1\v_1\u_1$ и $\u_2\v_2\u_2$ из $\E_\w(\eps, \E, m)\cup\E_\w(\eps, \E, m+1)$ что $|\u_1|=|\u_2|\in\E$ и они не $(|\u_1|, \eps)${--}ЭС.
	Тогда выполняются условия замечания~\ref{def:equiv}.
	По замечанию~\ref{def:equiv}(4) получаем $|\v_1|=|\v_2|$, откуда $|\u_1\v_1|=|\u_2\v_2|$.
	А т.к. факторы $\u_1\v_1\u_1$ и $\u_2\v_2\u_2$ различны и правые $\u_1$ и $\u_2$ пересекаются в позиции $m$ или стоят рядом, то эти слова расширяемы в $\w$, что противоречит замечанию~\ref{def:equiv}(2).
\end{proof}

{\cGr
	Не трудно заметить, что даже если $|\v_1|$ и $|\v_2|$ отличаются на 1,
	то существует квадрат в объединении левых повторов.
	Тогда, достаточно ограничить $\eps$ сверху б$\acute{\text{о}}$льшим числом,
	только нужно доказать аналог замечания~\ref{def:equiv}(4) для случая $||\y|-|\v||=1$.
	
	С помощью свойства \ref{def:unic} мы сможем сильно сократить необходимый набор $\V${--}образов
	для построения $(\eps,\E)${--}ГС до $2n$ слов.
	А так же, это свойство поможет породить сразу экспоненциальное множество таких слов.
}

Добавим обозначения образа и нерасширяемого фактора в образе
\begin{itemize}
	\item 
	В контексте разбора ДР2 слово со штрихом будет означать образ этого слова (например,
	$\x'$ или $(\x)'$ --- образы одного и того же слова $\x$).
	Слово $\w'=\varphi(\w)$ будем называть \textit{образом слова} $\w$.
	Через $\varphi_\k(\w)$ обозначим образ $\k$-й буквы слова $\w$, определённого обобщённой подстановкой $\varphi$.
	Образ одной буквы назовём \textit{символьным образом}.
	
	Так же, для любого слова $\w=\x\y...\z$ образы его частей $\x$, $\y$, ..., $\z$ будем обозначать  $\x'=\varphi_1(\w)...\varphi_{|\x|}(\w)$, $\y'=\varphi_{|\x|+1}(\w)...\varphi_{|\x\y|}(\w)$, ..., $\z'=\varphi_{|\w|-|\z|+1}(\w)...\varphi_{|\w|}(\w)$.
	
	\item 
	Пусть $\x\y\x\in\F(\w)$\footnote{Напомним, что $\F(\w)$ --- множество всех факторов слова $\w$},
	тогда нерасширяемое слово в $\w'$, содержащее образы (всех) подслов $\x\y\x$ и никакие другие,
	будем обозначать $\tx\ty\tx$ причём $\tx$ содержит только образы (всех) подслов $\x$.
\end{itemize}

{\cGr
	Условия Леммы \ref{l2} основаны на $\D^\eps_{3, n}${--}ГС.
	Но $\D^\eps_{\P, n}${--}ГС имеют <<усиленную>> экспоненту в зависимости от длины корня $p$ факторов по всем $p\in\P$.
	В ДР2 мы определяем $(\eps,\E)${--}ГС в зависимости от длины повтора (наибольшего) $e$ факторов по всем $e\in\E$,
	что удобнее для анализа свойств ГС.
	Впервые идея перерйти на зависимость от длины корня к длине повтора
	была предложена нашим НР в его редакции ($\eps${--}граничные слова,
	описанные им в совместной с автором научной статье) авторского метода решения
	на основе текста для лекций, самих лекций и ДР1 автора.
	
	Строго по определению $\eps${--}ГС не одно и то же, что и слово из $\D_{3,n}^\eps\cap\T_n$.
	Но они эквивалентны при $\eps\in[0,\frac{2}{n-1}]$.
	
	$\D_{3,n}^\eps${--}ГС является $\eps${--}ГС
	т.к. повторы длины $\ge3$ в ГС появляются только в факторах с периодом не менее $3n-3$.
	
	В обратную сторону.???перепроверить???
	Очевидно, что $\eps${--}ГС из $\T_n$.
	В нашей Лемме \ref{l2} достаточно рассмотреть только $\eps\in[0,\frac{2}{n-1}]$ (что вытекает из \ref{l2p0e}).
	Тогда достаточно доказать, что факторы с повторами длины $\ge3$ имеют период не менее $3n-1$,
	что мы уже выяснили в замечании \ref{nt_3nm1}.
	Тогда $\eps${--}ГС из $\D_{\ge3n-1,n}^\eps$.
	Используя замечание \ref{nt_D3nm1}, получим,
	что $\eps${--}ГС из $\D_{\ge3n-3,n}^\eps\cap\T_n$ при $\eps\in[0,\frac{2}{n-1}]$.
	
	
	Поэтому, 
	Лемма \ref{l2} для $\D_{3,n}^\eps${--}ГС эквивалентна Лемме \ref{l2} для $\eps${--}ГС.
	
}

\hyperlink{contents}{$\upuparrows$}

\subsubsection {Обобщение $\k${--}значной подстановки --- $(\L, \eps, \E, \Dt_\n)${--}подстановка ($(\L, \eps, \E, \Dt_\n)${--}П)}
Добавим вспомогательные объекты
\begin{itemize}
	\item $\V_\s\subset\A_n^\L$ --- множество слов длины $\L$ по всем $s\in\A_n$.
	{\cGr
		Т.е. образы для каждой буквы $s\in\A_n$.
		Условие равенства длинн можно и опустить.
	}
	\item $\sV=\bigcup_{a\in\A_n}\V_\a$.
	{\cGr
		Это, просто, все образы. Здесь нет условия раличия образов для разных букв.
	}
	\item $\sV_n=\{\V_\a\colon a\in\A_n\}$
	\item $\bfdt_\b:\V_\a\to2^{\V_\b}$ по всем $a\ne b\in\A_n$.
	{\cGr
		Можно было бы и не обуславливать $a\ne b$ для общности.
		$\bfdt_\s(\v)$ будет использоваться как подмножество допустимых образов из $\V_\s$,
		которые могут следовать сразу за некоторым образом $\v\in\mV$.
		Т.е. $\bfdt_\s(\v)\subseteq\V_\s$ при любом $\v\in\mV$.
		В нашей 3-хзначной подстановке использовался частный случай $\bfdt_\s(\v)=\V_\s$
		для любого $\v\in\mV\setminus\V_\s$.
		Это обобщение полезно, когда таблица проверки пар на условие \ref{l1c2} и \ref{l2c2} разрежена.
		
		Для лучшего понимания дальнейшего, можно считать,
		что $\bfdt_\b(\v\in\V_\a)=\V_\b$ и $\V_\a\cap\V_\b=\emptyset$ для любых $a\ne b\in\A_n$
	}
	\item $\Dt_n=\{\bfdt_\a\colon a\in\A_n\}$.
	{\cGr
		Будем считать, что $\bigcup_{\u\in\V_\b, b(\ne a)\in\A_n}\bfdt_\a(\u)=\V_\a$.
		Т.е. для любоо образа $\v\in\V_\a$ найдётся буква $b\ne a$, что для некоторого $\u\in\V_\b$ будет $\v\in\bfdt_\a(\u)$.
	}
	\item $\dt_{-}=\min\{|\bfdt_\a(\v)|\colon\v\in\sV\setminus\{\V_\a\}, a\in\A_n\}$.
	{\cGr
		В нашей 3-хзначной подстановке $\dt_{-}=3$
	}
\end{itemize}
Добавим условие по умолчанию:
\footnote{
	Данное условие не было добавлено в ДР2.
}
\begin{gather} \label{dt>E}
|\bfdt_\a(\v)|>|\E|\text{, где }\v\in\V_\b\text{ для всех }a\ne b\in\A_n,
\text{ или, что то же самое }
{\cRe \dt_{-}>|\E|}.
\end{gather}
Из этого условия и замечания~\ref{1inE} следует, что мощность $\sV$ не меньше $2n$.

Назовём \textit{обобщенной подстановкой} $\varphi$:

$\varphi(a)\in\V_\a$;

Множества $\V_\a$ для разных $a$ не пересекаются;

$\varphi$ применяется к любому слову побуквенно слева направо;

существует детерминированный алгоритм, выбирающий очередное значение $\varphi(a)$, возможно на основе предыдущих выбранных значений.



Возмём произвольное $(\eps, \E)${--}ГС $\w$.
Тогда определим
\paragraph{$(\L, \eps, \E, \Dt_\n)${--}подстановка:}
Обобщённую подстановку $\phi$ будем называть $(\L, \eps, \E, \Dt_n)${--}подстановкой ($(\L, \eps, \E, \Dt_n)${--}П),
если выполняются следующие условия: для любых $1<\j\le|\w|$, $1\le r\le|\w|$ и $0\le i\le|\w|-r$
\footnote{Пояснения и ограничения для $i$ и $r$ в ДР2 были пропущены, но не трудно установить их ограничения по контексту.}
\begin{description}
	\item[(M1)\label{m1}]
	$\phi_\j(\w)\in\bfdt_{\w[j]}(\phi_{\j-1}(\w))$,
	где $\bfdt_{\w[j]}\in\Dt_n$.
	\footnote{
		В данном условии в ДР2 $\w[j]$ частично было заменено на $a$.
		Здесь представлен некоторый рефакторинг данного условия.
		Напоминаем, что $j$ это  позиция букв, начинающиеся с 1.
		Одно из подразумеваемых значений этого условия --- непустота $\bfdt_{\w[k]}(\phi_{k-1}(\w))$
	}
	\item[(M2)\label{m2}]
	Выполняется импликация ---
	если фактор $\u\v\u=\w[\i+1, ..., \i+r]$ не $(|\u|, \eps)${--}ЭС и $|\u|\in\E$,
	
	\hfill тогда существует позиция
	$m\in\{\i+1, ..., \i+|\u|\}$, в которой $\phi_m(\w)\ne\phi_{m+|\u\v|}(\w)$.
	\footnote{
		Во избежание путаницы уточним,
		что в данном условии $m$ установлено для левого повтора, в отличие от констекста с $\E_\w(\eps, \E, \m)$.
	}
	
\end{description}

Идея отображения проста --- если повторяющееся подслово достаточно короткое, то и в образе этого подслова из-за малых общих префиксов и суффиксов символьных образов будет короткий повтор.
Если же повтор слишком длинный, то появляется необходимость <<разбить>> образ повтора для гарантированного сохранения граничности.

{\cGr
	\ref{m1} это условие непустоты $\bfdt_{\w[j]}(\phi_{\j-1}(\w))$
	т.е. возможность продолжать отображение $\w$ в $\phi(\w)$ слева направо
	после отображения $(j-1)$-й буквы в $\w$.
	А так же это условие,
	что образ последующий буквы $s$ определяется функцией/правилом $\bfdt_\s$
	от образа предыдущей буквы.
	В этом условии не учитывется полный контекст,
	а только локально --- соседним слева $\V${--}образом.
	Это условие можно заменить на $\phi_{j}(\w)\in\V_{\w[j]}$,
	если таблица проверки различных пар \ref{l2c2} без пробелов и размера $2n$.
	
	\ref{m2} это гарантия различия образов, хотябы одной пары <<синхронных букв>>
	(находящихся в разных повторах $\u$ с одинаковым отступом ($m-1$) слева в них,
	т.е. разница их позиций в $\w$ равна $|\u\v|$)
	в слишком длинных повторах.
	Заметьте, что это условие не исключает случая <<не $(|\u|, \eps)${--}ЭС и $|\u|\not\in\E$>>.
	Оно обуславливает только импликацию.

	$\phi$ определён только на $(\eps,\E)${--}ГС.
	Но достаточно определить и для простых граничных $\w$.
	И само условие \ref{m2} учитывает возможность не $(|\u|, \eps)${--}ЭС при $|\u|\not\in\E$.
}

Для оценки роста числа различных слов, порождаемых обобщённой подстановкой
$(\L, \eps, \E, \Dt_\n)${--}П над словом $\w$,
оценим эту функцию (роста) для некоторого подкласса этих слов. Обозначим
\begin{itemize}
	\item
	$\E_\m=\E_\w(\eps, \E, \m)$ для $\m\in\sN$.
	{\cGr
		$\E_m$ --- множество факторов $\v\subseteq\w$ с предельной экспонентой
		(т.е. $\exp(\v)=\rt(n)$),
		содержащие позицию $m$ в правом повторе
		(позиция, конечно, в глобальном контексте $\w$,
		а не относительно начала $\v$ или его повтора).
		Для нас (наших целей) это все факторы в $\w$, содержащие $m$ в правом повторе,
		которые ещё являются $\D_{\ge 3n-3, n}^\eps${--}ГС, но их образ (при 1{--}значной подстановке) уже нет.
	}
	\item
	$\I_\w(\E', \m)$ будет означать множество образов букв\sout{\cRe ы}
	\footnote{
		Не удачная формулировка в оригинале.
	}
	$\w[\m-|\u\v|]$ для каждого не $(|\u|, \eps)${--}ЭС $\u\v\u\in\E_\m\cap\E'$,
	где $\E'$ --- некоторое множество слов.
	{\cGr
		$\I_\w(\E',m)$ это $\V${--}образы букв в левых повторах в позиции $(m-\per(\w))$ для всех $\w\in\E'\subseteq\E_m$.
		Смысл $\I_\w(\E_m,m)$ в том, что это <<исключаемые>> $\V${--}образы,
		которые достаточно (но не необходимо) исключить из
		$\bfdt_{\w[m]}(\phi_{m-1}(\w))$,
		а оставшиеся использовать для отображения $m$-й буквы в слове $\w$,
		для сохранения хорошей экспоненты образов всех слов из $\E_m$.
		
		$\I_\w(\E_m,m)$ это $\V${--}образы,
		ВОЗМОЖНО, не допустимые в позиции $m$.
		Но могут быть и допустимы,
		если уже есть другая буква в правых повторах тех же факторов $\E_m$,
		которая отображается в не <<исключаемый>> образ.
		
	}
\end{itemize}

{\cGr
	\begin{note}\label{nt:IwInIw}
		$\I_\w(\E'_1,m)\subseteq\I_\w(\E'_2,m)$ при любых $\E'_1\subseteq\E'_2$.
		Откуда $\I_\w(\E_k\setminus\E_{k-1},k)\subseteq\I_\w(\E_k\setminus\E',k)$ при любых $\E'\subseteq\E_{k-1}$.
	\end{note}
	
	\begin{note}\label{nt:E>Iw}
		Испольуя свойство \ref{def:unic}, $|\E|\ge|\I_\w(\E_k,k)|$ для любого $k\in\mN$ и $\eps\le\frac{1}{n-1}$.
		Т.е. если $\E=\{1\}$ и $\dt_{-}\ge2$, то $|\bfdt_{\w[k]}(\phi_{k-1}(\w))\setminus\I_\w(\E_k,k)|>0$.
	\end{note}

{
}

}

\hyperlink{contents}{$\upuparrows$}

\subsubsection {Ослабление условий Леммы \ref{l2} из ДР1}

{\cGr
	Здесь мы предложим более слабые граничения на
	условия Леммы \ref{l2},
	заменив $\D^\eps_{3, n}${--}ГС на $(\eps,\E)${--}ГС.
	
	Вариаций ослабления данных условий много и они не все совместимы.
	Здесь предлагается один из вариантов обобщения.
}

Пусть $\phi$ --- $(\L, \eps, \E, \Dt_\n)${--}П,
$\x'=\phi(\x)$ и $\y'=\phi(\y)$ для некоторых $(\eps, \E)${--}ГС $\x$ и $\y$,
тогда $\Dt(\eps, \E, \Dt_\n)(\x', \y')$ --- наименьшее число $\k\ge0$,
что существует $(\eps, \E)${--}ГС $\w$ длины $\k+|\x\y|$,
в котором $\x'\in\PREF(\phi(\w))$ и $\y'\in\SUFF(\phi(\w))$ по всем $\phi$.
Если такого числа не существует, то $\k=\infty$.
В контексте разбора ДР2 $\Dt(\eps, \E, \Dt_\n)(\x', \y')$ будем обозначать $\Dt(\x', \y')$.
Например, если $\u\in\V_\a, \v\in\bfdt_\b(\u)$, то $\Dt(\u, \v)=0$.
\\
$\lcp(\u, \v)=\max\{|\x|\colon\x\in\PREF(\u)\cap\PREF(\v)\}$,
$\lcs(\u, \v)=\max\{|\x|\colon\x\in\SUFF(\u)\cap\SUFF(\v)\}$,\\
$\lcp(\sV_\n)=\max\{\lcp(\u, \v)\colon\u\in\V_\a, \v\in\V_\b, \a\ne\b\in$ $\A_\n\}$,\\
$\lcs(\sV_\n)=\max\{\lcs(\u, \v)\colon\u\in\V_\a, \v\in\V_\b, \a\ne\b\in$ $\A_\n\}$,\\
Напомним $\chi_\x=1-\chi_\E(|\x|)$ --- характеристическая функция.\\
$\sg_{\sV_\n}(\u_1, ..., \u_\m)$ --- минимальное число букв $\a_1, ..., \a_\k\in\A_\n$, что $\{\u_1, ..., \u_\m\}\subseteq\V_{\a_1}\cup...\cup\V_{\a_\k}$.
В контексте разбора ДР2 будем считать, что $\lambda\not\in\F(\x)$ для любого слова $\x\in\A_\n^*$ (в частности $\lambda\not\in\SUFF(\x)\cup\PREF(\x)$).
\smallskip
\parskip -10pt

\paragraph{$(\L, \eps, \E)${--}\textit{согласованные} множества:}
Семейство $\Dt_\n$ назовём $(\L, \eps, \E)${--}\textit{согласованным} (обозначим как $(\L, \eps, \E)${--}С),
если для любых $\a\ne\b\ne\ee\ne\a$ и $\ce\ne\d$ из $\A_\n$ и $\u_1\in\V_\a$, $\v_1\in\bfdt_\b(\u_1)$,
$\w_1\in\bfdt_\ee(\v_1)$, $\u_2\in\V_\ce$ и $\v_2\in\bfdt_\d(\u_2)$ выполняются следующие условия:
\parindent 0 mm
\parskip 2pt

(uf1.1) Если $\u\in\sV$,
тогда $\per(\u)=|\u|(=\L)$;
\parskip 5pt

(uf1.2) Если $\u\in\V_\a$, $\v\in\V_\b$,
тогда $\SUFF(\u)\cap\PREF(\v)=\emptyset$;
\parskip 5pt

(uf2.1.1) Если $\u_1=\al\x\y\z\x$, $\v_1=\y\al$,
тогда $\x\y\z\x\y$ --- $(|\x\y|, \eps\*\chi_{\x\y})${--}ЭС;
\parskip 3pt

(uf2.1.2) Если $\sg_{\sV_\n}(\u, \v)=2$, $\u=\al\x\y$, $\v=\z\x\al$,
тогда $\frac{\Dt(\u, \v)\*\L+|\x\y\z\x|+\eps\*\chi_\x}{\Dt(\u, \v)\*\L{+}|\x\y\z|}\le\rt(\n)$;
\parskip 3pt

(uf2.1.3) Если $\u_1=\al\x$, $\v_1=\y\z\x\y\al$, тогда $\x\y\z\x\y$ --- $(|\x\y|, \eps\*\chi_{\x\y})${--}ЭС;
\parskip 3pt

(uf2.2.1) Если $\sg_{\sV_\n}(\u, \v, \w)$=$3$, $\u$=$\al\x$, $\v$=$\y\z_1$, $\w$=$\z_2\x\y\al$,
тогда $\frac{\Dt(\u\v, \w)\*\L+|\x\y\z_1\z_2\x\y|+\eps\*\chi_{\x\y}}{\Dt(\u\v, \w)\*\L+|\x\y\z_1\z_2|}\le\rt(\n)$;
\parskip 3pt

(uf2.2.2) Если $\sg_{\sV_\n}(\u_1, \v_1, \w_1)$=$3$, $\u_1$=$\al\x$, $\v_1$=$\y\z\x$, $\w_1$=$\y\al$,
тогда $\x\y\z\x\y$ --- $(|\x\y|, \eps\*\chi_{\x\y})${--}ЭС;
\parskip 2pt

(uf2.2.3) Если $\sg_{\sV_\n}(\u, \v, \w)$=$3$, $\u$=$\al\x\y\z_1$, $\v$=$\z_2\x$, $\w$=$\y\al$,
тогда $\frac{\Dt(\u, \v\w)\*\L+|\x\y\z_1\z_2\x\y|+\eps\*\chi_{\x\y}}{\Dt(\u, \v\w)\*\L+|\x\y\z_1\z_2|}\le\rt(\n)$;
\parskip 2pt

(uf2.3) Если $\sg_{\sV_\n}(\u_1, \v_1, \u_2, \v_2)$=$4$ и $\u_1$=$\al\x$, $\v_1$=$\y\z_1$, $\u_2$=$\z_2\x$, $\v_2$=$\y\al$,
тогда 
${\frac{\Dt(\u_1\v_1, \u_2\v_2)\*\L+|\x\y\z_1\z_2\x\y|+\eps\*\chi_{\x\y}}{\Dt(\u_1\v_1, \u_2\v_2)\*\L{+}|\x\y\z_1\z_2|}\le\rt(\n)}$;
\parskip 3pt

(uf3.1) Если $\u_1\ne\u_2$,
тогда $\lcp(\u_1, \u_2)\le\L-\lcs(\sV_\n)-\eps$ и $\lcs(\u_1, \u_2)\le\L-\lcp(\sV_\n)-\eps$;
\parskip 3pt

(uf3.2) Если $\u_1=\al\x$, $\y\z\al=\v_1$, $\u_2=\al\x\y$ и $\z\al=\v_2$, тогда ${\tfrac{\L(\n-1){+}|\x\z|+\eps\*\chi_{\x\y\z}}{\L(\n-1)-|\y|}}\le\rt(\n)$.

\parskip 0pt
\parindent 10 mm

\begin{note}\label{uf:12}
	Из условия (uf1.2) следует, что $\V_\a\cap\V_\b=\emptyset$ для различных $\a$ и $\b$ из $\A_\n$.
\end{note}

\begin{prop}\label{uf:blocks}
	Если $\x$ непустой собственный фактор образа слова $\a\y\b\in\T_n$ при действии $(\L, \eps, \E, \Dt_\n)${--}подстановки,
	с началом в позиции $\s_1$ в $\a'$ и заканчивающийся в $\b'$, то при условии,
	что $\x$ фактор образа слова $\ce\z\d\in\T_n$, с началом в позиции $\s_2$ в $\ce'$ и заканчивающийся в $\d'$, имеем:
	\footnote{Напомним, что буквы и слова со штрихом, мы обозначаем образы этих букв и слов}\par
	$(1)$ $\y=\z$;\par
	$(2)$ Если $\y\ne\lambda$, тогда $\s_1=\s_2$;\par
	$(3)$ Если $\s_1\ne\s_2$, тогда $\{\a, \b\}\cap\{\ce, \d\}\ne\emptyset$.
	\begin {proof}
	Рассмотрим случай $\y=\lambda$ т.е. $\x$ --- фактор $\a'\b'$.
	Предположим, что $\z\ne\y$.
	Тогда $|\z'|=\L|\z|>0$ и $\z'$ --- фактор $\x\subset\a'\b'$ откуда $|\z'|<|\a'\b'|$.
	Если $|\z|>1$, то $|\z'|\ge2\L=|\a'\b'|$ что не возможно.
	Значит слово $\z$ можно рассматривать как букву.
	При этом $\z'$ не совпадает ни с $\a'$ ни с $\b'$, по условию начала и конца $\x$.
	Значит $\z'$ начинается в $\a'$ и заканчивается в $\b'$.
	Тогда $\PREF(\z')\cap\SUFF(\a')\ne\emptyset$ и $\SUFF(\z')\cap\PREF(\b')\ne\emptyset$.
	По (uf1.2) $\a'$, $\z'$ и $\b'$ не могут быть образами различных букв.
	Тогда $\a=\z=\b$, что не возможно для граничных слов $\a\b$ над $\n\ge5$ буквами.
	
	Остаётся случай $\y\ne\lambda$. При $\z=\lambda$ случай сводится к случаю при $\y=\lambda$.
	Тогда $\z\ne\lambda$.
	Предположим, что $\s_1\ne\s_2$.
	Б.О.О. считаем, что $\s_1<\s_2$.
	Тогда $\PREF((\z[1])')\cap\SUFF(\a')\ne\emptyset$ и $\SUFF((\z[1])')\cap\PREF((\y[1])')\ne\emptyset$.
	Отсюда, по (uf1.2) получаем, что $\a=\z[1]=\y[1]$, что опять же не возможно для граничного слова $\a\y\b$.
	Из этого противоречия и следует пункт (2).
	
	Значит при $\y\ne\lambda$ в слове $(\a\y\b)'$ фактор $\x$ начинается с той же позиции $\s$,
	что и в слове $(\ce\z\d)'$.
	Тогда либо $\y'\in\PREF(\z')$ либо $\z'\in\PREF(\y')$.
	При этом, по определённому $\s$, диапазон конечных позиций $\x$ в этих словах однозначно определяет их длины, т.е. $|(\a\y\b)'|=|(\ce\z\d)'|$.
	Откуда $\y'=\z'$ что, с учётом Замечания~\ref{uf:12}, возможно только при $\y=\z$. Т.е. выполняется пункт (1).
	
	В случае $\s_1\ne\s_2$ по пункту (2) получим $\y=\z=\lambda$.
	Тогда, т.к. $\x\ne\lambda$, либо $\PREF(\d')\cap\SUFF(\a')\ne\emptyset$, либо $\SUFF(\ce')\cap\PREF(\b')\ne\emptyset$. Отсюда по (uf1.2) вытекает пункт (3).
	\end {proof}
\end{prop}

\hyperlink{contents}{$\upuparrows$}

\subsection {Основые свойства, леммы и их следствия}

\subsubsection {Свойства $(\L, \eps, \E, \Dt_\n)${--}П --- порождение экспоненциального множества слов}

{\cGr
В данной секции мы считаем, что {\cRe$\dt_{-}>|\E|$}.
Следующее свойство можно считать доказательством индукцией по позиции $k$ в слове $\w$ того факта,
что при условии непустоты $\bfdt_{\a}(\phi_{k-1}(\w))\setminus\I_\w(\E_k\setminus\E_{k-1}, k)$.
Напомним, что номера позиций в ДР2 начинаем с 1 т.е. $k\in\mN$.
}

\begin{prop}\label{morf:step}
	Пусть $\a=\w[k]$ {\cGr при $k\in\mN$} и условия \ref{m1} и \ref{m2} выполняются для всех факторов слова $\w$ из $\E_m$
	по всем $m<k$
	и пусть $\phi_k(\w)\in\bfdt_{\a}(\phi_{k-1}(\w))\setminus\I_\w(\E_k\setminus\E', k)$, где $\E'\subseteq\E_{k-1}$. Тогда \ref{m1} и \ref{m2} выполняется для факторов из $\E_k$.
\end{prop}
\begin{proof}
	По условию, для любого не $(|\u|, \eps)${--}ЭС $\u\v\u\in\E_{k}\cap\E_{k-1}$ при $|\u|\in\E$, уже существует буква, образ которой отличен от образа буквы, стоящей на $|\u\v|$ позиций левее,
	а позиции правого $\u$ содержат некоторый $m<k$.
	Тогда, после фиксации образа буквы $\w[\k]=\a$ любым словом из $\bfdt_\a(\phi_{k-1}(\w))\setminus\I_\w(\E_{k}\setminus\E_{k-1}, k)$,
	выполняется \ref{m2} для любых таких факторов,
	у которых позиции правого повтора содержат номер $k$.
	
	При этом, {\cGr(по зам-ю \ref{nt:IwInIw})} легко понять, что $\bfdt_\a(\phi_{k-1}(\w))\setminus\I_\w(\E_{k}\setminus\E', k)\subseteq\bfdt_\a(\phi_{k-1}(\w))\setminus\I_\w(\E_{k}\setminus\E_{k-1}, k)$
	при любом $\E'\subseteq\E_{k-1}$.
	\ref{m1} выполняется т.к. образ $k$-й буквы выбран из $\bfdt_\a(\phi_{k-1}(\w))$.
\end{proof}

{\cGr
	Для лучшего понимания доказательства следующего свойства \ref{morf:expgr} поясним основные,
	используемые в нём, обозначения и неравенства.

	В свойстве достаточно добавить условие $\dt_{-}>|\E|$.
	Но в этом случае достаточно обусловиться, что $\dt_{-}>|\E_k\cup\E_{k+1}|$ (хотябы периодически для $k$).

	$\min\{|\delta_\c(\v)|:\v\in\delta_\b(\v_\a)\}=f_\c$ --- минимально возможный выбор $\V_\c${--}образов,
	среди следующих за некоторым $\v\in\delta_\b(\v_\a)$.

	Заметим, что $f_c\ge\dt_{-}>|\E|$.

	$(|\bfdt_\b(\v_\a)|-|\I_\w(\E_m,m)|)$ --- возможный(оставшийся) выбор $\V_\b${--}образов, в которые может отобразиться $\w[m]$.

	$f_\c-|\I_\w(\E_{m+1}\sm\E_m,m+1)|$ --- минимально возможный выбор $\V_\c${--}образов,
	в которые может отобразиться $\w[m+1]$.

}

\begin{prop}\label{morf:expgr}
	Пусть $\eps\le\frac{1}{\n-1}${\cGr, $\dt_{-}>|\E|$}, $\w\in\T_n$
	и существует некоторое количество слов, полученных $(\L, \eps, \E, \Dt_\n)${--}подстановкой
	(т.е. удовлетворяющим условиям \ref{m1} и \ref{m2}) над $\m-1$ первыми позициями слова $\w$,
	тогда действие той же (обобщённой) подстановки над следующими двумя буквами $\b\ce=\w[\m, \m+1]$
	увеличивает количество слов как минимум в $(\dt_{-}-|\E_\m|)(|\E_\m|+1)$ раз.
\end{prop}
\begin{proof}
	Можно считать, что выполняется условие свойства \ref{morf:step} до $(\m-1)$-й буквы
	слова $\w$, тогда по этому свойству $\m$-я и $(\m+1)$-я буквы слова $\w$
	могут быть последовательно фиксированы любыми словами $\v_\b$ из
	\\
	$\bfdt_\b(\v_\a)\setminus\I_\w(\E_\m, \m)$ и $\v_\ce$
	из $\bfdt_\ce(\v_\b)\setminus\I_\w(\E_{\m+1}\setminus\E_\m, \m+1)$ соответственно,
	где $\v_\a=\phi_{\m-1}(\w)$, $\b=\w[\m]$ и $\ce=\w[\m+1]$.
	
	Обозначим $\min\{|\bfdt_\ce(\v)|\colon\v\in\bfdt_\b(\v_\a)\}=f_\ce$
	{\cGr ---
		минимально возможный выбор $\V_\c${--}образов,
		среди следующих за некоторым $\v\in\bfdt_\b(\v_\a)$.
		Заметьте, что $f_\c\ge\dt_{-}>|\E|$
	}.
	С учётом $|\I_\w(\E, \s)|\le|\E|$ для любой позиции $\s$ в $\w$, получаем, что число различных вариантов фиксировать образы букв $\w[\m]$ и $\w[\m+1]$, не менее
	\footnote{
		$(|\bfdt_\b(\v_\a)|-|\I_\w(\E_m,m)|)$ --- возможный(оставшийся) выбор $\V_\b${--}образов, в которые может отобразиться $\w[m]$.
		
		$f_\c-|\I_\w(\E_{m+1}\sm\E_m,m+1)|$ --- минимально возможный выбор $\V_\c${--}образов,
		в которые может отобразиться $\w[m+1]$.
	}
	\begin{equation}\notag
	\begin{split}
	E
	&=(|\bfdt_\b(\v_\a)|-|\I_\w(\E_\m, \m)|)\min\{|\bfdt_\ce (\v)|-|\I_\w(\E_{\m+1}\setminus\E_\m,
 \m+1)|\colon\v\in\bfdt_\b(\v_\a)\}\\
	&={\cGr
		(|\delta_\b(\v_\a)|-|\I_\w(\E_m,m)|)(f_\c-|\I_\w(\E_{m+1}\sm\E_m,m+1)|)
	}\\
	&\ge(|\bfdt_\b(\v_\a)|-|\E_\m|)(\f_\ce-|\E_{\m+1}\setminus\E_\m|)
	\ge{\cGr(\dt_{-}-|\E_m|)(f_\c-|\E_{m+1}\sm\E_m|)}
	.\\
	\end{split}
	\end{equation}
	
	Используя свойство \ref{def:unic} получаем $|\E_\m\cup\E_{\m+1}|\le|\E|<\f_\ce$.
	Т.к. для любых множеств $A,B$ выполняется $|B\sm A|=|A\cup B|-|A|$.
	
	Тогда $|\E_{\m+1}\setminus\E_\m|=|\E_\m\cup\E_{\m+1}|-|\E_\m |<\f_\ce-|\E_\m|$, используя целочисленность $\f_\ce-|\E_{\m+1}\setminus\E_\m|\ge|\E_\m|+1$.
	
	Тогда $E\ge(|\bfdt_\b(\v_\a)|-|\E_\m|)(|\E_\m|+1)\ge(\dt_{-}-|\E_\m|)(|\E_\m|+1)$.
\end{proof}

\begin{cor}\label{morf:expgr1}
	Пусть $\eps\le\frac{1}{\n-1}$ и $\w\in\T_n$,
	тогда число различных образов, получаемых действием $(\L, \eps, \E, \Dt_\n)${--}П над словом $\w$,
	через каждые 2 буквы (слова $\w$) увеличивается как минимум в $\dt_{-}$ раз
	при выполнении условия $\dt_{-}>|\E|$.
\end{cor}
\begin{proof}
	Из замечания~\ref{1inE} следует $|\E|\ge1$,
	тогда по условию {\cGr(и целочисленности)} получаем $\dt_{-}\ge|\E|+1\ge2$.
	Тогда по Свойству~\ref{morf:expgr} получим {\cGr увеличение числа образов не менее чем в} $(\dt_{-}-|\E_\m|)(|\E_\m|+1)$ раз
	{\cGr
		любых 2-х соседних букв слова $\w$ (т.е. при любом $1\le m\le|\w|-1$).
		Т.е. для доказательства следствия достаточно убедиться в неотрицательности
		$(\dt_{-}-|\E_m|)(|\E_m|+1)-\dt_{-}=|\E_m|\dt_{-}-|\E_m|(|\E_m|+1)=|\E_m|(\dt_{-}-(|\E_m|+1))$.
		А т.к. $|\E|\ge|\E_m|$, то $\dt_{-}\ge|\E_m|+1$.
		Откуда и следует требуемое $(\dt_{-}-|\E_m|)(|\E_m|+1)$
	}$\ge\dt_{-}$.
\end{proof}

Получаем, что при последовательном фиксировании образов слева направо через каждые $2$ буквы слова $\w$, можно увеличивать число различных образов экспоненциально при $\dt_{-}>1$, не используя дополнительных образов, как было в наших первых вариантах доказательства экспоненциальной гипотезы.

Заметим так же, что никаких ограничений на $\Dt_\n$ не накладывается.

\begin{prop}\label{morf:k-val}
	Пусть $\eps\le\frac{1}{\n-1}$, $\k\ge2$ и $\E=\{1, ..., \k-1\}$.
	Тогда, при некоторых $\sV_\n$ и $\Dt_\n$, $\k${-}значная подстановка над граничным словом
	является элементом класса $(\L, \eps, \E, \Dt_\n)${--}П (т.е. удовлетворяет свойствам \ref{m1} и \ref{m2}).
\end{prop}
\begin{proof}
	Т.к. $|\E|=\k-1$, то достаточно, чтобы $|\bfdt_\b(\v_\a)|=|\V_\b|=\k$ ($\dt_{-}>|\E|$)
	\footnote{это условие обнаружено в ДР2. Т.е. в контексте ДР2 это недостающее условие можно было восстановить.}
	по всем $\a\ne\b\in\A_\n$.
	Тогда \ref{m1} выполняется по условию $\V_\a\subseteq\bfdt_\a(\v)$ по всем $\a\in\A_\n$ и $\v\in\sV\setminus\V_\a$.
	Если фактор $\x\y\x$ слова $\w$ не $(|\x|, \eps)${--}ЭС и $|\x|\in\E$, то $|\x\y|<(|\x|+\eps)\*(\n-1)\le|\x|(\n-1)+1$, что для граничного слова $\x\y\x$ возможно только при $|\x\y|_\a\le|\x|\le\k-1$ для любой буквы $\a$ из $\x$.
	По правилу $\k${-}значной подстановки получаем, что образы первых букв левого и правого $\x$ в слове $\x\y\x$ различны.
	Откуда следует импликация в условии \ref{m2}.
\end{proof}

\hyperlink{contents}{$\upuparrows$}

\subsubsection{Лемма \ref{l4} --- порождении $(\L, \eps, \E)${--}ГС через $(\L, \eps, \E, \Dt_\n)${--}П. РРДГС как следствие леммы}
Здесь мы докажем обобщённый вариант Леммы~\ref{l2}.

\begin{lem}\label{l4}
	Пусть $\eps(\L-1)=\lcs(\sV_\n)+\lcp(\sV_\n)$, $\w$ --- $(\eps, \E)${--}ГС,
	и множество $\Dt_\n$ --- $(\L, \eps, \E)${--}согласованно.
	Тогда любое слово вида $\phi(\w)$, где $\phi$ --- $(\L, \eps, \E, \Dt_\n)${--}подстановка --- $(\eps, \E)${--}ГС.
\end{lem}
\begin{proof}
	Пусть $\tx\ty\tx\subseteq\w’=\phi(\w)$ и $\tx\ty\tx$ не расширяемо в $\w'$.
	Так же считаем, что $|\ty|$ --- наименьшее среди всех $\tx\ty\tx$ с повтором $\tx$.
	Тогда проверим  $\tx\ty\tx$ на условие~\eqref{epsE} т.е. на $(|\tx|, \eps\*\chi_{\tx})${--}ЭС.
	
	Рассмотрим сначала случаи когда хотябы один из факторов $\tx$ содержит образ не менее одной буквы.
	Для сокращения перебора случаев, рассмотрим самый <<худший>>, т.е. будем считать, что\par
	$(**)$ образы букв всегда имеют общие суффиксы и префиксы.\\
	Т.к. слово $\tx\ty\tx$ нерасширяемо, то $\tx$ является фактором образа некоторого $\a\x\b$, с началом в $\a$ и концом в $\b$ и, в нашем случае, $\x\ne\lambda$.
	Тогда по Свойству~\ref{uf:blocks}(1) получаем, что любой $\tx$ содержит этот образ.
	Пусть $\y$ ---  минимальное по длине слово, что $\a\x\y\x\b\subseteq\w$ и $\tx\ty\tx\subseteq\phi(\a\x\y\x\b)$.
	Заметим, что $|\y|\ge1$ т.к. иначе в $\w$ существует квадрат $\x\x\ne\lambda$ что противоречит граничности слова $\w$.
	Из нерасширяемости и $(**)$ следует, что\par
	$(*)$ $\a\ne\y[|\y|]$ и $\b\ne\y[1]$.
	
	Т.к. $\x\ne\lambda$, то $\tx$ находится внутри образа не менее трёх букв. По Свойству~\ref{uf:blocks}(2) любое вхождение $\tx$ в слово $\w'$ начинается с одинаковых позиций образов букв $\a$ и $\y[|\y|]$ слова $\a\x\y\x\b$, тогда $|\tx\ty|=\L|\x\y|$.
	%
	
	Если $\x\y\x$ является $(|\x|, \eps)${--}ЭС, учитывая $(*)$, получим цепочку неравенств
	\begin{equation}\notag
	\begin{split}
	\frac{|\tx\ty\tx|+\eps}{|\tx\ty|}
	&=\frac{\L|\x\y\x|+\lcs(\a', (\y[|\y|])')+\lcp(\b', (\y[1])')+\eps}{\L|\x\y|}\\
	&\le\frac{\L|\x\y\x|+\lcs(\sV_\n)+\lcp(\sV_\n)+\eps}{\L|\x\y|}=\frac{\L|\x\y\x|+\L\eps}{\L|\x\y|}\le\rt(\n).\\
	\end{split}
	\end{equation}
	
	Если же $\x\y\x$ не $(|\x|, \eps)${--}ЭС.
	Т.к. $\w$ --- $(\eps, \E)${--}ГС, то $|\x|\in\E$.
	Тогда по (M2) существует $\m$ от $1$ до $|\x|$, что образы буквы $\x[\m]$ в левом и правом $\x$ различны.
	Получаем, что образ слова $\a\x\y\x\b$ содержит не расширяемые в $\w'$ подслова получаемых нами образов слов $\w_1=\a\x\y\x[1]...\x[\m]$ и $\w_2=\x[\m]...\x[|\x|]\y\x\b$ с периодом $\L|\x\y|$.
	В максимальном подслове $\w_1$ по $(*)$ различаются буквы $\w[1]=\a$ и $\w[1+|\x\y|]=\y[|\y|]$, а у букв $\w_1[|\w_1|-|\x\y|]$ и $\w_1[|\w_1|]$ различные образы, тогда, пользуясь (uf3.1) для подслова $\tx\ty\tx$ образа $\w_1$ выполняется
	\begin{equation}\notag
	\begin{split}
	\frac{|\tx\ty\tx|+\eps}{|\tx\ty|}
	&\le\frac{\L(|\x\y|+\m'-1)+(\lcs(\a', (\y[|\y| ])')+\lcp((\w[\m-|\x\y|])', (\w[\m])')+\eps}{\L|\x\y|}\\
	&\le\frac{\L|\x\y\x|-\L+(\lcs(\sV_\n)+\L-\lcs(\sV_\n)-\eps)+\eps}{\L|\x\y|}=\frac{|\x\y\x|}{|\x\y|}\le\rt(\n).\\
	\end{split}
	\end{equation}
	Аналогично проверяется неравенство для слова $\w_2$.
	
	Остаётся рассмотреть случаи, когда все $\tx$ содержатся в образе двух букв.
	Если $\tx\ty\tx$ содержится в образе двух букв, то $\tx\ty\tx\in\F(\a\b)$ удовлетворяет~\eqref{epsE} по (uf2.1).
	Теперь пусть $\tx\ty\tx$ содержится в образе слова $\a\y\b$ при $|\y|>0$.
	Рассмотрим случай $\tx\notin\F(\v)$ для любого $\v\in\sV$ т.е. любой $\tx$ начинается в образе одной буквы и заканчивается в образе другой.
	Тогда по Свойству~\ref{uf:blocks}(3) либо $|\tx\ty|\equiv0\pmod\L$ либо оба $\tx$ содержат факторы двух образов одной буквы.
	
	В первом случае достаточно оценить сумму длин общих суффикса и префикса пар образов $\a'$ с $(\y[|\y|])'$ и $(\y[1])'$ с $\b'$ соответственно.
	При $|\y|=1$ слово $\tx\ty\tx$ --- $(|\tx|, \eps\*\chi_{\tx})${--}ЭС по (uf2.2.2).
	При $|\y|>1$ рассмотрим три подслучая.
	
	Первый, когда все прообразы пар различны.
	Т.к. $|\y|-2\ge\Dt(\a'(\y[1])', (\y[|\y|])'\b')$, то по (uf2.3) получаем требуемое условие $(|\tx|, \eps\chi_{\tx})${--}экспоненциальности для слова $\tx\ty\tx$.
	Второй, при равенстве одной из пар прообразов.
	Б.О.О. пусть $\b=\y[1]$ и $\a\ne\y[|\y|]$, расстояние между $\b$ и $\y[1]$ не менее $\n-2$, откуда $|\tx\ty|\ge\L(\n-1)$.
	Из нерасширяемости получаем $\b'\ne(\y[1])'$.
	Тогда по (uf3.1) и $|\tx|\le\lcp(\b', (\y[1])')+\lcs(\a', (\y[|\y|])')$ получим $\tx\le\L-\eps$, откуда и $(|\tx|, \eps)${--}экспоненциальность слова $\tx\ty\tx$.
	В третьем подслучае, не трудно проверить, что при равенстве обеих пар прообразов между $\b$ и $\y[1]$ не менее $2\n-1$ букв.
	Т.е. $|\tx\ty|\ge2\*\L\*\n$.
	По (uf3.1) получим $|\tx|<2\*\L-\eps$, что опять же удовлетворяет условию $(|\tx|, \eps)${--}экспоненциальности слова $\tx\ty\tx$.
	
	Во втором случае между одинаковыми прообразами не менее $\n-2$ букв.
	При этом, с учётом $(*)$, сами образы должны иметь вид либо $\a'=\al\u\v$, $(\y[1])'=\z\al$, $(\y[|\y|])'=\al\u$, $\b'=\v\z\al$ либо $\a'=\al\u$, $(\y[1])'=\v\z\al$, $(\y[|\y|])'=\al\u\v$, $\b'=\z\al$.
	Тогда по (uf3.2) $\tx\ty\tx$ --- $(|\tx|, \eps\*\chi_{\tx})${--}ЭС для обоих случаев.
	
	Остаётся рассмотреть случаи, когда $\tx$ может находиться внутри некоторого символьного образа $\a'$.
	Если оба $\tx$ внутри символьных образов $\a'=\al\tx\z_1$ и $\b'=\z_2\tx\al$, с учётом $|\y|\ge\Dt(\a', \b')$, получим $|\tx\ty|\ge\Dt(\a', \b')\*\L+|\tx\z_1\z_2|$.
	Тогда $(|\tx|, \eps\*\chi_{\tx})${--}экспоненциальность слова $\tx\ty\tx$ вытекает из (uf2.1.2).
	Пусть теперь один из $\tx$ находится на стыке двух образов $\ce'\d'$, где $\d'\in\bfdt_\d(\ce')$.
	Пусть $\tx\ty\tx$ начинается в $\a'$ и заканчивается в $\ce'\d'$ образа минимального по включению слова $\a\y\ce\d$, тогда рассмотрим два случая.
	При $\a'=\ce'$ требуемое непосредственно вытекает  из (uf2.1.1).
	При $\a'\ne\ce'$ получим $|\y|\ge\Dt(\a, \ce\d)$, тогда по (uf2.2.3) получаем требуемое.
	Остаётся случай, когда $\tx\ty\tx$ начинается в образе первой буквы и заканчивается в образе последнего слова $\ce\d\y\a$. Аналогично предыдущим двум случаям, доказываем этот, используя (uf2.1.3) и (uf2.2.1).
	
	Значит, любой фактор $\tx\ty\tx$ слова $\w'$ --- $(0, \E)${--}ГС, а при $|\tx|\notin\E$ является $(|\tx|, \eps)${--}ЭС, что удовлетворяет условию $(\eps, \E)${--}ГС $\w'$.
\end{proof}

\begin{cor}\label{cor4}
	Пусть множество $\Dt_\n$ $(\L, \eps, \E)${--}С для некоторого $\E\subset\sN$ и\\
	$\eps(\L-1)=\lcs(\sV_\n)+\lcp(\sV_\n)$, тогда при $\dt_{-}>|\E|$
	\footnote{В ДР2 условие $\dt_{-}>|\E|$ было в расширенной версии, но автор пока не разобрался с этой версией}:\par
	$(1)$ Существует граничное слово бесконечной длины над $n$ буквенным алфавитом\par
	\hfill (аналог доказательства теоремы Дежан для частных случаев).\par
	$(2)$ Если $\eps\le\frac{1}{\n-1}$, экспоненциальная гипотеза над $n$ буквенным алфавитом верна.
\end{cor}
\begin{proof}
	По Лемме~\ref{l4} из любого $(\eps, \E)${--}ГС длины $\m$ под действием\\
	$(\L, \eps, \E, \Dt_\n)${--}П мы получим (только) $(\eps, \E)${--}ГС длины $\L\m$.
	Т.к. $\L>1$, то для любого $\m\in\sN$ можно построить $(\eps, \E)${--}ГС длины $\m$.
	Т.е. можно задать последовательность, удлиняющихся слов, откуда получим (1).
	
	Любое слово в этой последовательности является граничным.
	Тогда выполняются условия следствия~\ref{morf:expgr1} (свойства~\ref{morf:expgr}),
	откуда $(\L, \eps, \E, \Dt_\n)${--}П порождает последовательность из не менее чем $\dt_-^{\frac{\m}{3}}$ слов длины $\L\m$.
	Т.к. $\L$ --- константа, то, пробегая по возрастанию по всем $\m\in\sN$, получим экспоненциальный рост числа таких слов с линейным ростом их длины.
	По Лемме~\ref{l4} все эти слова $(\eps, \E)${--}ГС, а значит граничные откуда получаем (2).
\end{proof}

\hyperlink{contents}{$\upuparrows$}

\subsubsection{Лемма \ref{l5} --- ГС с малыми экспонентами длинных факторов как следствие РРДГС}
Здесь мы докажем обобщённый вариант Леммы~\ref{l2}, а так же сформулируем гипотезу о существовании слов с <<сильной>> граничностью, и докажем ослабленный вариант этой гипотезы как следствие экспоненциальной гипотезы.

Сформулируем основную гипотезу этой части о существовании <<сильно>> граничных слов.

\begin{con}\label{con:1}
	$\forall\n\ge5\ \exists\w\in\T_n\cap\A_\n^{\om}\colon\underset{\x\in\F(\w), |\x|\ge\l}{\max}\exp(\x)\xrightarrow{\l\to\infty}1$.
\end{con}

Уточним понятие экспоненциального роста числа элементов от их <<размера>>, вычисляемого в некоторой метрике (в нашем случае от длины слова).
Назовём экспоненциальный рост множества $S$ с метрикой $\rho$ \textit{строгим}, если существует константа $c$, что для любых $\k\in\mathbb{N}$ выполняется неравенство $|S_{\k+c}|\ge2\*|S_{\k}|$, где $S_\k=\{s\colon s\in S, \rho(s)=\k\}$.

Например $\k${--}значной подстановкой над $\n${-}арным алфавитом с 1 дополнительным образом мы можем получить удвоение числа вариантов слов не более чем на каждом $\k(\n+1)$-ом шаге отображаемого слова.
Т.е. достаточно взять $c=\L\*\k\*(\n+1)$, где $\L$ --- длина образов букв.
Если дополнительных образов $\k\*(\n+1)$, то удвоение (иногда и утроение) происходит на каждой отображаемой букве.
Под действием нашей же (обобщённой) подстановки строгий экспонениальный рост происходит без дополнительных образов, например при $\dt_{-}>|\E|$ на каждой второй отображемой букве т.е. $c=2\L$.

Следующее следствие экспоненциальной гипотезы является ослаблением Гипотезы~\ref{con:1}.
{\cGr
Дополним, что формулировка следующей леммы может быть очевидным образом сформулирована в более общих терминах,
где вместо <<обобщённая подстановка со строгим эксп.ростом>> может быть использован множество слов (язык),
предствимое в виде дерева, с удваивающимися ветками на ограниченном расстоянии
(о чём было сказано в отдельной беседе (хотя, скорее монологе автора) с нашим НР, по воспоминаниям автора).
Точнее, расстояние между ближайшими вершинами (бифукациями) не превосходит некоторой (общей) константы.
В таком случае доказательство остаётся неизменным, кроме мест использования подстановки,
которые достаточно заменить на выбор ветки в месте (вершине) удвоения.
Но корректность идеи изменения формулировки ещё должна быть проверена автором т.к. данная лемма не перепроверялась после 2013 г.,
но это одно из утверждений автора из 2013 г., в котором автор уверен наверняка.
}

\begin{lem}\label{l5}
	Пусть обобщённая подстановка над граничным словом из $\T_n$ порождает множество граничных слов из $\T_n$, со строгим экспоненциальным ростом, при этом образы различных букв всегда различны
	, тогда выполняется следующее утверждение:
	$$
	\forall\eps>0\ \exists\w\in\T_n\cap\A_\n^{\om}, \exists \N\in\mathbb{N}\colon
	\forall\x\y\x\in F(\w), |\x|\ge \N\longrightarrow
	\exp(\x\y\x)\le1+\eps.
	$$
\end{lem}
\begin{proof}
	Пусть обобщённая подстановка порождает экспоненциальный рост (удваивает число слов) не более чем через каждые $\p$ букв.
	Докажем утверждение теоремы через последовательность выборов подклассов слов
	
	Применим наша подстановка над словом $\w$ удовлетворяющим условию:
	\begin{equation} \label{wrd:cond}
	\forall\k\in\sN, \u\v\u\in\F(\w), |\u|\ge\L^{\N+\k}\longrightarrow\frac{|\u\v\u|+2}{|\u\v|}\le1+\frac{1}{1.5^{\min\{\m, \k\}}}
	\end{equation}
	где $\N$ зависит от $\m$, а $\m$ определяется $\eps>0$ (вместо основания $1.5$ можно взять любое число из интервала $(1, \Phi)$, где $\Phi=\frac{\sqrt5+1}{2}$ --- пропорция золотого сечения).
	Точнее $\m\colon\frac{1}{1.5^\m}\le\eps$ т.е. можно взять $\m=\lceil-\log_{1.5}\eps\rceil$.
	Тогда $\N$ возьмём такое, что выполняются неравенства:
	\begin{equation} \label{N}
	\frac{2\p(\lceil\log_{1.5}\L^{\N+\m}\rceil+\m)}{0.08\L^\N}<1
	\text{ и }
	\frac{\p(\lceil\log_{1.5}(\L^{\N+\m}+2)\rceil+\m+1)+2}{0.5\L^\N}<1.
	\end{equation}
	Докажем, что можно получить результат (слово) удовлетворяющее тому же условию~\eqref{wrd:cond}.
	Тогда доказательство теоремы будет следовать по индукции.
	В качестве БИ (слово над которым применяется обобщённая подстановка в первый раз) будет одна буква. Очевидно,
	что импликация в \eqref{wrd:cond} выполняется.
	\\
	Ш.И.:
	Разобъём исходное слово $\w$ на непересекающиеся промежутки длины $\Ie=\p\*\lceil(\N+\m)\*\log_{1.5}\L+\m\rceil$.
	Рассмотрим множество слов:
	$$
	\U_\Ce=\big\{\w[\p-\i, ..., \p+\i], \w[\p-\i, ..., \p+\i+1]\colon\p\in\{\Ce\*\Ie, ..., (\Ce+1)\*\Ie-1\}, \i\in\sN, \i\ge\L^\N/2\big\}
	$$
	Понятно, что у всех факторов слова $\w$ с повторами длины не менее $\L^\N$ правый такой повтор попадает ровно в одно $\U_\Ce$.
	Оставим в $\U_\Ce$ только те слова $\u$, для которых существует фактор слова $\w$ вида $\u\v\u$
	(с центром правого $\u$ в $\{\Ce\*\Ie, ..., (\Ce+1)\*\Ie-1\}$) таких,
	что условие~\eqref{wrd:cond} для $\tu\tv\tu$ не выполняется.
	Обозначим полученное множество повторов (т.е. $\u$) $\U'_\Ce$ и соответствующее ему множество слов (т.е. $\u\v\u$) $\W'_\Ce$.
	Докажем, что для всех $\u\v\u\in\F(\w)$, у которых правый $\u=\u_1\a\u_2\in\U'_\Ce$, существуют образы $\tu\tv\tu_1$ и $\tu_2\tv\tu$ удовлетворяющие уловию~\eqref{wrd:cond} (т.е. с разными образами $\a$ в левом и правом $\u$) индукцией по $\Ce$.
	Напомним, что словом $\tx\ty\tx$ обозначается нерасширяемое слово в $\w'$, содержащее только образы факторов слова $\x\y\x\in\F(\w)$, причём $\tx$ содержит только образы факторов слова $\x$ (например, для $\tu\tv\tu_1$ имеется ввиду $\tu_1\a\tu_2\tv\tu_1$, где $\tu_1$ --- аналог $\tx$ в этом обозначении).
	\\
	Б.И. при $\Ce=0$ очевидно, что левого $\u$ нет, значит импликация~\eqref{wrd:cond} выполняется.
	Точнее $\W'_\Ce$ и $\U'_\Ce$ пусты.
	\\
	Ш.И.
	Пусть $|\u|\ge\L^{\N+\m}$, тогда по П.И. $\frac{|\u\v\u|+2}{|\u\v|}\le1+\frac{1}{1.5^\m}$.
	Откуда в частности $\frac{|\u_1|}{|\u_1\v_1|}<\frac{1}{1.5^{\m-1}}$.
	Тогда, после применения нашей подстановки к $\u\v\u$ период не может стать меньше $|\u\v|\L$
	т.к. образы разных букв разные и длина нерасширяемого в $\w'$ образа ($\tu\tv\tu$)
	не будет превосходить $(|\u\v\u|+2)\L-2$) для любого $\tu\tv\tu$ при $|\tu|\ge\L^{\N+\m+1}$ выполняется
	$$
	\frac{|\tu\tv\tu|+2}{|\tu\tv|}\le\frac{|\u\v\u|\L+2\L-2+2}{|\u\v|\L}=\frac{|\u\v\u|+2}{|\u\v|}\le1+\frac{1}{1.5^\m}
	$$
	Т.е. если $|\u|\ge\L^{\N+\m}$, то $\u\not\in\U'_\Ce$.
	
	Пусть $\L^{\N+\k-1}\le|\u|<\L^{\N+\m}, \k=0, ..., \m$ и $\u\in\U'_\Ce$.
	Докажем, что можно образы левого и правого $\u$ <<разбить>> на части отличающихся (длинами) от $|\u|$ не более чем в $1.5$ раза.
	
	Пусть $\u_1\v_1\u_1$ и $\k$ такие, что $\u_1\in\U'_\Ce$
	и выполняются неравенства:
	\begin{equation} \label{wrd:cond2}
	1+\frac{1}{1.5^{\k}}<\frac{|\u_1\v_1\u_1|+2}{|\u_1\v_1|}\le1+\frac{1}{1.5^{\k-1}},
	\end{equation}
	что эквивалентно $|\u_1\v_1|<1.5^\k(|\u_1|+2)\le1.5|\u_1\v_1|$.
	Тогда фиксируем (выберем) образ в позиции (слова $\w$) из множества
	$\{\Ce\Ie+\p\log_{1.5}|\u_1\v_1|, ..., \Ce\Ie+\p(\log_{1.5}|\u_1\v_1|+1)-1\}$
	образом отличным от образа соответствующей буквы в левом $\u_1$ (т.е. букввы на $|\u_1\v_1|$ позиций левее).
	Такая позиция существует по условию строгого экспоненциального роста (удвоения) через каждые $\p$ букв.
	
	Тогда длина наибольшего из полученных отрезков $\u_{1_1}\a\u_{1_2}=\u_1$, пусть $|\u_{1_1}|$, не превосходит
	$\frac{|\u_1|+2\p(\log_{1.5}|\u_1\v_1|+1)}{2}<\frac{|\u_1|+2\p(\log_{1.5}(1.5^\k(|\u_1|+2))+1)}{2}\le\frac{|\u_1|+2\p(\log_{1.5}(|\u_1|+2)+\m+1)}{2}\le\frac{|\u_1|}{1.5}-2$
	т.е. $|\u_{1_1}|+2<1.5|\u_1|$ с учётом условия~\eqref{N}.
	Откуда
	$$
	\frac{|\tu_1\tv_1\tu_{1_1}|+2}{|\tu_1\tv_1|}\le\frac{|\u_1\v_1|\L+2\L-2+2+|\u_{1_1}|\L}{|\u_1\v_1|\L}=1+\frac{|\u_{1_1}|+2}{|\u_1\v_1|}\le1+\frac{|\u_1|}{1.5|\u_1\v_1|}<1+\frac{1}{1.5^\k}
	$$
	Теперь, т.к. $1+\frac{1}{1.5^{\k}}<\frac{|\u_1\v_1\u_1|+2}{|\u_1\v_1|}$, то по~\eqref{wrd:cond} такое возможно только если $|\u_1|<\L^{\N+\k}$.
	При этом $|\tu_{1_1}|+\L\le|\tu_{1_1}|+|\a'|\le|\tu_1|<|\u_1|\L+2\L\le(\L^{\N+\k}-1)\L+2\L=\L^{\N+\k+1}+\L$.
	Значит условие~\eqref{wrd:cond} выполняется и для $\tu_1\tv_1\tu_{1_1}$.
	
	В частности мы доказали, что при такой процедуре (<<разбиения>>), если $\L^{\k-1}\le|\u_1|<\L^{\k}$ и $\frac{|\u_1\v_1\u_1|}{|\u_1\v_1|}\le1+\frac{1}{1.5^{\k-1}}$, то $\L^{\k}\le|\tu_{1_1}|<\L^{\k+1}$ и $\frac{|\tu_1\tv_1\tu_{1_1}|}{|\tu_1\tv_1|}\le1+\frac{1}{1.5^\k}$.
	А так как, если $\u\v\u\in\F(\w)$ и $|\u|\ge\L^{\N+\m}$, то $\u\v\u\not\in\W'_\Ce$, тогда достаточно проделать эту операцию для $\u\v\u\in\F(\w)$ удовлетворяющих условию $\L^{\N-1}\le|\u|<\L^{\N+\m}$ (и $\frac{|\u|}{|\u\v|}\le1$ при $|\u|<\L^\N$, что выполняется в силу граничности $\w$).
	
	Для остальных случаев (с повторами $\u\in\U'_\Ce$) проделываем ту же процедуру, если для них найдётся буква в повторе, близкой к его центру, при помощи которой можно разбить повтор на меньшие части.
	Для этого достаточно доказать, что промежутки в которых фиксируются буквы, не пересекаются для разных $\u\v\u$ из $\W'_\Ce$, подходящих под условие~\eqref{wrd:cond2}.
	
	Возмём пару слов $\u_{\i-1}\v_{\i-1}\u_{\i-1}$ и $\u_\i\v_\i\u_\i$ из $\W'_\Ce$ и соответствующие им $\k_{\i-1}$ и $\k_\i$ удовлетворяющие условию~\eqref{wrd:cond2}, тогда множества\\
	$\{\Ce\Ie+\p\log_{1.5}|\u_\i\v_\i|, ..., \Ce\Ie+\p(\log_{1.5}|\u_\i\v_\i|+1)-1\}$ и\\
	$\{\Ce\Ie+\p\log_{1.5}|\u_{\i-1}\v_{\i-1}|, ..., \Ce\Ie+\p(\log_{1.5}|\u_{\i-1}\v_{\i-1}|+1)-1\}$
	не должны пересекаться.
	Для этого достаточно, чтобы либо $\log_{1.5}|\u_\i\v_\i|+1\le\log_{1.5}|\u_{\i-1}\v_{\i-1}|$ т.е. $1.5|\u_\i\v_\i|\le|\u_{\i-1}\v_{\i-1}|$, либо $|\u_\i\v_\i|\ge1.5|\u_{\i-1}\v_{\i-1}|$.
	
	Рассмотрим наибольшее общее пересечение (подслово) $\u$ в правых $\u_{\i-1}$ и $\u_\i$ в контексте $\w$.
	Если оно не совпадает с минимальным (по длине) из них, то длины $\u_{\i-1}$, $\u_\i$ и $\u$ <<соизмеримы>> (при достаточно больших $\N$).
	Точнее, $|\u_\i|\le|\u_{\i-1}|+2\Ie$ откуда по условию~\eqref{N} получим $|\u_{\i-1}|\ge|\u_\i|\*(1-\frac{2\Ie}{|\u_\i|})\ge|\u_\i|\*(1-0.08)$ (независимо от того какой из $|\u_\i|$ и $|\u_{\i-1}|$ максимален).
	
	
	Предположим, что $1.5^2|\u_\i\v_\i|>1.5|\u_{\i-1}\v_{\i-1}|>|\u_\i\v_\i|$.
	Рассмотрим подслово\\
	$\u\v\u\in\F(\w)$, у которого левый $\u$ стоит в левом $\u_{\i-1}$, а правый в левом $\u_\i$.
	Тогда для $\j=\arg\max\{|\u_\j\v_\j|\colon\j\in\{\i-1, \i\}\}$
	\begin{equation}\notag
	\begin{split}
	\frac{|\u\v\u|+2}{|\u\v|}
	&\ge\frac{\big||\u_{\i-1}\v_{\i-1}|-|\u_\i\v_\i|\big|-\Ie+\min\{|\u_\i|, |\u_{\i-1}|\}-\Ie+2}{\big||\u_{\i-1}\v_{\i-1}|-|\u_\i\v_\i|\big|}
	>1+\frac{\min\{|\u_\i|, |\u_{\i-1}|\}-2\Ie+2}{|\u_\j\v_\j|/3}\\
	&>1+\frac{|\u_\j|(1-0.08)-2|\u_\j|0.08+2}{|\u_\j\v_\j|/3}
	>1+\frac{3(|\u_\j|+2)(1-0.25)}{|\u_\j\v_\j|}
	>1+\frac{3\cdot0.75}{1.5^{\k_\j}}\\
	&=1+\frac{1}{1.5^{\k_\j-2}}
	\end{split}
	\end{equation}
	Т.е. при достаточно больших $\N$ для выполнения условия~\eqref{wrd:cond} необходимо, чтобы
	$|\u|<\L^{\N+\k_\j-2}$.
	При этом, при этих же $\N$ выполняется $\L^{\N+\k_\j-2}>|\u|\ge(1-2\cdot0.08)|\u_\j|\ge(1-2\cdot0.08)\L^{\N+\k_\j-1}\ge(2-4\cdot0.08)\L^{\N+\k_\j-2}>\L^{\N+\k_\j-2}$, что неверно (неравенство $|\u_\j|\ge\L^{\N+\k_\j-1}$ следует из ограничения~\eqref{wrd:cond2} на $\k_\j$ и условия $\W'_\Ce$).
	
	Теперь, пусть $\u_\i$ и $\u_{\i-1}$ <<несоизмеримы>> т.е. $\u$ равно минимальному по длине из них.
	Б.О.О. пусть $\u=\u_\i$.
	Рассмотрим подслово $\u\v\u\in\F(\w)$, где один из повторов лежит в левом $\u_{\i-1}$ другой в левом $\u_\i$ (т.е. совпадает с ним).
	Тогда, либо $|\u\v|+|\u_{\i-1}\v_{\i-1}|=|\u_\i\v_\i|$, либо $|\u\v|+|\u_\i\v_\i|=|\u_{\i-1}\v_{\i-1}|$.
	
	Предположим, что $1.5^2|\u_\i\v_\i|>1.5|\u_{\i-1}\v_{\i-1}|>|\u_\i\v_\i|$.
	Тогда $|\u\v|<0.5|\u_\i\v_\i|$ и $|\u\v|<0.5|\u_{\i-1}\v_{\i-1}|$.
	При этом $|\u|=|\u_\i|\ge\L^{\N+\k_\i-1}$ откуда по~\eqref{wrd:cond} получим $\frac{|\u|+2}{|\u\v|}\le\frac{1}{1.5^{\k_\i-1}}$.
	При этом из нашего предположения
	$\frac{|\u|+2}{|\u\v|}>\frac{|\u_\i|+2}{0.5|\u_\i\v_\i|}>\frac{1}{1.5^{\k_{\i}-1}}$
	\ т.е. предположение не верно.
	
	Остаётся доказать, что <<зарезервированных>> (для <<разбиения>>) позиций $\{\Ce\Ie, ..., (\Ce+1)\Ie-1\}$ достаточно для всех элементов из $\W'_\Ce$.
	Т.к. все $\u\in\U'_\Ce$ имеют длину менее $\L^{\N+\m}$ и периоды соответствующих им слов из $\W'_\Ce$ отличаются не менее чем в $1.5$ раза, то по~\eqref{wrd:cond2} для некоторого $\k\le\m$
	$$
	|\W'_\Ce|\le\max_{|\u\v\u\in\W'_\Ce|}\lceil\log_{1.5}(|\u\v|)\rceil\le\lceil\log_{1.5}(1.5^{\k}|\u|)\rceil\le\lceil\log_{1.5}\L^{\N+\m}+\m\rceil.
	$$
	При этом, каждому элементу из $\W'_\Ce$ достаточно ровно $\p$ (подряд идущих) букв.
	Т.е. всем элементам из $\W'_\Ce$ достаточно $\p\lceil\log_{1.5}\L^{\N+\m}+\m\rceil=\Ie=|\{\Ce\Ie, ..., (\Ce+1)\Ie-1\}|$ позиций.
\end{proof}
Заметим так же, что после построения такого слова можно снова применить ту же обобщённую подстановку к полученному слову,
чтобы получить множество граничных слов со строгим экспоненциальным ростом и почти такимже свойством
(экспонента в пределе, если и увеличится, то не более чем в 2 раза <<отклонится>> от 1.
Точнее через нашу подстановку над словом с экспонентой для $1+\eps$ для длинных повторов получим слово с факторами,
у которых повторы в $\L$ раз длиннее, имеющими экспоненту не больше $1+2\eps$).

\hyperlink{contents}{$\upuparrows$}

\newpage

\section{Конструкции $\D_{3,n}^\eps${--}ЦГС для $|\A_n|\ge5$}\label{constructs}

В данном разделе описаны конструкции для $bc${--}корней предположительно удовлетворяющих условиям Леммы \ref{l3}.
Для нечётных было проверено в 2012 году.
В следующей версии автор планирует проверить корректность данных конструкций.

\subsection{Конструкция для нечётных $|\A_n|\ge5$}

Пусть $m=\frac{n-3}{2}$
\begin{description}
	\item[1]
	$b=00, c=001$ --- $bc$-код
	\item[2] Внутренние блоки по всем $0\le i<2m + 7$ чётные (назовём $d${--}кодами/блоками)
	
	$d_i=\phi(cb^{\frac{n+i-5}{2}}c)$
	
	$|d_i|=n+i+1$
	\item[3] Внешние блоки по всем $0\le i<m$ (назовём $q${--}кодами/блоками)
	
	\begin{tabular}{ c c }
		$q_i=d_6d_0d_4d_{2i+8}d_0d_6$, & $q'_i=d_6d_0d_{2i+8}d_4d_0d_6$
	\end{tabular}
	
	$|q_i|=|q'_i|=6n +7+1+5+2i+9+1+7=6n+2i+30$
	
	\item[4.1] Тогда первые $n+1$ $bc${--}корня  слова определяются так
	$$w_0=q_0q'_0\Bigg(\prod_{i=0}^{m-1}q_i\Bigg)q_0q'_0\Bigg(\prod_{i=0}^{m-1}q'_i\Bigg)$$
	Остальные получаем циклическими сдвигами $w_0$ влево на целые блоки $q_i,q'_i$.
	
	$|w_i|=4(6n+2\cdot1+28)+2\big((6n+2\cdot1+28)+...+(6n+2m+28)\big)=6.5n^2+36n+37.5$
	
	\item[4.2]
	Наконец, определим $3n+3$ $bc${--}корней слова.
	Для произвольного $0\le i<n+1$ возмём $w_i=d_6ud_6$, тогда
	
	\begin{tabular}{ c c c }
		$w_i^-=ud_6d_6$ & $w_i^0=w_i$ & $w_i^+=d_6d_6u$
	\end{tabular}
	
	Обозначим
	\begin{tabular}{ c c c }
		$W_n^-=\{w_i^-: i\in\{0,...,n\}\},$ & $W_n^0=\{w_i^0: i\in\{0,...,n\}\},$ & $W_n^+=\{w_i^+: i\in\{0,...,n\}\}$
	\end{tabular}
	
\end{description}

\hyperlink{contents}{$\upuparrows$}

\subsection{Конструкция для чётных $|\A_n|\ge5$}

Здесь мы будем использовать всю индексацию с 0,
в отличие от ДР2, где, например, индекс <<$j$>> начинается с 1.
Это удобнее для реализации программы проверки.
Так же, представленные здесь, конструкции будут несколько отличаться.
Точнее, последовательностью внешних блоков $q_{j,i},q'_{j,i}$ в $bc${--}корнях $w_i$.
Представленные в ДР2 конструкции оказались не подходящими для применения нашей схемы в леммах \ref{l1}, \ref{l2}.
Даже при некоторых $n\ge5$ сами слова (не только пары) не являются граничными.

Новая версия, пердставленная ниже, проверена (на компьютере) не до конца.
Но для $n\in\{6,...,32\}\setminus\{8\}$ (для всех $k\ge5$) сами слова являются $\D_{3, n}^\eps${--}ГС
(на сколько автор успел проверить, Но для $n=8$-ми слово только ГС).
Отличие новой конструкции в основном только в $w_i$ и незначительные изменения в $q_{j,i}${--}кодах,
которые можно было не менять (но, для боолее красивой закономерности, индексы немного изменены).

Пусть $m=\frac{n\pmod{6}}{2}, h=\lfloor n/12\rfloor+1$
\begin{description}
	\item[1]
	$b=00, c=001$ --- $bc$-код
	\item[2] Внутренние блоки по всем $0\le i\le h$, $0\le j\le3$ (назовём $d${--}кодами/блоками)
	
	$d_{i,j}=\phi(b^{\frac{n+2j-6}{2}}c^{2i+1})$
	Заметим, что для разных пар $i,j$ $d_{i,j}$ всегда различны.
	
	$|d_{i,j}|=n+2j+6i-3$
	\item[3] Внешние блоки по всем $0\le i\le h-1$ (назовём $q${--}кодами/блоками)
	
	%
	%
	
	%
	%
	%
	
	{
		\begin{tabular}{ c c c }
			$q_{0,i}=d_{i,3}d_{0,0}d_{i,1}d_{i,2}d_{i,1}d_{0,0}$ & $q_{1,i}=d_{i,3}d_{0,0}d_{i,3}d_{i,1}d_{i,1}d_{0,0}$ & $q_{2,i}=d_{i,3}d_{0,0}d_{i,2}d_{i,3}d_{i,1}d_{0,0}$\\ 
			$q'_{0,i}=d_{i,3}d_{0,0}d_{i,2}d_{i,1}d_{i,1}d_{0,0}$ & $q'_{1,i}=d_{i,3}d_{0,0}d_{i,1}d_{i,3}d_{i,1}d_{0,0}$ & $q'_{2,i}=d_{i,3}d_{0,0}d_{i,3}d_{i,2}d_{i,1}d_{0,0}$
		\end{tabular}
	}
	
	$|q_{0,i}|=|q'_{0,i}|=6n+2(3+0+2+1+1+0)+6(4i)-18=6n-4+24i$
	
	$|q_{1,i}|=|q'_{1,i}|=6n+2(3+0+3+1+1+0)+6(4i)-30=6n-2+24i$
	
	$|q_{2,i}|=|q'_{2,i}|=6n+2(3+0+3+2+1+0)+6(4i)-30=6n-0+24i$
	
	\item[4.1]
	Тогда $bc${--}корень (из $2((h-1)*6)+10=12h-2$ $q${--}кодов)
	$$
	w_0=
	\Bigg(\prod_{i=0}^{h-2}\prod_{j=0}^{2}q_{j,i}q'_{j,i}\Bigg)
	\Bigg(\prod_{j=0}^{1}q_{j,h-1}q'_{j,h-1}\Bigg)
	q_{2,h-1}
	\Bigg(\prod_{i=0}^{h-2}\prod_{j=0}^{2}q'_{j,i}q_{j,i}\Bigg)
	\Bigg(\prod_{j=0}^{1}q'_{j,h-1}q_{j,h-1}\Bigg)
	q'_{2,h-1}
	$$
	Остальные получаем циклическими сдвигами $w_0$ влево на целые блоки $q_{i,j},q'_{i,j}$
	(обозначим $n_{12}=\lfloor n/12\rfloor\cdot12$)
	
	Число различных сдвигов по $q${--}кодам равно $|w_0|_q=2((h-1)\cdot3\cdot2+2\cdot2+1)=12h-2=n_{12}+12-2=n_{12}+10$.
	
	$|w_i|
	\ge|w_0|_q\cdot\min\{|q_{j,i}|: i,j\in\mN_0\}
	=(n_{12}+10)(6n-4)
	$
	
	%
	
	Заметьте, что при <<неконстантном>> перераспределении $c$ по краям $d${--}кодов, $bc${--}корни
	не будут циклически эквивалентными, в отличие от конструкции для нечётных $n$.
	\item[4.2]
	Определим $3n+12$ осевых слова.
	Возмём $w_i=d_{i,3}ud_{i,0}$
	
	\begin{tabular}{ c c c }
		$w_i^-=ud_{i,0}d_{i,3}$ & $w_i^0=w_i$ & $w_i^+=d_{i,0}d_{i,3}u$
	\end{tabular}
	
	Обозначим
	\begin{tabular}{ c c c }
		$W_n^-=\{w_i^-: i\in\{0,...,n+3\}\},$ & $W_n^0=\{w_i^0: i\in\{0,...,n+3\}\},$ & $W_n^+=\{w_i^+: i\in\{0,...,n+3\}\}$
	\end{tabular}
	
\end{description}


\hyperlink{contents}{$\upuparrows$}

\subsection{Общие свойства конструкций и сведение проверки условий лемм к полиномиальной}\label{poly_proofs}

\begin{note}
	\ref{l1c2} вытекает из \ref{l2c2}.
\end{note}

$q${--}факторами/кодами в $w_i$ будем называть факторы/коды $w_i$ на уровне $q${--}кодов.

$bc${--}факторами/кодами в $w_i$ будем называть факторы/коды $w_i$ на уровне $bc${--}кодов.

$n^+$ обозначим натуральное число равное или большее $n$.

<<з-е>> будет сокращением от <<замечание>>.

\begin{defn}
	Одинаковые $bc${--}факторы в $w_i$ и $w_j$ синхронны,
	если они расширяемы (т.е. сохраняя одинаковость) в $w_i^\infty$ и $w_j^\infty$ вправо не ограничено.
	
	Факторы в $v_i$ и $v_j$ синхронны, если их $bc${--}коды синхронны в $w_i^k$ и $w_j^k$ ($i$ и $j$ могут быть одинаковы, конечно же).
\end{defn}

\begin{note}\label{nt:conjNotTW}
	Если слова $u$ и $v$ сопряжены (равны при некотором циклическом сдвиге одного из них), то $\lexp(uv)\ge2$.
\end{note}

\begin{note}\label{nt:ifSyncThenConj}
	Если одинаковые $n^+${--}факторы $u$ и $v$ синхронны в $v_i$ и $v_j$, то $v_i$ и $v_j$ сопряжены.
\end{note}

Пусть $k${-}$bc${--}корень $w_i$ ограничен сверху полиномиальным от $n$ (степени $p_1$) числом $q${--}кодов.
Длины $bc${--}кодов в $q${--}кодах ограничены сверху полиномом $O(n^{p_2})$.
Т.е. длина $w_i$ ограничена $O(n^{p_1+p_2})$.
Одинаковые несинхронные $q${--}факторы в $w_i^\infty$ и $w_j^\infty$ имеют длину $O(n^{p_3})$ (количество $q$ в $q${--}факторе).
Обозначим $f_n=\max\{|w_0|,n^{p_3+p_2+1}\}$ ($|w_0|$ представима как $n^{\log_n(|w_0|)}$).

В нашем случае, и для нечётных и для чётных, $p_1=1$, $p_2=1$, $p_3=0$.
Тогда

\begin{note}\label{nt:local_poly_one}
	Длина 2-х одинаковых несинхронных $bc${--}факторов в $w_i^\infty$ и $w_j^\infty$
	не превосходит $O(n^{p_1+p_2})$ и даже $O(n^{p_3+p_2})$.
	Откуда, все <<запрещённые>> (не $\D_{3, n}^\eps${--}ГС, $\eps\ge0$) факторы
	не превосходят длины $O(n^{p_1+p_2+1})$ и даже $O(n^{p_3+p_2+1})$.
	%
\end{note}

Соответственно, по замечанию \ref{nt:local_poly_one} нам достаточно проверить все факторы в $v_iv_j$ длины $O(n^2)$.
Здесь же, можно заметить

\begin{note}\label{nt:glob_poly_one}
	Существует такая константа $t\le|v_0|/f_n$, что при отсутствии <<запрещённых>>
	
	\hfill факторов в $\pref_{t\cdot f_n}(v_i)$,
	<<запрещённые>> факторы отстутствуют и в $v_i$ (независимо от $n$).
\end{note}

Из этого замечания следует, что проверка слова $v_0$ на $\D_{3, n}^\eps${--}граничность полиномиальна (точнее $O(n^2)$).
А с подтверждением $v_0$ на $\D_{3, n}^\eps${--}граничность,
по Лемме \ref{l3}, следует и $\D_{3, n}^\eps${--}граничность всех осталных $v_i\in\V$.

Для проверки несопряжённых пар $v_iv_j$ получим аналогичное

\begin{note}\label{nt:poly_pair}
	Для любых несопряжённых $\V${--}образов $v_i$ и $v_j$ cуществует такая константа $t\le|v_0|/f_n$,
	что при отсутствии <<запрещённых>> факторов в $\suff_{t\cdot f_n}(v_i).\pref_{t\cdot f_n}(v_j)$,
	<<запрещённые>> факторы отстутствуют и в $v_iv_j$.
\end{note}
\begin{proof}
	ОП. Предположим, что есть контрпример в $v_iv_j$ неограниченный длиной $O(f_n)$.
	По з-ю \ref{nt:glob_poly_one} для факторов в $v_i$ и $v_j$ это свойство выполняется.
	Тогда повторы контрпримера пересекают и $v_i$ и $v_j$.
	
	По несопряжённости и з-ю \ref{nt:ifSyncThenConj} $n^+${--}повторы не синхронны,
	а более короткие повторы могут быть только в коротком <<запрещённом>> факторе (ограниченный $O(n^{2})$).
	Тогда по з-ю \ref{nt:local_poly_one} длины повторов ограничены $O(n^{p_3+p_2})$,
	даже когда один из повторов пересекает оба $\V${--}образа.
	Откуда, все запрещённый факторы не длиннее $O(n^{p_3+p_2+1})$.
	
	Тогда, все неэквивалентные запрещённые факторы в $v_iv_j$ на ходятся в факторе длины $|w_0|+O(n^{p_3+p_2+1})$.
	Откуда следует достаточность проверки центрального фактора длины $O(f_n)$ в конкатенации $v_i.v_j$.
\end{proof}

{\bf
	Для проверки на $\D_{3, n}^\eps${--}граничность всех $\V${--}образов
	из з-я \ref{nt:glob_poly_one} и Леммы \ref{l3} следует достаточность проверки полиномиального префикса одного $\V${--}образа.
	
	Для проверки на $\D_{3, n}^\eps${--}граничность любой несопряжённой пары $\V${--}образов
	Из з-я \ref{nt:poly_pair} следует достаточность проверки центра полиномиальной длины в конкатенации этой пары.
	
	Конкатенация любой сопряжённой пары не является $\D_{3, n}^\eps${--}ГС по з-ю \ref{nt:conjNotTW}
}

Проверка пары $\V${--}образов полиномиальна.
Об этом есть в ДР1 (в параграфе \$8), а так же в тексте лекций (на последних страницах).
В данной работе в разделе/параграфе \ref{conjIsPoly_Algo} описан полиномиальный алгоритм поиска <<канонического>> слова
для любой пары $\V${--}образов (версия ДР2).
Раздел \ref{conjIsPoly_Algo}

\hyperlink{contents}{$\upuparrows$}

\subsection{Проверка конструкций на условия лемм \ref{l1} и \ref{l2} кроме \ref{l1c2} и \ref{l2c2} (в след.версии)}\label{check1}

Данная часть в процессе разработки.
Но метод проверки не отличается от метода проверки для нечётных, предложенного в ДР1.
Потому, читатель может попробовать самостоятельно проверить условия.
Хоть и доказано в общем случае, что проверка сводится к полиномиальной (здесь \ref{poly_proofs}).
Но автор  пока не может утверждать,
какой (конкретной) длины факторы в $v_iv_j$ достаточно провериь для полной проверки частных случаев
(а с ними и доказательства этих случаев).

Для нечётных в ДР1 было установлено,
что остаточно проверить фактор длины $5|w_0|$ для проверки $\D^\eps_{3, n}${--}граничности самих слов $w_i$.
А для проверки пары $\V${--}образов $v_i, v_j$ на $\D^\eps_{3, n}${--}граничность
достаточно проверить центральный фактор их конкатенации с длиной $8|w_0|$.

Введём условие на множестве $\V${--}образов при $\eps=\frac{\l(\V)+\r(\V)}{L-1}$
\begin{equation}
\tag{C3}\label{c3}
\eps+\l(\V)+\r(\V)\le\frac{L}{n-1}
\end{equation}
Заметим, что это неравенство эквивалентно неравенству $\frac{\l(\V)+\r(\V)}{L-1}\le\frac{1}{n-1}$,
как и неравенству $\eps\le\frac{1}{n-1}$.

Для начала 
\begin{note}
	Если $|\V|$ в Лемме \ref{l1} не превосходит $3n$, то выполнение остальных условий Леммы \ref{l1}
	вытекает из выполнения условий Леммы \ref{l2} и условия (\ref{c3}).
	Условие \ref{l2c3} вытекает из условия (\ref{c3}).
\end{note}
\begin{proof}
	По условию Леммы \ref{l2} $|\V|\ge3n$, а знаит $|\V|$ достаточно большое для Леммы \ref{l1}.
	
	Условие \ref{l1c1} эквиалентно условию \ref{l2c1}.
	
	Условие \ref{l1c2} вытекает из \ref{l2c2} при любых $\eps\ge0$.
	
	Возмём произвольные $v,v'\in\V$ и $u\ne w\in\V\setminus\{v,v'\}$.
	Тогда $l$, при котором $\pref_l(u)=\pref_l(v)$, не превосходит $\l(\V)$.
	Аналогично, $r$, при котором $\suff_r(v')=\suff_r(w)$, не превосходит $\r(\V)$.
	Тогда, используя (\ref{c3})
	\parskip -7pt
	
	$$
	\eps+\max\{l+r: \pref_l(u)=\pref_l(v), \suff_r(v')=\suff_r(w)\}
	\le\eps+\l(\V)+\r(\V)
	\le\frac{L}{n-1}
	$$
	\parskip -3pt
	
	Понятно, что это условие усиливает условие \ref{l2c3}
	(т.к. в (\ref{c3}) $u,w$ могут быть ещё и $\V${--}образами общей буквы, хоть между собой, хоть с $v$ или с $v'$.
	При этом, $v$ и $v'$ не обязаны быть $\V${--}образами общей буквы).
	Т.е. условие \ref{l2c3} вытекает из (\ref{c3}).
	При этом, когда $v=v'$ и $\eps=0$ получим условие \ref{l1c3}.
	Т.е. условие \ref{l1c3} вытекает из (\ref{c3}) при любых $\eps\ge0$.
\end{proof}
Т.е. достаточно доказать выполнение условий \ref{l2c1}, \ref{l2c2}, (\ref{c3}) и ограничение на $L\ge6(n-1)$.

Т.е. в данной секции \ref{check1} докажем \ref{l2c1}, (\ref{c3}).

Заметьте, что для чётных каждый $q${--}код в $w_i$ дублируется по 1-му разу,
но не дублируются даже пары соседних $q${--}кодов в $w_0^\infty$, в отличие от конструкции для нечётных.
Тогда, с учётом, что $|q_i|=|q'_i|$ и $|q_{i,j}|=|q'_{i,j}|$ можно заметить, что

\begin{note}
	Длина одинаковых несинхронных факторов в произвольных $w_i^\infty,w_j^\infty$ не превосходит $|q'_{m-1}q_0q'_0q_0|$ для нечётных,
	и $|q'_{1,h-1}q_{1,h-1}q'_{2,h-1}|$ для чётных. {\bf\cRe перепроверить}
\end{note}

\begin{description}
	\item [(|V|$\ge$3n)\label{prV}]
	Для нечётных $|\V|=3(n+1)>3n$.
	
	Для чётных $|\V|=3(\lfloor n/12\rfloor\cdot12+10)\ge3n$
	При этом, при любых $n\not=10\pmod{12}$ 
	количество $|\V|>3n$.
	
	\item [(L$\ge$6n-6)\label{prL}]
	Для нечётных $n\ge5$ оценим длину
	$L\ge3(6.5n^2+36n+37.5)>6(n-1)$.
	
	Для чётных $n\ge5$ длина 
	$L\ge3|w_0|>|q_{j,i}|\ge6n-4>6(n-1)$ при любых $i$ и $j\in\{0,1,2\}$.
	
	\item [(l2.c1)\label{prc1}]
	Очевидно, что все наши $\D_{3, n}^\eps${--}ЦГС одинаковой длины при одинаковом $n$
	т.к. эквивалентны некоторым циклическим сдвигам друг друга.
	А простота слов сразу вытекает из замечания \ref{pr1}. 

	\item [(C3)\label{prc3}]
	Рассмотрим подмножества $\V^-,\V^0,\V^+$ по отдельности для поиска $\l(\V)$ и $\r(\V)$.
	Данный пункт {\bf\cRe требует доработки},
	но проверка условий не критична для доказательства корректности конструкций
	т.к. задача уже сведена к кончной компьютерной проверке, а данный пункт только позволяет оптимизировать проверку.
	
	Для начала зметим, что все $d${--}коды различны при различных индексах.
	Т.е. достаточно находить различие $d${--}кодов по раличным индексам (с учётом, что длина $d${--}кода меньше при меньшем индексе).
	
	{\bf$\bullet$} В общем случае, для $\l(\V)$ сравнивать достаточно $\W^-_n,\W^0_n,\W^+_n$
	т.к. общий префикс слова заканчивается вместе с общим префиксом его кода.
	А для $\r(\V)$ аналогично, но необходимо учитывать предпосылку перед общим суффиксом кода
	(т.е. добавлять $n$ --- убедиться, что не больше).
	
	{\bf Т.е. для $\l(\V)$ и $\r(\V)$ ищем общие префикс и суффикс кодов, но к суффиксу добавляем $n$.}
	
	Можно заметить, что общие префиксы и суффиксы кодов из различных $\W^*_n, \W^{\not *}_n$ короче чем из общего $\W^*_n$,
	уже за счёт общих $q${--}кодов у вторых.
	
	{\bf(1)} Несинхронные пары из различных $\W^*_n, \W^{\not *}_n$.
	
	Нечётные $\l(\W^*_n, \W^{\not *}_n), \r(\W^*_n, \W^{\not *}_n)\le n+|d_0d_6|$
	
	Чётные преф. $\l(\W^*_n, \W^{\not *}_n)\le|d_{0,0}d_{i,3}d_{0,0}|$ (при $q^-_{1,i}$ и $q'^+_{2,i}$)
	
	Чётные суфф. $\r(\W^*_n, \W^{\not *}_n)\le n+|d_{i,1}|$
	
	{\bf(2)} Для несинхронных пар из общего $\W^*_n$ получим
	
	Нечётные преф. $\l(\W^*_n)\le$ + $|q_6|-|q_0|$ + $|q_0q'_0|$ + $|d_6d_6d_0d_{2i+8}|$ при $i=m-2=(n-7)/2$ и $\V^+$
	
	Нечётные суфф. $\l(\W^*_n)\le$ + $n$ + $|d_0d_{2i+8}d_6d_6|$ + $|q_0q'_0|$ + $|q_6|-|q_0|$ при $i=m-2=(n-7)/2$ и $\V^-$
	
	Чётные суфф. $\l(\W^*_n)\le$ + $n$ + $|d_{i,1}d_{i,1}d_{0,0}|$ + $|q_{2,h-1}|$ + $|d_{i,3}|$ при $i=h-2$ и $\V^-$
	
\end{description}

\hyperlink{contents}{$\upuparrows$}

\subsection{Сведение к полиномиальной проверке условий \ref{l1c2} и \ref{l2c2} (в след.версии)}
Это самый рулинный раздел, но не требующий высшей математики.
Так же, как и предыдущий раздел требует доработки.

\subsubsection{Случай несопряжённых $u\not=v\in\V$ --- сведение к полином-й проверке $uv\in\D_{\ge3n-3,n}\cap\T_n$}

Можно заметить, что для нечётных пар из различных $\W^*_n, \W^{\not *}_n$
при индексах $i$ и $j=(i+|\W^0|/2)\pmod{|\W^0|}$ код на стыке $w_i.w_j$,
либо влево, либо вправо не может расширяться ни на какой целый $q${--}код.

Аналогично и для чётных.

Пусть различные $v_i,v_j\in\V$ не сопряжены.
Тогда любой повтор $v'=v_iv_j[c_1,..,c_1+|v'|]=v_iv_j[c_2,..,c_2+|v'|]$ (БОО $0<c_1<c_2$ и $c_2+|v'|<|v_i|=|v_j|$)
имеет длину $|v|-n\le18n+86$ т.к. $|v|-n$ это длина кода Пансьё,
который должен совпадать для достаточно длинного фактора $v$,
но нет 2-х $bc'${--}кореней нашей конструкции с общим кодом длины больше $18n+86$.
Проверим это.

Заметим, что $d_{i,j}$ синхронизируются только если одинаковы.
Т.е. если в $bc'${--}коде найден код совпадающий с $d_{i,j}$, то это он и есть
т.к. $b$ ограничено $c$-шками и однозначно определён его количеством, зависящим от $i$.

А значит максимальное пресечение $v_i$ и $v_j$ возможно при максимальном количестве состыковок $d_{i,j}$.
Ну а последовательность пары $d_{i,j}$-х однозначно синхронизируют $q_{m,i}$ по краям.
Аналогично и $q_{m,i}$ синхронизируются только если одинаковы.

{\bf Найдём максимальные <<запрещённые>> повторы $v_i$ и $v_j$, ни один из которых не содержит их стык.}
Важная тонкость при проверке лок.эксп. $\V${--}образа,
что $bc'${--}корень не превосходит полинома от $n$,
т.к. нам необходимо проверять короткие(полиномиальной длины факторы) по всей длине $bc'${--}корня.

{\bf для $n\equiv_2 1$}
Когда по $q_i$ коду стыков максимум т.е. $q'_mq_1q'_1q_1$ и $q_mq_1q'_1q'_1$.
Тогда $|v|\le|d_5d_1d_7|+2(6n+2+28)+|d_7d_1d_5|=2(6n+30+(3n+7+1+5))=18n+86$.

{\bf для $n\equiv_2 0$} {\cRe(в след.версии)}
%

{\bf Найдём максимальные <<запрещённые>> повторы $v_i$ и $v_j$, один из которых находится и на их стыке.}

{\bf СЛУЧАЙ 1} $v_i$ содержит повтор $v$ полностью
(т.к. $bc'${--}корни не симметричны, то нужно отдельно рассматривать случай $v\subset v_i$).
Т.е. можно представить $v=v_lv_r$, где $v_l$ --- непустой суффикс в $v_i$, а $v_r$ --- непустой префикс в $v_j$.

Заметим сразу, что $|v_l|$ и $|v_r|$ не превосходят <<запрещённых>> повторов внутри $v_i$ и $v_j$.
Т.к. иначе это приводит к синхронизации по $bc${--}корням,
что приводит либо к сопряжённости $v_i$ и $v_j$ (в случае слишком длинного $|v_r|$),
либо к представлению $v_i$ в виде целой степени (>1) некоторого его фактрора.
И то и другое противоречит условиям несопряжённости
\footnote{В ДР1 ссылка на условие (6.1), видимо имеется ввиду условие в названии главы 6.2,
	а глава 6.1, виимо, была добавлена позже, но ссылка на (6.1) так и осталась неисправленой.}
и, уже проверенному за полином, условию хоорошей экспоненты любого $\V${--}образа.
В этом и есть достаточная идея сведения к полиномиальной проверке.

{\bf для $n\equiv_2 1$}
Заметим, что на стыке $v_i$ и $v_j$ есть пара $d_7d_7$ только при $v_i,v_j\in\V^x$ с общим $x\in\{-,0,+\}$.
При этом на стыке всегда есть либо $d_7$ либо  $d_1d_1$.
Т.е. если в $v_i$ нашёлся длинный $bc'$ повтор (для $v$),
то он должен содержать $d_7$ ($d_1d_1$ не содержится в наших $bc${--}корнях),
но <<внутри>> $bc'${--}кода любого $\V${--}образа встречается парой $d_7d_7$.
А значит наш случай (с целым $q^{-0+}${--}кодом в $v_l$ и $v_r$) возможен только при общем $x\in\{-,0,+\}$.
БОО рассмотрим только $x=0$, остальные отличаются только $d_7$ кодом на краях.

Т.к. $v_i\ne v_j$, то на стыке слева и справа от $d_7d_7$ продолжения могут совпадать целыми $q^{-0+}${--}кодами,
только при $v_i,v_j\in\{v_1,v_2,v_3,v_4,v_{m+3},v_{m+4},v_{m+5},v_{m+6}\}$
т.к. только у них в префиксе или суффиксе есть повторяющиеся $q${--}коды в $bc'${--}корне (причём дважды).
Если слева или справа нет полного совпадения по $q^{-0+}${--}коду, то,
даже при общем суффиксе и префиксе $q_m$ и $q_{m-1}$ длина общего $bc'${--}кода не превышает порядка $4n$,
когда уже для совпадения $q_1$ порядок $6n$.

Т.е. только пары, содержащие $q_1$ и $q'_1$ на стыке имеют наиболее длинный $bc'${--}кода общий с внутренним $w_i^k$.
В такие пары в циклическом слове $w_i^k$ кроме $q_1$ и $q'_1$ входят
$q_m,q'_m$, стоящие первыми в парах и
$q_2,q'_2$, стоящие вторыми в парах.
Но $q_m,q'_m$ и $q_2,q'_2$ входят в $bc'${--}корень по 1-му разу,
а значит однозначно находятся в циклическом $w_i$.
И, если совпадение по целому $q\in\{q_m,q'_m\}$, то есть синхронизация по $w_i$
(т.е. расстояние между совпадениями кратно $|w_i|$),
а значит $v_i$ это степень больше 1 \contr\ правилом построения ЦГС.
Если же $q\in\{q_2,q'_2\}$, то синхронизация будет $w_i$ с $w_j$ (т.е. левый повтор будет до самого конца $v_i$),
а значит $v_i$ будет сопряжено с $v_j$ \contr\ с нашим условием.

Тогда сгруппируем (см. конструкцию $w_i$ для нечётных)

$w_1,w_3,w_{m+3}$ начинаются на $q_1$,
а
$w_2,w_{m+4},w_{m+5}$ начинаются на $q'_1$

$w_2,w_4,w_{m+4}$ заканчиваются на $q_1$,
а
$w_3,w_{m+5},w_{m+6}$ заканчиваются на $q'_1$


Тогда

(1) пары $\{w_2,w_4,w_{m+4}\}\times\{w_2,w_{m+4},w_{m+5}\}$ дают $...q_1.q'_1...$ на стыке $v_i.v_j$.

(2) Пары $\{w_3,w_{m+5},w_{m+6}\}\times\{w_1,w_3,w_{m+3}\}$ дают $...q'_1.q_1...$.
%
%
%
%
%
%
\\
Пусть $n\ge7$, тогда $m\ge2$.
Это не повлияет на полиномиальность проверки, но упростит сведение к ней,
используя то, что теперь $m\ne1$.

Из этих пар нужно исключить:

1) С одинаковыми $bc${--}корнями (т.к. в этом случае $v_i=v_j$)
т.е. $w_2w_2, w_{m+4}w_{m+4}$ в (1) и $w_3w_3$ в (2).

2) Которые приводят к противоречию с условиями несопряжённости $v_i$ и $v_j$,
и правилу построения наших ЦГС (что $v_i$ и $v_j$ не являются целыми степенями больше 1).

2.1) Например,
если на стыке $q'_1.q_1$, то существует единственное соответствие в любом $bc'${--}корне,
а конкретно там где (циклически) перед $q'_1q_1$ стоит $...q'_mq_1$ (т.е. некоторый суффикс из $w_2^k$).
Тогда, любая пара $\{w_3\}\times\{w_1,w_3,w_{m+3}\}$ в (2), при совпадении $q'_1q_1$ в $w_3^k$,
даёт "неограниченное" совпадение $bc'${--}кода влево от стыка $v_3.v_j$ при любом $j\in\{1,3,m+3\}$,
из чего следует нарушение правила построения наших ЦГС.
Т.е. в (2) исключаются все пары $bc${--}корней, где слева $w_3$.

2.2)
Так же, где (циклически) после $q'_1q_1$ стоит $q_2...$ (т.е. некоторый префикс из $w_4^k$).
Тогда, любая пара $\{w_3,w_{m+5},w_{m+6}\}\times\{w_3\}$ в (2), при совпадении $q'_1q_1$ в $w_3^k$,
гарантирует сопряжённость слов $v_i$ и $v_3$ при любом $i\in\{3,m+5,m+6\}$,
из чего следует нарушение условия несопряжённости.
Т.е. в (2) исключаются все пары $bc${--}корней, где справа $w_3$.

Т.е. (2) сокращается до проверки 4-х пар $\{w_{m+5},w_{m+6}\}\times\{w_1,w_{m+3}\}$.

2.3) Если на стыке $q_1.q'_1$,
то существует ровно 2 соответствия в любом $bc'${--}корне.

2.3.1) Где (циклически) перед $q_1q'_1$ стоит $...q_m$ (т.е. некоторый суффикс из $w_{m+3}^k$).
В этом случае, аналогично 2.1)
В (1) исключаются все пары $bc${--}корней, где слева $w_{m+4}$.
В случаях $w_2$ и $w_4$, слева от $q_1.q'_1$ стоит $q'_m$ и $q'_1$ (соответственно),
с которыми максимальный общий $bc'${--}код слева от стыка (для $w_4$) не превосходит $|d_5d_1d_7|+|q_1|=9n+43$

Для правого $v_j$ рассмотрим 2 случая

2.3.1.1) Где (циклически) после $q_1q'_1$ стоит $q'_1...$ (т.е. некоторый префикс из $w_{m+5}^k$).
В этом случае, аналогично 2.2)
В (1) исключаются все пары $bc${--}корней, где справа $w_{m+4}$.
В случаях $w_2$ и $w_{m+5}$, справа от $q_1.q'_1$ стоит $q_1$ и $q'_2$ (соответственно),
с которыми максимальный общий $bc'${--}код справа от стыка (для $w_{m+5}$) не превосходит $|q'_1|+|d_7d_1d_9|=9n+47$
(т.е. максимум $v_4v_{m+5}$)

2.3.1.2) Где (циклически) после $q_1q'_1$ стоит $q_1...$ (т.е. некоторый префикс из $w_3^k$).
В этом случае, аналогично 2.2)
В (1) исключаются все пары $bc${--}корней, где справа $w_2$.
В случаях $w_{m+4}$ и $w_{m+5}$, справа от $q_1.q'_1$ стоит $q'_1$ и $q'_2$ (соответственно),
с которыми максимальный общий $bc'${--}код справа от стыка не превосходит $|q'_1|+|d_7d_1d_5|=9n+43$
(т.е. максимум $v_4v_{m+4}$ и $v_4v_{m+5}$)

2.3.2) Где (циклически) перед $q_1q'_1$ стоит $...q'_m$ (т.е. некоторый суффикс из $w_1^k$).
В этом случае, аналогично 2.1)
В (1) исключаются все пары $bc${--}корней, где слева $w_2$.
В случаях $w_4$ и $w_{m+4}$, слева от $q_1.q'_1$ стоит $q'_1$ и $q_m$ (соответственно),
с которыми максимальный общий $bc'${--}код слева от стыка (для $w_4$) не превосходит $|d_9d_5d_1d_7|+|q_1|=10n+52$

Для правого $v_j$ рассмотрим 2 случая

2.3.2.1) Где (циклически) после $q_1q'_1$ стоит $q'_1...$ (т.е. некоторый префикс из $w_{m+5}^k$).
В этом случае, аналогично 2.2)
В (1) исключаются все пары $bc${--}корней, где справа $w_{m+4}$.
В случаях $w_2$ и $w_{m+5}$, справа от $q_1.q'_1$ стоит $q_1$ и $q'_2$ (соответственно),
с которыми максимальный общий $bc'${--}код справа от стыка (для $w_{m+5}$) не превосходит $|q'_1|+|d_7d_1d_9|=9n+47$
(т.е. максимум $v_4v_{m+5}$)

2.3.1.2) Где (циклически) после $q_1q'_1$ стоит $q_1...$ (т.е. некоторый префикс из $w_3^k$).
В этом случае, аналогично 2.2)
В (1) исключаются все пары $bc${--}корней, где справа $w_2$.
В случаях $w_{m+4}$ и $w_{m+5}$, справа от $q_1.q'_1$ стоит $q'_1$ и $q'_2$ (соответственно),
с которыми максимальный общий $bc'${--}код справа от стыка не превосходит $|q'_1|+|d_7d_1d_5|=9n+43$
(т.е. максимум $v_4v_{m+4}$ и $v_4v_{m+5}$)

Получаем максимум для (2) при $v_4.v_{m+5}$, где в $v_4$ повтор имеет $bc'${--}код входит в $q'_mq_1.q'_1q'_1$,
а левый в $q'_1q_1.q'_1q'_2$.

Остаётся проверить (1), где мы свели проверку для 4-х пар $\{w_{m+5},w_{m+6}\}\times\{w_1,w_{m+3}\}$.

ТУДУ

Остаётся рассмотреть случаи, когда слева и срава от стыка $q${--}коды различны у повторов.

Заметьте, что случай 2-х $q${--}совпадений имеет конечное число случаев,
а значит для них можно искать плохие факторы отдельно,
а остальным достаточно проверить для более коротких факторов,
что несколько снижает старший коэффицент в полиноме проверок.

В любом случае, либо $|v_l|$ не превосходит максимально длинного общего суффикса различных $q${--}кодов,
либо $v_r$ не превосходит максимально длинного общего перфикса различных $q${--}кодов.
Это можно доказать в общем случае для нашего <<вида>> конструкций ОП.

Для $w_4w_{m+5}$ получим $|v_r|\le|q'_1|$

\hyperlink{contents}{$\upuparrows$}

\subsubsection{Случай сопряжённых $u\not=v\in\V$ (раздел \ref{conjIsPoly_Algo})}
Проверка произвольных $k${-}$bc${--}корневых пар $u\not=v$ на сопряжённость сводится к полиномиальной в разделе \ref{conjIsPoly_Algo}.

\hyperlink{contents}{$\upuparrows$}

\newpage

\section {Эффективная проверка условий}

В данной части, как и в разборе ДР2 \ref{SW2}, автор, почти без изменений копирует некоторые части текста из ДР2,
основанных на ДР1.
Здесь, так же будет некоторый рефакторинг текста
(т.е. изменения без потери смысла, может с некоторой редакцией того, что подразумевается).
В данном разделе представлены некоторые оптимизационны алгоритмы проверки слов на $\D_{3, n}^\eps${--}граничность,
а так же, предложена полиномиальная схема алгоритма проверки на сопряжённость пары слов из нашего набора $\V$,
построенного по схеме нашей Леммы \ref{l3} (т.е. $\D_{3, n}^\eps${--}ЦГС).

Заметим, что $\bfdt_\a(\v)$ можно рассматривать как булеву функцию --- может ли слово $\u\in\V_\a$ следовать за словом $\v$.
При этом отношение этих символьных образов не коммутативно, что так же ослабляет условия $(\L, \eps){-}$согласованности множества.

Для вычисления значений $\Dt(\u, \v)$ достаточно вычислить расстояния между вершинами орграфа, где вершинами являются символьные образы, а существование дуги, например от $\v_\a\in\V_\a$ к $\v_\b\in\V_\b$, определяются значением булевой функции: $(\v_\b\in\bfdt_\b(\v_\a))?$.
Тогда значение $\Dt(\v_\a, \v_\b)$ будет определяться минимальной длиной пути от $\v_\a$ к $\v_\b$.
\ Очевидно, что построение орграфа и поиск расстояний полиномиальны, но с увеличением количества условий усложняется доказательство, как и программа проверки этих условий.
\begin{note} \label{epsE-tr-enought}
	Если $\w$ --- $(\eps', \E)${--}ГС, то $\w$ --- $(\eps, \E)${--}ГС для любого $\eps\le\eps'$.
\end{note}
\parindent 0 mm

Т.е., в нашем случае, для проверки на $(\eps, \E)${--}ГС достаточно проверить слово на $(\frac{1}{\n-1}, \E)${--}граничность.
\parindent 10 mm

\hyperlink{contents}{$\upuparrows$}

\paragraph{Немного философии и анализа компьютерных доказательств}

Можно ли считать компьютерное доказательство полноценным?
Для начала поймём, что человеческая проверка так же имеет вероятность ошибки.
Пусть при независимой проверке различными $m<\infty$ группами,
вероятность не заметить ошибку группой не меньше некоторого $p>0$.
Тогда вероятность ошибочности решения не меньше $p^m>0$.

Тогда сравним следующий подход компьютерного решения:

1) Напишем программу проверки компьютерной части и докажем её корректность.
Тогда к доказательству так же относится и это доказательство корректности.

2) Все группы проверки проверяют полное доказательство (т.е. вместе с копьютерной частью).

Тогда вероятность ошибки в полном сведении к комп-й проверке так же не превосходит $p^m>0$.
Только необходимо учитывать вероятность ошибки/сбоя (без обнаружения проверяющими) всех компьютеров во время проверки,
но эта вероятность на порядки меньше и может быть уменьшена на порядки с увеличением числа повторных запусков проверки.

\paragraph{\large Неформальное дополнение.}
Все алгоритмы, их анализ и др.утверждения, представленные в этом разделе, были написаны автором в 2013 г. в ДР2,
но не пререпроверялись (автором планируется редакция в следующих версиях данной работы).
Поэтому могут быть ошибки или недоговорки, но идейная основа (метод решения разбираемых задач) верная
(автором перепроверялись достаточно детально и внимательно).

\hyperlink{contents}{$\upuparrows$}

\subsection {Алгоритм проверки на $(\eps, \E)${--}граничность слова $\w$ за $O(|w|\log{|w|})$}

%

\subsubsection{Алгоритм}

Здесь мы опишем построенный нами эффективный алгоритм проверки наличия <<запрещённых>> подслов в слове $\w$.

Пусть $\f(\n)>1$ и наборы $\R_0, ..., \R_\r$ и $P_0, ..., P_{\r-1}$ удовлетворяют условиям
\parskip -15pt

\begin{gather} \label{alg1:cond1}
\R_0<\R_1< ...<\R_\r=|\w|, 1\le P_0<P_1< ...<P_\r,\\\label{alg1:cond2}
\R_\i\ge\min\bigg\{\r\in\sN\colon\frac{\r+P_\i+\eps}{\r}\le\f(\n)\bigg\}=\bigg\lceil\frac{P_\i+\eps}{\f(\n)-1}\bigg\rceil
\end{gather}
\parskip -2pt

Напишем сначала вспомогательные процедуры построения массива позиций предыдущей буквы в слове $\w$.
Пусть $\s_\i=|\w|\pmod{P_\i}+1$, $\m_\i=\lceil(1+\R_\i-\s_\i)/P_\i\rceil$, $pt_1(\m)=\s_\i+\m\*P_\i$ и $pt_2(\m)=pt_1(\m)-\R_\i$.
\parindent 0 mm
\parskip 10pt

\tracingtabularx\begin{tabularx}{\linewidth}{>{\setlength{\hsize}{.90\hsize}}X>{\setlength{\hsize}{1.10\hsize}}X}
	
	\textbf{Procedure} Prewious(\w):
	\parindent 5 mm
	\parskip 5pt
	
	1\ for($\j=1$; $\j\le|\w|$; $\j\ {+}{+}$)\par
	\parskip 0pt
	2\ \ \ \ \ if($LastPositionOf[\w[\j]]$)\par
	3\ \ \ \ \ \ \ \ \ $prew[\j]=\j-LastPositionOf[\w[\j]]$;\par
	4\ \ \ \ \ else\par
	5\ \ \ \ \ \ \ \ \ $prew[\j]=0$;\par
	6\ \ \ \ \ $LastPositionOf[\w[\j]]=\j$;\par
	\parindent 0 mm
	&
	\textbf{Procedure} Start\_Points():
	\parindent 0 mm
	\parskip 5pt
	\parindent 5 mm
	
	1\ for($\i=0$; $\i<\r$; $\i\ {+}{+}$)\par
	\parskip 0pt
	2\ \ \ \ \ $\w'=\w$;\par
	3\ \ \ \ \ for($\m=\m_\i$; $pt_1(\m)\le|\w'|$; $\m\ {+}{+}$)\par
	4\ \ \ \ \ \ \ \ \ $\w'[pt_2(\m)]=\w'[pt_1(\m)]$;\par
	5\ \ \ \ \ Prewious($\w'$);\par
	6\ \ \ \ \ for($\m=\m_\i$; $pt_1(\m)\le|\w'|$; $\m\ {+}{+}$)\par
	7\ \ \ \ \ \ \ \ \ $points_\i[\m]=pt_2(\m)-prew[pt_2(\m)]$;\par
	8\ \ \ \ \ \ \ \ \ if($prew[pt_2(\m)]=0\ \&\&\ \w[pt_2(\m)]\ne\w[pt_1(\m)]$)\par
	9\ \ \ \ \ \ \ \ \ \ \ \ \ $points_\i[\m]=0$;\par
\end{tabularx}
\parskip 0pt

\parindent 0 mm
Процедура Start\_Points() вычисляет ближайшую позицию буквы $\w[\s_\i+\m\*P_\i]$ слева от позиции $\s_\i+\m\*P_\i-\R_\i$ или равную ей.
Теперь опишем процедуру поиска максимального повтора среди всех факторов с определённой длиной корня $\r=pt_1-pt_2$, для каждого $\r$ в диапазоне от $\R_\i$ до $\R_{\i+1}-1$.
\parskip 5pt


\begin{figure}[h]
	\textbf{Procedure} Max\_Repeat($\R_\i$, $\R_{\i+1}$):
	\parindent 10 mm
	
	1\ for($pt_1=\s_\i+\m_\i\*P_\i$; $pt_1\ne|\w|+1$; $pt_1\ {+}{=}\ P_\i$)//$pt$ от point\par
	\parskip 0pt
	2\ \ \ \ \ for($pt_2=points_\i[(pt_1-\s_\i)/P_\i]$; $pt_2>pt_1-\R_{\i+1}$ \&\& $pt_2\ge1$; $pt_2\ {-}{=}\ prew[pt_2]$)\par
	3\ \ \ \ \ \ \ \ \ $Lrep=Rrep=1$;\par
	4\ \ \ \ \ \ \ \ \ while$\big(pt_2-Lrep\ge1$\ \ \ \ \&\& $\w[pt_2-Lrep]=\w[pt_1-Lrep]\big)$\par
	5\ \ \ \ \ \ \ \ \ \ \ \ \ $Lrep\ {+}{+}$;\par
	6\ \ \ \ \ \ \ \ \ while$\big(pt_1+Rrep\le|\w|$ \&\& $\w[pt_2+Rrep]=\w[pt_1+Rrep]\big)$\par
	7\ \ \ \ \ \ \ \ \ \ \ \ \ $Rrep\ {+}{+}$;\par
	8\ \ \ \ \ \ \ \ \ $max\_rep[pt_1-pt_2]=\max\big\{max\_rep[pt_1-pt_2], Lrep+Rrep-1\big\};$
\end{figure}
\begin{prop}
	Если в $\w$ существует <<запрещённый>> фактор с периодом $\R_\i\le\r<\R_{\i+1}$, то процедура Max\_Repeat($\R_\i$, $\R_{\i+1}$) найдёт максимальный повтор среди всех подслов с такой длиной корня.
\end{prop}
\begin{proof}
	Пусть $\u\v\u$ такой фактор (см. рис. ниже) что $|\u\v|=\r=pt_1-pt_2$, тогда $\frac{\r+|\u|+\eps}{\r}>\f(\n)$.
	%
	\begin{figure}[h]
		\centerline{
			\unitlength=1mm
			\begin{picture}(115,24)(7,-10)
			\put(-12,5.3){\makebox(0,0)[cb]{$\w=$}}
			\put(-6,5){\line(0,1){2.5}}
			\put(-6,5){\line(1,0){2}}
			\put(7,5){\line(1,0){5}}
			\put(25,5){\line(1,0){100}}
			\multiput(30.4,5)(22,0){2}{\line(0,1){2.5}}
			\multiput(98.4,5)(22,0){2}{\line(0,1){2.5}}
			\multiput(41,2)(68,0){2}{\makebox(0,0)[cb]{$\u$}}
			\put(75,2){\makebox(0,0)[cb]{$\v$}}
			\multiput(19,5)(113,0){2}{\makebox(0,0)[cc]{$\cdots$}}
			\put(2,5){\makebox(0,0)[cc]{$\cdots$}}
			\put(140,5){\line(0,1){2.5}}
			\put(140,5){\line(-1,0){2}}
			\multiput(30.5,6.3)(68,0){2}{\fcolorbox[rgb]{0.8, 0.8, 0.8}{0.8, 0.8, 0.8}{\fontsize{5}{10}$\ \ \ \ \ \ \ \ \ \ \ \ \ \ \ \!\!\!$}}
			%
			\put(80,9){\line(1,0){45}}
			\put(82,9){\line(0,-1){1.5}}
			\put(102,9){\line(0,-1){1.5}}
			\put(122,9){\line(0,-1){1.5}}
			\put(132,9){\makebox(0,0)[cc]{$\cdots$}}
			\multiput(22,9)(20,0){3}{\makebox(0,0)[cc]{$\cdots$}}
			\put(82,10){\makebox(0,0)[cb]{$\m{-}P_\i$}}
			\put(102,10){\makebox(0,0)[cb]{$pt_1{=}\m$}}
			\put(122,10){\makebox(0,0)[cb]{$\m{+}P_\i$}}
			\put(102,7.5){\circle*{1.5}}
			\put(140.5,9){\line(0,-1){1.5}}
			\put(140.5,9){\line(-1,0){1.5}}
			\put(140.5,10){\makebox(0,0)[cb]{$|\w|{+}1$}}
			\put(0,9){\line(0,-1){1.5}}
			\put(0,9){\line(1,0){4}}
			\put(0,10){\makebox(0,0)[cb]{$\s_\i{+}\m_\i\*P_\i$}}
			\put(25,0){\line(1,0){15}}
			\put(49,0){\makebox(0,0)[cc]{$\cdots$}}
			\put(57,0){\line(1,0){15}}
			\put(34,0){\line(0,1){4}}
			\multiput(24.4,5.6)(1,0){3}{\fontsize{5}{10}$\diamond$}
			\multiput(28.4,5.6)(1,0){5}{\fontsize{5}{10}$\diamond$}
			\multiput(34.4,5.6)(1,0){6}{\fontsize{5}{10}$\diamond$}
			\put(7,-1){\line(0,1){2}}
			\put(72,-1){\line(0,1){2}}
			\multiput(67.4,5.6)(1,0){5}{\fontsize{5}{10}$\diamond$}
			\multiput(56.4,5.6)(1,0){3}{\fontsize{5}{10}$\diamond$}
			\multiput(60.4,5.6)(1,0){6}{\fontsize{5}{10}$\diamond$}
			\multiput(7.4,5.6)(1,0){2}{\fontsize{5}{10}$\diamond$}
			\multiput(10.4,5.6)(1,0){2}{\fontsize{5}{10}$\diamond$}
			\put(10,0){\line(0,1){4}}
			\put(28,0){\line(0,1){4}}
			\put(28,5.6){\makebox(0,0)[cb]{\fontsize{6}{10}$\a$}}
			\put(67,5.6){\makebox(0,0)[cb]{\fontsize{6}{10}$\a$}}
			\put(10,5.6){\makebox(0,0)[cb]{\fontsize{6}{10}$\a$}}
			\multiput(34,5.6)(68,0){2}{\makebox(0,0)[cb]{\fontsize{6}{10}$\a$}}
			\put(67,0){\line(0,1){4}}
			\put(60,0){\line(0,1){4}}
			\put(60,5.6){\makebox(0,0)[cb]{\fontsize{6}{10}$\a$}}
			\put(7,0){\line(1,0){5}}
			\put(19,0){\makebox(0,0)[cc]{$\cdots$}}
			\put(34,3.5){\circle*{1.5}}
			\put(34,-3){\makebox(0,0)[cb]{\fontsize{5.5}{10}$pt_2$}}
			\put(8,-3){\makebox(0,0)[cb]{\fontsize{5.5}{10}$pt_1{-}\R_{\i+1}$}}
			\put(73,-3){\makebox(0,0)[cb]{\fontsize{5.5}{10}$pt_1{-}\R_\i$}}
			\qbezier[18](67,0)(63.5,2)(60,0)
			\qbezier[21](28,0)(31,2)(34,0)
			\put(64,1.5){\makebox(0,0)[cb]{\fontsize{6}{10}$J_1$}}
			\put(31,1.5){\makebox(0,0)[cb]{\fontsize{6}{10}$J_2$}}
			\qbezier[9](28,0)(26.5,1)(25,1)
			\qbezier[9](60,0)(58.5,1)(57,1)
			\qbezier[18](34,0)(37,1.5)(40,1)
			\qbezier[6](10,0)(11,0.5)(12,0.8)
			\qbezier[136](34,-5)(68,-15)(102,-5)
			\qbezier[10](34,5)(34,-0.5)(34,-6)
			\qbezier[10](102,5)(102,-0.5)(102,-6)
			\put(68,-9.5){\makebox(0,0)[cb]{\fontsize{8}{10}$|\u\v|$}}
			\qbezier[72](67,-4)(84.5,-9)(102,-4)
			\qbezier[10](67,5)(67,0.5)(67,-5)
			\put(84,-6){\makebox(0,0)[cb]{\fontsize{8}{10}$jump$}}
			\end{picture}
		}
		\caption{Здесь '$\diamond$' - переменный символ из $\A_\n\setminus\{\a\}$, $jump=pt_1{-}points_\i[(pt_1{-}s_\i)/P_\i]$,
			$\J_1=prew[pt_1{-}jump]$, $\J_2=prew[pt_2]$.}
	\end{figure}

	Получим $|\u|>\r\*(\f(\n)-1)-\eps\ge\R_\i\*(\f(\n)-1)-\eps\ge P_\i$.
	Значит, переменная $pt_1$ при поиске повтора, нарушающего ограничение для экспоненты, пробегая по всем позициям $pt_1\in\sN$, для которых $pt_1\pmod{P_\i}=\s$ (при любом, заранее выбранном $\s$, от 0 до $P_\i-1$), укажет на некоторую букву $\a$ правого $\u$, в позиции $\m$ слова $\w$.
	Остаётся убедиться, что разница $pt_1-pt_2$ не пропустит значение $|\u\v|$.
	Т.к. $\w[\m-|\u\v|]=\w[\m]=\a$, то достаточно, чтобы $pt_2$ не пропустил ни одной буквы $\a$ в позициях от $pt_1-\R_{\i+1}+1$ до $pt_1-\R_\i$, что обеспечивают массивы $prew$ и $points_\i$.
\end{proof}

\parskip 0pt
Теперь мы можем написать алгоритм проверки существования в слове такого фактора с любым корнем любой длины (в частности с любым периодом).
\parskip 0pt

1\ Выбрать множества $\{\R_0, ..., \R_\r\}$ и $\{P_0, ..., P_{\r-1}\}$, с учётом условий $\eqref{alg1:cond1}$ и $\eqref{alg1:cond2}$;\par
2\ Start\_Points();\par
3\ Prew($\w$);\par
4\ for($\i=1$; $\i<|\w|$; $\i++$)\par
5\ \ \ \ \ $max\_rep[\i]=0$;\par
6\ for($\i=0$; $\i<\r$; $\i++$)\par
7\ \ \ \ \ Max\_Repeat($\R_\i$, $\R_{\i+1}-1$);\par
8\ for($\i=1$; $\i<|\w|$; $\i++$)\par
9\ \ \ \ \ if($max\_rep[\i]+\i+\eps\*\chi_{max\_rep[\i]}>\f(\n)\*\i$)\par
10\,\ \ \ \ \ \ \ print <<false>> \&\& exit;\par
11print <<true>>;
\parskip 5pt

Заметим, что, если в слове нет <<запрещённых>> факторов с некоторым периодом, то предложенный алгоритм может не найти максимальный повтор среди факторов с таким периодом.
Т.е. данный алгоритм не подходит для вычисления $\lexp(\w)$.
\parskip 0pt

\hyperlink{contents}{$\upuparrows$}

\subsubsection {Оценка временной сложности алгоритма}

Добавим к $\b\ce{-}$коду правило:
\begin{equation} \label{bc:rule}
\P(\w)[\i]\ne
\begin{cases}
"\-"\,\!\!, & \text{если } \P(\w)[\i-\n-2]=\P(\w)[\i-\n]=\P(\w)[\i-2]="\-"\,\!\!, \i>\n\\
"\0"\,\!\!, & \text{если } \P(\w)[\i-\n]="\0"\,\!\!, \i>\n.
\end{cases}
\end{equation}
Не трудно проверить, что код, нарушающий правило~\eqref{bc:rule}, порождает не граничное слово.
\begin{note} \label{bc:rem1}
	Если слово $\w$ получено из $\b\ce{-}$кода с правилом~\eqref{bc:rule}, то все подслова в слове $\mu(1_\n, \w)$ с периодом не более $3\*(\n-1)$ являются $(\eps, \{1\}){-}$г при $\eps\le\frac{1}{\n-1}$.
\end{note}

Пусть $\w\in\{"\0"\,\,\!, "\-"\,\!\!, "\+"\}^*$.
$\w$ назовём целым, если он выражается при помощи $\b\ce{-}$кода.
Т.к. целое слово не заканчивается символом $"\-"$, то
\begin{note} \label{bc:rem2}
	Суффикс длины $\n$ любого слова длины не менее $\n$, код которого - целое слово, состоит из различных букв.
\end{note}

Для начала посчитаем число итераций в процедуре Max\_Repeat.
Цикл в 1-й строке работает $O(|\w|/P_\i)$.
Внутренний цикл в строке 2 в среднем работает $O((\R_{\i+1}-\R_\i)/\n)$.
Пусть $\S_{\w, \r}=\{(\u, \i)\colon\u\v\u=\w[\i, ..., \i+|\u\v\u|-1], |\u\v|=\r\}$, тогда циклы в 4, 5 и 6, 7 строках можно взять равным среднему значению длин повторов $rp=|\u|$ среди всех факторов $\u\v\u$ слова $\w$ с длиной некоторого его корня $|\u\v|=rt$, т.е.
$$
\Eq(|\u|\ |\ (\u, \i)\in\S_{\w, rt})=\frac{1}{|\S_{\w, rt}|}\underset{\fontsize{5}{10}{(\u, \i)}\in\S_{\w, rt}}{\sum}|\u|
$$
и обозначим как $\Eq_{\w, rt}$.
Теперь подсчитаем среднее число итераций, взяв $\Eq_\w=\underset{\R_0\le rt<\R_\r}{\max}\Eq_{\w, rt}$
$$
O\Bigg(\frac{|\w|}{P_\i}\*\Bigg(\frac{\R_{\i+1}{-}\R_\i}{\n}\*+\sum_{\i=1}^{(\R_{\i+1}{-}\R_\i)/\n}\frac{1}{\n}\sum_{\j=\R_\i{+}\i\*\J}^{\R_\i{+}(\i{+}1)\*\n{-}1}\Eq_{\w, \j}\Bigg)\Bigg)\le O\bigg(\frac{|\w|\*(\R_{\i+1}{-}\R_\i)\*(\Eq_\w+1)}{P_\i\*\n}\bigg)
$$
Оценим теперь временную сложность этого алгоритма.
Строка 1 занимает $O(\r)$.
Процедура в 1-й строке занимает $O(\r|\w|)$, а в 3-й $O(|\w|)$.
Циклы в строках 4, 5 и 8, 9, 10 выполняются за $O(|\w|)$.
Тогда, оценим число итераций цикла в 6, 7 строках, по всем вариантам множества $\{\R_1, ..., \R_{\r-1}\}$, с учётом, что $\R_\i\thicksim\frac{P_\i+\eps}{\f(\n)-1}$ получим
$$
O\Bigg(\frac{|\w|\*(\R_1-\R_0)\*(\Eq_\w+1)}{P_0\*\n}+...+\frac{|\w|\*(\R_\r-\R_{\r-1})\*(\Eq_\w+1)}{P_{\r-1}\*\n}\Bigg)
=O\Bigg(\frac{|\w|\*(\Eq_\w+1)}{\n\*(\f(\n){-}1)}\bigg(\frac{P_1}{P_0}+...+\frac{P_\r}{P_{\r-1}}\bigg)\Bigg)\\
$$
Где $P_\r$ - константа, определённая правилом $\eqref{alg1:cond2}$ по фиксированому $\R_\r$ в $\eqref{alg1:cond1}$.
Пусть $P_\i\ge1$ для $\i=1, ..., \r-1$ вещественны и $\r\in\sN$, тогда вычислим при каких аргументах и их количестве функция $\g(P_1, ..., P_{\r-1})=P_1/P_0+...+P_\r/P_{\r-1}$ достигает минимума.
Сначала фиксируем все эти аргументы кроме $P_\i$, тогда существует частная производная по ней $\frac{\partial\g}{\partial P_\i}(P_1, ..., P_{\r-1})=1/P_{\i-1}-P_{\i+1}/P_\i^2=0$.
Откуда экстремум $\g$ достигается при $P_1/P_0=...=P_\r/P_{\r-1}$ и $P_\r=P_1^\r/P_0^{\r-1}$.

Значит экстремум $\g$ при фиксированом $\r$ равен $\r\*P_1/P_0=\r\*\sqrt[\leftroot{2}\uproot{2}\r]{\m}=\h(\r)$, где $\m=P_\r/P_0$.
Теперь найдём $\r$, при котором $\h(\r)$ минимальна.
$\h'(\r)=\sqrt[\leftroot{2}\uproot{2}\r]{\m}+\r\big({-}\sqrt[\leftroot{2}\uproot{1}\r]{\m}\*\ln(\m)/\r^2\big)=0$, откуда $\r=\ln(\m)$ и $P_1=P_0\*e$.
Т.е. экстремум достигается при верхнем или нижнем округлении полученного $\r$.
Для определённости возмём $\r=\big\lfloor\ln(\frac{P_\r}{P_0})\big\rfloor$.
При этом, экстремум единственен.
Тогда, если он не минимизирует значение $\g$, то при вариации некоторого из аргументов число итераций будет не больше, что неверно.
Значит, экстремум и есть минимум.


Перейдём к целочисленным значениям.
Пусть $P_\i=\lceil P_0\*e^\i\rceil$, для $\i=1, ..., \r-1$, вычислив $\R_\i$ по $\eqref{alg1:cond2}$ получим сложность $O\Big(\frac{|\w|\*(\Eq_\w+1)\*\ln(|\w|)}{\n\*(\f(\n)-1)}\Big)$.
Например, для граничных слов $\f(\n)=\frac{\n}{\n-1}$.
Значит, для проверки слова $\w$ на $(\eps, \E){-}$г получаем число итераций порядка не более чем $O\big(|\w|\*\ln(|\w|)\*(\Eq_\w+1)\big)$.


Для граничных слов и $\eps\le\frac{1}{\n-1}$ получим значения $P_0$ и $P_\r$. По $\eqref{alg1:cond2}$ $\R_\i\ge\min\{\r\colon\r\ge(P_\i+\eps)\*(\n-1)\}=P_\i\*(\n-1)+\lceil\eps\*(\n-1)\rceil$, т.е. $P_\i\le\frac{\R_\i-1}{\n-1}$ или $P_\i\le\big\lfloor\frac{\R_\i-1}{\n-1}\big\rfloor$.
По \eqref{alg1:cond1} $P_0\ge1$ и значения $P_\i$ для сокращения числа итераций в цикле 1-й строки, нужно выбирать наибольшие из возможных.
Тогда $\R_0\ge\n$ и $P_\r=\big\lfloor\frac{|\w|-1}{\n-1}\big\rfloor$.
Т.е. алгоритм позволяет проверить слово на существование <<запрещённых>> подслов с периодом не менее $\n$.
Если слово удовлетворяет условию $\b\ce{-}$кода, то с учётом Замечания~\ref{bc:rem1} и $1\in\E$, запрещённых подслов с периодом менее $\n$ нет.
Более того, учитывая $\eqref{bc:rule}$ по Замечанию~$\ref{bc:rem1}$, оно не имеет запрещённых подслов с периодом менее $3\n-2$ при $\n>5$, $\eps\le\frac{1}{\n-1}$ и $1\in\E$.
При таких условиях достаточно проверять подслова с периодом не менее $\R_0\ge3\n-2$, т.е. $P_0=3$.

Таким образом, для проверки слова на $(\eps, \E){-}$г, полученого из $\b\ce{-}$кода, при $\eps\le\frac{1}{\n-1}$ можно взять
$$
\r=\Big\lfloor\ln\tfrac{|\w|-1}{3(\n-1)}\Big\rfloor, P_\i=\lceil3\*e^\i\rceil, \R_\i=P_\i\*(\n-1)+1, \i=0, ..., \r-1 \text{ и } \R_\r=|\w|.
$$

\hyperlink{contents}{$\upuparrows$}

\subsection {Полиномиальная проверка $bc${--}корневых слов на сопряжённость}\label{conjIsPoly_Algo}
Т.к. для любого $\i=1, ..., |\sV|$ слова $\w_1=\v_\i\u_\i$ и $\w_\i=\u_\i\v_\i$ сопряжены при $|\u_\i|>0$, то возмём перестановку $\pi_\i\equiv_\pi\v_\i$ (с учётом условия (uf1.2) получим $\pi_1=1_\n$).
Пусть $\pi\equiv_\pi\w_1$ (понятно, что $\pi^\k=1_\n$).
Тогда,\par
$(*)$ $\mu((\pi^\j\*\pi_\i)^{-1}, \w_1^\k)$ и $\mu(1_\n, \w_\i^\k)$ сопряжены, для любых $\i$ и $\j$.\\
$\mu(\pi_1\*\pi_2, \u)=\mu(\pi_2, \mu(\pi_1, \u))$.

\begin{prop}
	$\mu(1_\n, \w_\p^\k)$ сопряжено с $\mu(1_\n, \w_\q^\k)$, если и только, если  существует $\m\in\sZ$, такое, что $\pi_\p=\pi^\m\*\pi_\q$.
\end{prop}
\begin{proof}
	Пусть $\mu(1_\n, \w_\p^\k)=\u\v$ и $\mu(1_\n, \w_\q^\k)=\v\u$ сопряжены для некоторых $\p\ne\q$, б.о.о считаем, что $|\v_\p|<|\v_\q|$ (т.е. $\p<\q$).
	Тогда сществуют $\x\y\z=\w_1$, что $\y\z\x=\w_\p$ и $\z\x\y=\w_\q$ (т.е. $\pi_\p\equiv_\pi\x=\v_\p$, $\y\z=\u_\p$, $\pi_\q\equiv_\pi\x\y=\v_\q$ и $\z=\u_\q$).
	Видно, что $\w_\p^\om=\y\z\w_1^\om$ и $\w_\q^\om=\z\w_1^\om$ значит $\mu(1_\n, \y\z\w_1^\om)=\u(\v\u)^\om$ и $\mu(1_\n, \z\w_1^\omega)=(\v\u)^\om$.
	Тогда $\z\w_1^\om=\y'\z\w_1^\om$ для некоторого $\y'\in\P(\z\x\y)$.
	Откуда $\w_1=\w^{\l/\g}$ для некоторого $\w$, где $\l=|\w_1|$ и $\g=\gcd(\l-|\y'|, \l)$.
	Т.к. $\w_1$ не является степенью, то $\g=\l$.
	Значит $|\u|=\y\pmod{\l}$.
	Тогда для любых $\t\in\sZ$ получим $\mu(\pi_\p^{-1}\*\pi_\q\*\pi^\t\*\pi^{\m-\t}, \z\w_1^\om)=(\v\u)^\om$ и $\mu(1_\n, \z\w_1^\om)=(\v\u)^\om$, где $\m=\lfloor|\u|/\l\rfloor$.
	Значит $\pi_\p^{-1}\*\pi^\t\*\pi_\q\*\pi^{\m-\t}=1_\n$, откуда $\pi^\t\*\pi_\q\*\pi^{\m-\t}=\pi_\p$ в частности $\pi^\m\*\pi_\q=\pi_\p$.
	
	Обратно, пусть $\m$ такое, что $\pi_\p=\pi^\m\*\pi_\q$.
	Тогда $\mu((\pi^0\pi_\p)^{-1}, \w_1^\k)=\mu((\pi^\m\pi_\q)^{-1}, \w_1^\k)$.
	При этом $\mu((\pi^{\t_1}\pi_\i)^{-1}, \w_1^\k)$ и
	\\
	$\mu((\pi^{\t_2}\pi_\i)^{-1}, \w_1^\k)$ сопряжены при любых $\t_1, \t_2\in\sZ$ и $\i=1, ..., |\sV|$.
	Т.к. все соряжённые слова образуют класс эквивалентности,  то $\mu((\pi^\t\pi_\p)^{-1}, \w_1^\k)$ сопряжено с $\mu((\pi^{\m+\s}\pi_\q)^{-1}, \w_1^\k)$ при любых $\s$ и $\t$ из $\sZ$.
	По $(*)$ получим, что $\mu((\pi^\t\pi_\p)^{-1}, \w_1^\k)$ сопряжено с $\mu(1_\n, \w_\p^\k)$, а $\mu((\pi^{\m+\s}\pi_\q)^{-1}, \w_1^\k)$ с $\mu(1_\n, \w_\q^\k)$.
	Опять же, т.к. сопряжённые слова образуют класс эквивалентности, то $\mu(1_\n, \w_\p^\k)$ сопряжено с $\mu(1_\n, \w_\q^\k)$.
	%
\end{proof}

Теперь введём множества $\C_\i=\{\pi^\m\*\pi_\i\colon\m=1, ..., \k\}$.
Т.е. $\C_\i$ получено умножением слева на $\pi_\i$ каждого элемента группы $\{\pi^\t\colon\t\in\sZ\}$ относительно умножения.
Значит $\pi^\t\pi_\i\in\C_\i$ для любых $\t\in\sZ$.
\parskip 0pt
\begin{note}
	$\C_\p\cap\C_\q\ne\emptyset$, если и только, если $\C_\p=\C_\q$.
\end{note}
\begin{proof}
	Т.к. $\C_\p$ и $\C_\q$ не пусты, то достаточность очевидна.
	Пусть $\C_\p\cap\C_\q\ne\emptyset$, тогда существуют $\m_1$ и $\m_2$, что $\pi^{\m_1}\*\pi_\p=\pi^{\m_2}\*\pi_\p$. По построению множеств получим $\C_\q\ni\pi^{\m_2+\t}\*\pi_\q=\pi^{\m_1+\t}\*\pi_\p\in\C_\p$, для любых $\t\in\sZ$, откуда из конечности и цикличности $\C_\q=\C_\p$.
\end{proof}

Таким образом, для проверки на сопряжённость $bc${--}корневых слов достаточно сравнить по одной перестановке из множеств $\C_\p$ и $\C_\q$, которые нужно выбирать однозначно (каноническое слово, инвариант).
Например по наименьшему в лексикографическом порядке (в дальнейшем минимальным) слову $\u$ длины $\n$, для котрого $\pi^\s\pi_\i\equiv_\pi\u$ для некоторого $\s\in\{1, ..., \k\}$.
Тогда разложим перестановку $\pi$ на циклические перестановки $\pi_1, ..., \pi_\s$.
Возмём множество позиций, в которых буквы меняют свою позицию перестановкой $\pi_\i$ и переставим по $\pi_\j$,
тогда обозначим это множество $\I(\pi_\i, \pi_\j)$.
При этом занумерованы они так, что из двух перестановок $\pi_\i$ и $\pi_\j$ меньший номер имеет та, у которой наименьшее число из $\I(\pi_\i, \pi_\j)$ и $\I(\pi_\j, \pi_\j)$ меньше.
Тогда для эффективного поиска такого $\s$, чтобы слово $\u$ было минимальным можно последовательно вычислять (фиксировать) положения букв в позициях $\I(\pi_1, \pi_\j)$, затем фиксировать буквы в позициях $\I(\pi_2, \pi_\j)$, не нарушая положений фиксированных букв и т.д..
Пусть $\I_\i=|\I(\pi_\i, \pi_\j)|$ (т.е. $\pi_\i^{\I_\i\*\s}=1_\n$ при любых $\s\in\sZ$ и $\lcm(\I_1, ..., \I_\s)=\k$), $\u_0=\u$, $\m_\i=\min\{\m\colon\m\in\I(\pi_\i, \pi_\j)\}$ и $\pi_{c_0}=\pi_\j$, тогда для каждого $\i=1, ..., \s$ выполним процедуру.
\parindent 15pt

i.1 Положить $\r_\i=\lcm(1, \I_1, ..., \I_{\i-1})$ и $\h=\gcd(\I_\i, \r_\i)$\par
i.2 Найти позицию $\j_\i$ минимального символа находящихся в позициях, попадающих в $\m_\i$ перестановками $\pi^{\h\,\*\m}\*\pi_{c_{\i-1}}$ в слове $\u$ по всем $\m=1, ..., \I_\i/\h$.\par
i.3 Найти $\m\in\sZ$, при котором позиция $\j_\i$ перестановками $\pi_\i^{\m\,\*\h}\*\pi_{c_{\i-1}}$ переходит в $\m_\i$.\par
i.4 Найти $\t_\i\in\sZ$, при котором $\t_\i\r_\i\pmod{\I_\i}=\m\,\*\h$.\par
i.5 Положить $\pi_{c_\i}=(\pi_{\i+1}^{\t_\i\r_\i\pmod{\I_{\i+1}}}\oplus...\oplus\pi_\s^{\t_\i\r_\i\pmod{\I_\s}})\pi_{c_{\i-1}}$.
\parindent 30pt

i.1 вычисляет минимальные расстояния $\r$ между степенями $\pi$, при которых фиксированные буквы в $\I(\pi_1, \pi_\j)\cup...\cup\I(\pi_{\i-1}, \pi_\j)$ не меняют позиций.
Циклическое расстояние (по модулю $\I_\i$) $\h$ между ближайшими позициями из $\I(\pi_\i, \pi_\j)$, в которые могут попадать буквы из них же, не нарушая позиций фиксированных букв.
i.2 находит позицию буквы, которая должна попасть в $\m_\i$ для вычисления сдвига букв в позициях $\I(\pi_\i, \pi_\j)$ слова $\u$, чтобы они расположились как в минимальном слове.
i.3 находит положение букв в позициях $\I(\pi_\i, \pi_\j)$ как в минимальном слове.
i.4 находит степень $\t_\i\*\r_\i$, при которой буквы в позициях $\I(\pi_\i, \pi_\j)$ слова $\u$, не нарушая позиций фиксированных букв.
i.5 вычисляем новую перестановку $\pi_{c_\i}$ для которой буквы слова $\mu(\pi_{c_\i}, \u)$ в позициях $\I(\pi_1, \pi_\j)\cup...\cup\I(\pi_\i, \pi_\j)$ располагаются как в минимальном слове, <<прокручивая>> буквы в слове $\u$, ровно $\t_1\*\r_1+...+\t_\i\*\r_\i$ раз перестановкой $\pi_\i$  после чего переставив буквы по $\pi_\j$, т.е. $\pi_{c_\i}=\pi^{\t_\i\r_\i}\pi_{c_{\i-1}}=(\pi_1^{\t_\i\r_\i}\oplus...\oplus\pi_\s^{\t_\i\r_\i})\pi_{c_{\i-1}}$.
При этом $\pi_\l^{\t_\i\r_\i}=\pi_\l^{\m\I_\l}\*\pi_\l^{\t_\i\r_\i\pmod{\I_\l}}=\pi_\l^{\t_\i\r_\i\pmod{\I_\l}}$, для некоторого $\m\in\sZ$, т.к. $\t_\i\r_\i$ кратно $\I_1, ..., \I_{\i-1}$, то получим равенство из i.5.
Таким образом, $\mu(\pi_{c_\s}, \u)$ и будет минимальным словом.

Можно, вместо чисел $\I_\i$ и $\r$ хранить их разложение на простые числа с учётом кратности.
Тогда для вычисления $\lcm(\I_1, ..., \I_{\i-1})$ достаточно брать объединение $\I_1, ..., \I_{\i-1}$, где кратность числа суммируется из кратностей этого числа в этих множествах.
Для вычисления $\gcd(\I_\i, \r)$ достаточно брать пересечение множества простых чисел из $\I_\i$ и $\s$, где кратность равна разности кратностей этого числа в этих множествах.
i.2 и i.3 вычисляются за $O(\I_\i/\h)$.
Для вычисления $\t$ достаточно проверить не более $\I_\i/\h$ вариантов.
Таким образом поиск минимального слова, а с ним и проверка на сопряжённость, становится полиномиальной от $n$.

\hyperlink{contents}{$\upuparrows$}

\section{Выводы}

\subsection{Результаты}
\begin{itemize}
	\item
	В данной работе автором описаны полностью независимые (компьютерные) методы доказательства частных случаев
	экспоненциальной версии граничной теоремы (РРДГСТ) над 5-ти буквенными и более алфавитами.
	\item
	Предложен метод построения циклических граничных слов для частных алфавитов с $5$ буквами и более ---
	построены конкретные конструкции для $\D_{3, n}^\eps${--}ЦГС при $n\ge5$ (требующие компьютерной проверки корректности).
	\item
	Предложен метод построения (а сним и доказательство существования)
	граничных слов с почти единичной экспонентой длинных факторов
	на основе равномерно растущего дерева граничных слов (РРДГС) для любых алфавитов с $2$ буквами и более.
	(планируется редакция доказательства в следующих версиях данной работы).
\end{itemize}

\subsection{Планы на будущие версии данной работы}
В следующей версии данного манускрипта автор планирует доработать и обобщить результаты и
перевести на английский (не одновременно).
А так же проверить предложенные конструкции \ref{constructs} на выполнение условия \ref{l2c2} для частных алфавитов.

В неформалных частях (только в русской версии).
Сравнит авторскую редакцию нашей ДР1 с редакцией от нашего НР, опубликованную им в общей с автором статье
(написанной на основе ДР1 и лекций автора на семинаре).
И много дилемм и философии по вопросам авторства, наболевших у автора.

\hyperlink{contents}{$\upuparrows$}

\begin{center}%
	{\bfseries Благодарности}%
	
	Спасибо Шуру А.М. за ценные советы и комментарии по текстам дипломных работ автора
	
	(в частности, за более удачные формулировки некоторых утверждений)
	
	и за ценные замечания и дополнения во время лекций автора в марте 2011.
	
	%
	А так же, за открытие доступа к электронному архиву arxiv.org.
	
\end{center}%


\begin{thebibliography}{99}
	
	
	
	\bibitem {C}%
	Carpi, A.:
	\newblock {\it On Dejean’s conjecture over large alphabets},
	\newblock Theoret. Comput. Sci. 385, 137–151 (2007)
	
	\bibitem {CR1}%
	Currie, J.D., Rampersad, N.:
	\newblock {\it Dejean’s conjecture holds for $n\ge27$},
	\newblock RAIRO Inform. Theor.App. 43, 775–778 (2008)
	
	\bibitem {CR2}%
	Currie, J.D., Rampersad, N.:
	\newblock {\it A proof of Dejean’s conjecture},
	\newblock Math. Comp. 80, 1063–1070 (2009)
	
	\bibitem {Dej}%
	Dejean, F.:
	\newblock {\it Sur un theoreme de Thue},
	\newblock J. Combin. Theory. Ser. A 13, 90–99 (1972)
	
	\bibitem {KR}%
	Kolpakov, R., Rao, M.:
	\newblock {\it On the number of Dejean words over alphabets of 5, 6, 7, 8, 9 and 10 letters},
	\newblock Theoret. Comput. Sci. 412, 6507–6516 (2011)
	
	\bibitem {MC}%
	Mohammad-Noori, M., Currie, J.D.:
	\newblock {\it Dejean’s conjecture and Sturmian words},
	\newblock European J. Combinatorics, 28, 876–890 (2007)
	
	\bibitem {M}%
	Moulin-Ollagnier:
	\newblock {\it Proof of Dejean’s conjecture for alphabets with 5, 6, 7, 8, 9, 10 and 11 letters},
	\newblock Theoret. Comput. Sci. 95, 187–205 (1992)
	
	\bibitem {O}%
	Ochem, P.:
	\newblock {\it A generator of morphisms for infinite words},
	\newblock RAIRO Inform. Theor.App. 40, 427–441 (2006)
	
	\bibitem {P}%
	Pansiot, J.J.:
	\newblock {\it A propos d’une conjecture de F. Dejean sur les repetitions dans les mots},
	\newblock Discrete Appl. Math. 7, 297–311 (1984)
	
	\bibitem {Rao}%
	Rao, M.:
	\newblock {\it Last cases of Dejean’s conjecture},
	\newblock Theoret. Comput. Sci. 412, 3010–3018 (2009)


\paragraph{Ссылки на электронные письма/переписки (с датами) со студенческими работами автора}
\label{links}

Здесь представлен список доступных переписок с нашим научным руководителем на то время (Шуром А.М.).
Письма импортированы из itnvi@mail.ru в igor.n.tunev@gmail.com.
Некоторые письма в itnvi@mail.ru
продублированы из-за сборщика (или импортов) почты/писем с другими электронными почтами автора.
В следующих версиях, возможно, будут выложены ссылки на более полные переписки.
\parskip -12pt


\paragraph{Первое письмо посланное нашему НР со вспомогательным текстом для лекций} (07.03.11, *.docx)
\parskip 0pt

{\color{cyan}share.streak.com/rTMJek8cWTZgbdN7owmTIP}\label{link_lect}
\parskip -12pt

\paragraph{Письмо нашему НР перед защитой ДР1} (финальный вариант ДР1, 24.06.11, *.doc)
\parskip 0pt

{\color{cyan}share.streak.com/BWdavQJtwgLVTVxJV6dOwz}\label{link_bachelor_final}

%
\parskip -12pt


\paragraph{Часть переписки с нашим НР перед защитой ДР2} (финальный вариант ДР2 и презентация, 13.06.13, *.pdf)
\parskip 0pt

{\color{cyan}share.streak.com/rbWBsBcXpeCXnhQVBDQKWe}\label{link_master's_corresp_final}


Файлы *.docx можно смотреть через сервис в mail.ru:
отправляем файл себе на почту в mail.ru и открываем его там.
\parskip -12pt

\paragraph{Папка на Google Drive с доп. ссылками и данными} (на случай, если устареют ссылки на переписки).
\parskip 0pt

{\color{cyan}drive.google.com/drive/folders/19Q7uBj0NBQoE6xrO7V-xcmafXEAFwhs7}

%
%

\end{thebibliography}
\end{document}